\documentclass{amsart}
\usepackage{amssymb}
\usepackage{amsmath}
\usepackage{amsfonts}
\usepackage{geometry}
\usepackage{graphicx}
\usepackage{float}

\newtheorem{theorem}{Theorem}
\theoremstyle{plain}

\newtheorem{corollary}{Corollary}

\newtheorem{definition}{Definition}
\newtheorem{example}{Example}

\newtheorem{lemma}{Lemma}

\newtheorem{proposition}{Proposition}
\newtheorem{remark}{Remark}

\numberwithin{equation}{section}
\input{tcilatex}

\begin{document}
\title[Rates]{Convergence rates for the full Gaussian rough paths}
\author{Peter Friz}
\thanks{P.K. Friz has received funding from the European Research Council
under the European Union's Seventh Framework Programme (FP7/2007-2013) / ERC
grant agreement nr. 258237.}
\address{TU Berlin, Fakult\"{a}t II\\
Institut f\"{u}r Mathematik, MA 7-2\\
Strasse des 17. Juni 136\\
10623 Berlin\\
Germany}
\address{Weierstrass Institut f\"{u}r Angewandte Analysis und Stochastik\\
Mohrenstrasse 39\\
10117 Berlin\\
Germany}
\email{friz@math.tu-berlin.de}
\author{Sebastian Riedel}
\thanks{S. Riedel is supported by an IRTG (Berlin-Zurich) PhD-scholarship
and a scholarship from the Berlin Mathematical School (BMS)}
\address{TU Berlin, Fakult\"{a}t II\\
Institut f\"{u}r Mathematik, MA 7-4\\
Stra\ss e des 17. Juni 136\\
10623 Berlin\\
Germany}
\email{riedel@math.tu-berlin.de}

\begin{abstract}
Under the key assumption of finite $\rho $-variation, $\rho \in \lbrack 1,2)$%
, of the covariance of the underlying Gaussian process, sharp a.s.
convergence rates for approximations of Gaussian rough paths are
established. When applied to Brownian resp. fractional Brownian motion
(fBM), $\rho =1$ resp. $\rho =1/\left( 2H\right) $, we recover and extend
the respective results of [Hu--Nualart; Rough path analysis via fractional
calculus; TAMS 361 (2009) 2689-2718] and [Deya--Neuenkirch--Tindel; A
Milstein-type scheme without L\'{e}vy area terms for SDEs driven by
fractional Brownian motion; AIHP (2011)]. In particular, we establish an
a.s. rate $k^{-\left( 1/\rho -1/2-\varepsilon \right) }$, any $\varepsilon
>0 $, for Wong-Zakai and Milstein-type approximations with mesh-size $1/k$.
When applied to fBM this answers a conjecture in the afore-mentioned
references.
\end{abstract}

\maketitle

% \tableofcontents

\section{Introduction}

Recall that \textit{rough path theory} \cite{L98, LQ02, FV10} is a general
framework that allows to establish existence, uniqueness and stability of
differential equations driven by multi-dimensional continuous signals $%
x\colon \left[ 0,T\right] \rightarrow \mathbb{R}^{d}$ of low regularity.
Formally, a \textit{rough differential equation (RDE)} is of the form%
\begin{equation}
dy_{t}=\sum_{i=1}^{d}V_{i}\left( y_{t}\right) \,dx_{t}^{i}\equiv V\left(
y_{t}\right) \,dx_{t};\quad y_{0}\in \mathbb{R}^{e}  \label{eqn_determ_rde}
\end{equation}%
where $\left( V_{i}\right) _{i=1,\ldots ,d}$ is a family of vector fields in 
$\mathbb{R}^{e}$. When $x$ has finite $p$-variation, $p<2$, such
differential equations can be handled by Young integration theory. Of
course, this point of view does not allow to handle differential equations
driven by Brownian motion, indeed%
\begin{equation*}
\sup_{D\subset \left[ 0,T\right] }\sum_{t_{i}\in D}\left\vert
B_{t_{i+1}}-B_{t_{i}}\right\vert ^{2}=+\infty \text{ a.s.,}
\end{equation*}%
leave alone differential equations driven by stochastic processes with less
sample path regularity than Brownian motion (such as fractional Brownian
motion (fBM) with Hurst parameter $H<1/2$). Lyons' key insight was that low
regularity of $x$, say $p$-variation or $1/p$-H\"{o}lder for some $p\in
\lbrack 1,\infty )$, can be compensated by including "enough" higher order
information of $x$ such as all increments 
\begin{eqnarray}
\mathbf{x}_{s,t}^{n} &\equiv &\int_{s<t_{1}<\dots <t_{n}<t}dx_{t_{1}}\otimes
\ldots \otimes dx_{t_{n}}  \label{IIIntro} \\
&\equiv &\sum_{1\leq i_{1},\ldots ,i_{n}\leq d}\left( \int_{s<t_{1}<\dots
<t_{n}<t}dx_{t_{1}}^{i_{1}}\ldots dx_{t_{n}}^{i_{n}}\right)
\,e_{i_{1}}\otimes \ldots \otimes e_{i_{n}}\in \left( \mathbb{R}^{d}\right)
^{\otimes n}
\end{eqnarray}%
where "enough" means $n\leq \left[ p\right] $ ($\left\{ e_{1},\ldots
,e_{d}\right\} $ denotes just the usual Euclidean basis in $\mathbb{R}^{d}$
here). Subject to some generalized $p$-variation (or $1/p$-H\"{o}lder)
regularity, the ensemble $\left( \mathbf{x}^{1},\dots ,\mathbf{x}^{\left[ p%
\right] }\right) $ then constitutes what is known as a rough path.\footnote{%
A basic theorem of rough path theory asserts that further iterated integrals
up to any level $N\geq \left[ p\right] $, i.e.%
\begin{equation*}
S_{N}\left( \mathbf{x}\right) :=(\mathbf{x}^{n}:n\in \left\{ 1,\dots
,N\right\} )
\end{equation*}%
are then deterministically determined and the map $\mathbf{x}\mapsto
S_{N}\left( \mathbf{x}\right) $, known as Lyons lift, is continuous in rough
path metrics.
\par
{}} In particular, no higher order information is necessary in the Young
case; whereas the regime relevant for Brownian motion requires second order
- or level $2$ - information ("L\'{e}vy's area"), and so on. Note that the
iterated integral on the r.h.s. of (\ref{IIIntro}) is not - in general - a
well-defined Riemann-Stieltjes integral. Instead one typically proceeds by
mollification - given a multi-dimensional sample path $x=X\left( \omega
\right) $, consider piecewise linear approximations or convolution with a
smooth kernel, compute the iterated integrals and then pass, if possible, to
a limit in probability. Following this strategy one can often construct a
"canonical" enhancement of some stochastic process to a (random) rough path.
Stochastic integration and differential equations are then discussed in a
(rough) pathwise fashion; even in the complete absence of a semi-martingale
structure.

It should be emphasized that rough path theory was - from the very beginning
- closely related to higher order Euler schemes. Let $D=\left\{
0=t_{0}<\ldots <t_{\#D-1}=1\right\} $ be a partition of the unit interval.%
\footnote{%
A general time horizon $\left[ 0,T\right] $ is handled by trivial
reparametrization of time.} Considering the solution $y$ of $\left( \ref%
{eqn_determ_rde}\right) $, the step-$N$ Euler approximation $y^{\text{Euler}%
^{N};D}$ is given by%
\begin{eqnarray*}
y_{0}^{\text{Euler}^{N};D} &=&y_{0} \\
y_{t_{j+1}}^{\text{Euler}^{N};D} &=&y_{t_{j}}^{\text{Euler}%
^{N};D}+V_{i}\left( y_{t_{j}}^{\text{Euler}^{N};D}\right) \mathbf{x}%
_{t_{j},t_{j+1}}^{i}+\mathcal{V}_{i_{1}}V_{i_{2}}\left( y_{t_{j}}^{\text{%
Euler}^{N};D}\right) \mathbf{x}_{t_{j},t_{j+1}}^{i_{1},i_{2}} \\
&&+\ldots +\mathcal{V}_{i_{1}}\mathcal{\ldots V}_{i_{N-1}}V_{i_{N}}\left(
y_{t_{j}}^{\text{Euler}^{N};D}\right) \mathbf{x}_{t_{j},t_{j+1}}^{i_{1},%
\ldots ,i_{N}}
\end{eqnarray*}%
at the points $t_{j}\in D$ where we use the Einstein summation convention, $%
\mathcal{V}_{i}$ stands for the differential operator $%
\sum_{k=1}^{e}V_{i}^{k}\partial _{x_{k}}$ and $\mathbf{x}_{s,t}^{i_{1},%
\ldots ,i_{n}}=\int_{s<t_{1}<\dots <t_{n}<t}dx_{t_{1}}^{i_{1}}\ldots
dx_{t_{n}}^{i_{n}}$. An extension of the work of A.M. Davie (cf. \cite{D07}, 
\cite{FV10}) shows that the step-$N$ Euler scheme\footnote{%
... which one would call Milstein scheme when $N=2$ ...} for an RDE driven
by a $1/p$-H\"{o}lder rough path with step size $1/k$ (i.e. $D=D_{k}=\left\{ 
\frac{j}{k}:j=0,\ldots ,k\right\} $) and $N\geq \left[ p\right] $ will
converge with rate $O\left( \frac{1}{k}\right) ^{\left( N+1\right) /p-1}$.
Of course, in a probabilistic context, simulation of the iterated
(stochastic) integrals $\mathbf{x}_{t_{j},t_{j+1}}^{n}$ is not an easy
matter. A natural simplification of the step-$N$ Euler scheme thus amounts
to replace in each step%
\begin{equation*}
\left\{ \mathbf{x}_{t_{j},t_{j+1}}^{n}:n\in \left\{ 1,\dots ,N\right\}
\right\} \text{ }\leftrightarrow \text{ }\left\{ \frac{1}{n!}\left( \mathbf{x%
}_{t_{j},t_{j+1}}^{1}\right) ^{\otimes n}:n\in \left\{ 1,\dots ,N\right\}
\right\}
\end{equation*}%
which leads to the \emph{simplified }step-$N$ Euler scheme%
\begin{eqnarray*}
y_{0}^{\text{sEuler}^{N};D} &=&y_{0} \\
y_{t_{j+1}}^{\text{sEuler}^{N};D} &=&y_{t_{j}}^{\text{sEuler}%
^{N};D}+V_{i}\left( y_{t_{j}}^{\text{sEuler}^{N};D}\right) \mathbf{x}%
_{t_{j},t_{j+1}}^{i}+\frac{1}{2}\mathcal{V}_{i_{1}}V_{i_{2}}\left(
y_{t_{j}}^{\text{sEuler}^{N};D}\right) \mathbf{x}_{t_{j},t_{j+1}}^{i_{1}}%
\mathbf{x}_{t_{j},t_{j+1}}^{i_{2}} \\
&&+\ldots +\frac{1}{N!}\mathcal{V}_{i_{1}}\mathcal{\ldots V}%
_{i_{N-1}}V_{i_{N}}\left( y_{t_{j}}^{\text{sEuler}^{N};D}\right) \mathbf{x}%
_{t_{j},t_{j+1}}^{i_{1}}\ldots \mathbf{x}_{t_{j},t_{j+1}}^{i_{N}}.
\end{eqnarray*}%
Since $\mathbf{x}_{t_{j},t_{j+1}}^{1}=X_{t_{j},t_{j+1}}\left( \omega \right)
=X_{t_{j+1}}\left( \omega \right) -X_{t_{j}}\left( \omega \right) $ this is
precisely the effect in replacing the underlying sample path segment of $X$
by its piecewise linear approximation, i.e.%
\begin{equation*}
\left\{ X_{t}\left( \omega \right) :t\in \left[ t_{j},t_{j+1}\right]
\right\} \text{ }\leftrightarrow \left\{ X_{t_{j}}\left( \omega \right) +%
\frac{t-t_{j}}{t_{j+1}-t_{j}}X_{t_{j},t_{j+1}}\left( \omega \right) :t\in %
\left[ t_{j},t_{j+1}\right] \right\} .
\end{equation*}%
Therefore, as pointed out in \cite{DT} in the level $N=2$ H\"{o}lder rough
path context, it is immediate that a Wong-Zakai type result, i.e. a.s.
convergence of $y^{\left( k\right) }\rightarrow y$ for $k\rightarrow \infty $
where $y^{\left( k\right) }$ solves%
\begin{equation*}
dy_{t}^{\left( k\right) }=V\left( y_{t}^{\left( k\right) }\right)
\,dx_{t}^{\left( k\right) };\quad y_{0}^{\left( k\right) }=y_{0}\in \mathbb{R%
}^{e}
\end{equation*}%
and $x^{\left( k\right) }$ is the piecewise linear approximation of $x$ at
the points $\left( t_{j}\right) _{j=0}^{k}=D_{k}$, i.e.%
\begin{equation*}
x_{t}^{\left( k\right) }=x_{t_{j}}+\frac{t-t_{j}}{t_{j+1}-t_{j}}%
x_{t_{j},t_{j+1}}\text{\quad if }t\in \left[ t_{j},t_{j+1}\right] \text{, }%
t_{j}\in D_{k},
\end{equation*}%
leads to the convergence of the simplified (and implementable!)\ step-$N$
Euler scheme.

While Wong-Zakai type results in rough path metrics are available for large
classes of stochastic processes \cite[Chapter 13, 14, 15, 16]{FV10} our
focus here is on \textit{Gaussian} processes which can be enhanced to rough
paths. This problem was first discussed in \cite{CQ02} where it was shown in
particular that piecewise linear approximation to fBM are convergent in $p$%
-variation rough path metric if and only if $H>1/4$. A practical (and
essentially sharp) structural condition for the covariance, namely finite $%
\rho $-variation based on rectangular increments for some $\rho <2$ of the
underlying Gaussian process was given in \cite{FV10AIHP} and allowed for a
unified and detailed analysis of the resulting class of Gaussian rough
paths. This framework has since proven useful in a variety of different
applications ranging from non-Markovian H\"{o}rmander theory \cite{CF10} to
non-linear PDEs perturbed by space-time white-noise \cite{Hai10}. Of course,
fractional Brownian motion can also be handled in this framework (for $H>1/4$%
) and we shall make no attempt to survey its numerous applications in
engineering, finance and other fields.

Before describing our main result, let us recall in more detail some aspects
of Gaussian rough path theory (e.g. \cite{FV10AIHP}, \cite[Chapter 15]{FV10}%
, \cite{FV11}). The basic object is a centred, continuous Gaussian process
with sample paths $X\left( \omega \right) =\left( X^{1}\left( \omega \right)
,\ldots ,X^{d}\left( \omega \right) \right) \colon \left[ 0,1\right]
\rightarrow $ $\mathbb{R}^{d}$ where $X^{i}$ and $X^{j}$ are independent for 
$i\neq j$. The law of this process is determined by $R_{X}\colon \left[ 0,1%
\right] ^{2}\rightarrow \mathbb{R}^{d\times d}$, the covariance function,
given by%
\begin{equation*}
R_{X}\left( s,t\right) =\text{diag}\left( E\left( X_{s}^{1}X_{t}^{1}\right)
,\ldots ,E\left( X_{s}^{d}X_{t}^{d}\right) \right) .
\end{equation*}%
We need

\begin{definition}
\label{def_2D_grid_variation}Let $f=f\left( s,t\right) $ be a function from $%
\left[ 0,1\right] ^{2}$ into a normed space; for $s\leq t,u\leq v$ we define
rectangular increments as%
\begin{equation*}
f\left( 
\begin{array}{c}
s,t \\ 
u,v%
\end{array}%
\right) =f\left( t,v\right) -f\left( t,u\right) -f\left( s,v\right) +f\left(
s,u\right) .
\end{equation*}%
For $\rho \geq 1$ we then set%
\begin{equation*}
V_{\rho }\left( f,\left[ s,t\right] \times \left[ u,v\right] \right) =\left(
\sup_{\substack{ D\subset \left[ s,t\right]  \\ \tilde{D}\subset \left[ u,v%
\right] }}\sum_{\substack{ t_{i}\in D  \\ \tilde{t}_{j}\in \tilde{D}}}%
\left\vert f\left( 
\begin{array}{c}
t_{i},t_{i+1} \\ 
\tilde{t}_{j},\tilde{t}_{j+1}%
\end{array}%
\right) \right\vert ^{\rho }\right) ^{1/\rho }
\end{equation*}%
where the supremum is taken over all partitions $D$ and $\tilde{D}$ of the
intervals $\left[ s,t\right] $ resp. $\left[ u,v\right] $. If $V_{\rho }(f,%
\left[ 0,1\right] ^{2})<\infty $ we say that $f$ has finite ($2D$) $\rho $%
-variation.
\end{definition}

The main result in this context (see e.g. \cite[Theorem 15.33]{FV10}, \cite%
{FV11}) now asserts that if there exists $\rho <2$ such that $V_{\rho
}\left( R_{X},\left[ 0,1\right] ^{2}\right) <\infty $ then $X$ lifts to an 
\emph{enhanced Gaussian process} $\mathbf{X}$ with sample paths in the $p$%
-variation rough path space $C^{0,p-var}\left( \left[ 0,1\right] ,G^{\left[ p%
\right] }\left( \mathbb{R}^{d}\right) \right) $, any $p\in \left( 2\rho
,4\right) $. (This and other notations are introduced in section 2.) This
lift is "natural" in the sense that for a large class of smooth
approximations $X^{\left( k\right) }$ of $X$ (say piecewise linear,
mollifier, Karhunen-Loeve) the corresponding iterated integrals of $%
X^{\left( k\right) }$ converge (in probability) to $\mathbf{X}$ with respect
to the $p$-variation rough path metric. (We recall from \cite{FV10} that $%
\rho _{p\text{-var}}$, the so-called inhomogeneous $p$-variation metric for $%
G^{N}\left( \mathbb{R}^{d}\right) $-valued paths, is called $p$-variation
rough path metric when $\left[ p\right] =N$; the It\={o}-Lyons map enjoys
local Lipschitz regularity in this $p$-variation rough path metric.)
Moreover, this condition is sharp; indeed fBM falls into this framework with 
$\rho =1/\left( 2H\right) $ and we known that piecewise-linear
approximations to L\'{e}vy's area diverge when $H=1/4$.

Our main result (cf. Theorem \ref{theorem_main01}), when applied to
(mesh-size $1/k$) piecewise linear approximations $X^{\left( k\right) }$ of $%
X$, reads as follows.

\begin{theorem}
\label{theorem_main01_intro}Let $X=\left( X^{1},\ldots ,X^{d}\right) \colon %
\left[ 0,1\right] \rightarrow \mathbb{R}^{d}$ be a centred Gaussian process
on a probability space $\left( \Omega ,\mathcal{F},P\right) $ with
continuous sample paths where $X^{i}$ and $X^{j}$ are independent for $i\neq
j$. Assume that the covariance $R_{X}$ has finite $\rho $-variation for $%
\rho \in \lbrack 1,2)$ and $K\geq V_{\rho }\left( R_{X},\left[ 0,1\right]
^{2}\right) $. Then there is an enhanced Gaussian process $\mathbf{X}$ with
sample paths a.s. in $C^{0,p-var}\left( \left[ 0,1\right] ,G^{\left[ p\right]
}\left( \mathbb{R}^{d}\right) \right) $ for any $p\in \left( 2\rho ,4\right) 
$ and$\in $%
\begin{equation*}
\left\vert \rho _{p-var}\left( S_{\left[ p\right] }\left( X^{\left( k\right)
}\right) ,\mathbf{X}\right) \right\vert _{L^{r}}\rightarrow 0
\end{equation*}%
for $k\rightarrow \infty $ and every $r\geq 1$ ($\left\vert \cdot
\right\vert _{L^{r}}$ denotes just the usual $L^{r}\left( P\right) $-norm
for real valued random variables here). Moreover, for any $\gamma >\rho $
such that $\frac{1}{\gamma }+\frac{1}{\rho }>1$ and any $q>2\gamma $ and $%
N\in \mathbb{N}$ there is a constant $C=C\left( q,\rho ,\gamma ,K,N\right) $
such that%
\begin{equation*}
\left\vert \rho _{q-var}\left( S_{N}\left( X^{\left( k\right) }\right)
,S_{N}\left( \mathbf{X}\right) \right) \right\vert _{L^{r}}\leq
Cr^{N/2}\sup_{0\leq t\leq 1}\left\vert X_{t}^{\left( k\right)
}-X_{t}\right\vert _{L^{2}}^{1-\frac{\rho }{\gamma }}
\end{equation*}%
holds for every $k\in \mathbb{N}$.
\end{theorem}

As an immediate consequence we obtain (essentially) sharp a.s. convergence
rates for Wong-Zakai approximations and the simplified step-$3$ Euler scheme.

\begin{corollary}
Consider a RDE with $C^{\infty }$-bounded vector fields driven by a Gaussian
H\"{o}lder rough path $\mathbf{X}$. Then mesh-size $1/k$ Wong-Zakai
approximations (i.e. solutions of ODEs driven by $X^{\left( k\right) }$)
converge uniformly with a.s. rate $k^{-\left( 1/\rho -1/2-\varepsilon
\right) }$, any $\varepsilon >0$, to the RDE solution. The same rate is
valid for the simplified (and implementable) step-$3$ Euler scheme.
\end{corollary}

\begin{proof}
See Corollary \ref{cor_wong_zakai_piecew_lin} and Corollary \ref%
{Cor_rate_simple_euler}.
\end{proof}

\bigskip Several remarks are in order.

\begin{itemize}
\item Rough path analysis usually dictates that $N=2$ (resp. $N=3$) levels
need to be considered when $\rho \in \lbrack 1,3/2)$ resp. $\rho \in \lbrack
3/2,2)$. Interestingly, the situation for the Wong-Zakai error is quite
different here - referring to Theorem \ref{theorem_main01_intro}, when $\rho
=1$ we can and will take $\gamma $ arbitrarily large in order to obtain the
optimal convergence rate. Since $\rho _{q-\text{var}}$ is a rough path
metric only in the case $N=\left[ q\right] \geq \left[ 2\gamma \right] $, we
see that we need to consider all levels $N$ which is what Theorem \ref%
{theorem_main01_intro} allows us to do. On the other hand, as $\rho $
approaches $2$, there is not so much room left for taking $\gamma >\rho $.
Even so, we can always find $\gamma $ with $\left[ \gamma \right] =2$ such
that $1/\gamma +1/\rho >1$. Picking $q>2\gamma $ small enough shows that we
need $N=\left[ q\right] =4$.

\item The assumption of $C^{\infty }$-bounded vector fields in the corollary
was for simplicity only. In the proof we employ local Lipschitz continuity
of the It\={o}-Lyons map for $q$-variation rough paths (involving $N=\left[ q%
\right] $ levels). As is well-known, this requires $\mathrm{Lip}%
^{q+\varepsilon }$-regularity of the vector fields\footnote{%
...in the sense of E. Stein; cf. \cite{LQ02, FV10} for instance.}. Curiously
again, we need $C^{\infty }$-bounded vector fields when $\rho =1$ but only $%
\mathrm{Lip}^{4+\varepsilon }$ as $\rho $ approaches the critical value $2$.

\item Brownian motion falls in this framework with $\rho =1$. While the a.s.
(Wong-Zakai) rate $k^{-\left( 1/2-\varepsilon \right) }$ is part of the
folklore of the subject (e.g. \cite{GS06}) the $C^{\infty }$-boundedness
assumption appears unnecessarily strong. Our explanation here is that our
rates are \textit{universal} (i.e. valid away from one universal null-set,
not dependent on starting points, coefficients etc). In particular, the
(Wong-Zakai) rates are valid on the level of stochastic flows of
diffeomorphisms; we previously discussed these issues in the Brownian
context in \cite{FR11}.

\item A surprising aspect appears in the proof of theorem \ref%
{theorem_main01_intro}. The strategy is to give sharp estimates for the
levels $n=1,\ldots ,4$ first, then performing an induction similar to the
one used in Lyon's Extension Theorem (\cite{L98}) for the higher levels.
This is in contrast to the usual considerations of level $1$ to $3$ only
(without level $4$!) which is typical for Gaussian rough paths. (Recall that
we deal with Gaussian processes which have sample paths of finite $p$%
-variation, $p\in (2\rho ,4)$, hence $\left[ p\right] \leq 3$ which
indicates that we would need to control the first $3$ levels only before
using the Extension Theorem.)

\item Although theorem \ref{theorem_main01_intro} was stated here for
(step-size $1/k$) piecewise linear approximations $\left\{ X^{\left(
k\right) }\right\} $, the estimate holds in great generality for (Gaussian)
approximations whose covariance satisfies a uniform $\rho $-variation bound.
The statements of Theorem \ref{theorem_main01} and Theorem \ref%
{theorem_as_wong_zakai_rate} reflect this generality.

\item Wong-Zakai rates for the Brownian rough path (level $2$) were first
discussed in \cite{HN08}. They prove that Wong-Zakai approximations converge
(in $\gamma $-H\"{o}lder metric) with rate $k^{-\left( 1/2-\gamma
-\varepsilon \right) }$ (in fact, a logarithmic sharpening thereof without $%
\varepsilon $) provided $\gamma \in \left( 1/3,1/2\right) $. This
restriction on $\gamma $ is serious\ (for they fully rely on "level $2$"
rough path theory); in particular, the best "uniform" Wong-Zakai convergence
rate implied is $k^{-\left( 1/2-1/3-\varepsilon \right) }=k^{-\left(
1/6-\varepsilon \right) }$ leaving a significant gap to the well-known
Brownian a.s. Wong-Zakai rate.

\item Wong-Zakai (and Milstein) rates for the fractional Brownian rough path
(level $2$ only, Hurst parameter $H>1/3$) were first discussed in \cite{DT}.
They prove that Wong-Zakai approximations converge (in $\gamma $-H\"{o}lder
metric) with rate $k^{-\left( H-\gamma -\varepsilon \right) }$ (again, in
fact, a logarithmic sharpening thereof without $\varepsilon $) provided $%
\gamma \in \left( 1/3,H\right) $. Again, the restriction on $\gamma $ is
serious\ and the best "uniform" Wong-Zakai convergence rate - and the
resulting rate for the Milstein scheme - is $k^{-\left( H-1/3-\varepsilon
\right) }$. This should be compared to the rate $k^{-\left(
2H-1/2-\varepsilon \right) }$ obtained from our corollary. In fact, this
rate was conjectured in \cite{DT} and is sharp as may be seen from a precise
result concerning Levy's stochastic area for fBM, see \cite{NTU10}.
\end{itemize}

%\begin{acknowledgement}
%P.K. Friz has received funding from the European Research Council under the
%European Union's Seventh Framework Programme (FP7/2007-2013) / ERC grant
%agreement nr. 258237. S. Riedel is supported by an IRTG (Berlin-Zurich)
%PhD-scholarship.\bigskip
%\end{acknowledgement}

The remainder of the article is structured as follows: In Section \ref%
{section_notations}, we repeat the basic notions of (Gaussian) rough paths
theory. Section \ref{section_it_int_and_shuffle} recalls the connection
between the shuffle algebra and iterated integrals. In particular, we will
use the shuffle structure to see that in order to show the desired
estimates, we can concentrate on some iterated integrals which somehow
generate all the others. Our main tool for showing $L^{2}$ estimates on the
lower levels is multidimensional Young integration which we present in
Section \ref{section_multidim_young}. The main work, namely showing the
desired $L^{2}$-estimates for the difference of high-order iterated
integrals, is done in Section \ref{section_main_estimates}. After some
preliminary Lemmas in Subsection \ref{subsection_special_cases}, we show the
estimates for the lower levels, namely for $n=1,2,3,4$ in Subsection \ref%
{subsection_lower_levels} , then give an induction argument in Subsection %
\ref{subsection_higher_levels} for the higher levels $n>4$. Section \ref%
{section_main_result} contains our main result, namely sharp a.s.
convergence rates for a class of Wong-Zakai approximations, including
piecewise-linear and mollifier approximations. We further show in Subsection %
\ref{subsection_simple_euler} how to use these results in order to obtain
sharp convergence rates for the simplified Euler scheme.

\section{Notations and basic definitions\label{section_notations}}

For $N\in \mathbb{N}$ we define%
\begin{equation*}
T^{N}\left( \mathbb{R}^{d}\right) =\mathbb{R}\oplus \mathbb{R}^{d}\oplus
\left( \mathbb{R}^{d}\otimes \mathbb{R}^{d}\right) \oplus \ldots \oplus
\left( \mathbb{R}^{d}\right) ^{\otimes N}=\oplus _{n=0}^{N}\left( \mathbb{R}%
^{d}\right) ^{\otimes n}
\end{equation*}%
and write $\pi _{n}:T^{N}\left( \mathbb{R}^{d}\right) \rightarrow \left( 
\mathbb{R}^{d}\right) ^{\otimes n}$ for the projection on the $n$-th Tensor
level. It is clear that $T^{N}\left( \mathbb{R}^{d}\right) $ is a
(finite-dimensional) vector space. For elements $g,h\in T^{N}\left( \mathbb{R%
}^{d}\right) $, we define $g\otimes h\in T^{N}\left( \mathbb{R}^{d}\right) $
by%
\begin{equation*}
\pi _{n}\left( g\otimes h\right) =\sum_{i=0}^{n}\pi _{n-i}\left( g\right)
\otimes \pi _{i}\left( h\right) .
\end{equation*}%
One can easily check that $\left( T^{N}\left( \mathbb{R}^{d}\right)
,+,\otimes \right) $ is an associative algebra with unit element $\mathbf{1}%
=\exp \left( 0\right) =1+0+0+\ldots +0$ . We call it the \emph{truncated
tensor algebra of level }$N$. A norm is defined by%
\begin{equation*}
\left\vert g\right\vert _{T^{N}\left( \mathbb{R}^{d}\right)
}=\max_{n=0,\ldots ,N}\left\vert \pi _{n}\left( g\right) \right\vert
\end{equation*}%
which turns $T^{N}\left( \mathbb{R}^{d}\right) $ into a Banach space.

For $s<t$, we define%
\begin{equation*}
\Delta _{s,t}^{n}=\left\{ \left( u_{1},\ldots ,u_{n}\right) \in \left[ s,t%
\right] ^{n}~;~u_{1}<\ldots <u_{n}\right\}
\end{equation*}%
which is the $n$-simplex on the square $\left[ s,t\right] ^{n}$. We will use 
$\Delta =\Delta _{0,1}^{2}$ for the $2$-simplex over $\left[ 0,1\right] ^{2}$%
. A continuous map $\mathbf{x\colon }\Delta \rightarrow T^{N}\left( \mathbb{R%
}^{d}\right) $ is called \emph{multiplicative functional} if for all $s<u<t$
one has $\mathbf{x}_{s,t}=\mathbf{x}_{s,u}\mathbf{\otimes x}_{u,t}.$For a
path $x=\left( x^{1},\ldots ,x^{d}\right) \colon \left[ 0,1\right]
\rightarrow \mathbb{R}^{d}$ and $s<t$, we will use the notation $%
x_{s,t}=x_{t}-x_{s}$. If $x$ has finite variation, we define its $n$-th
iterated integral by%
\begin{eqnarray*}
\mathbf{x}_{s,t}^{n} &=&\int_{\Delta _{s,t}^{n}}dx\otimes \ldots \otimes dx
\\
&=&\sum_{1\leq i_{1},\ldots ,i_{n}\leq d}\int_{\Delta
_{s,t}^{n}}dx^{i_{1}}\ldots dx^{i_{n}}e_{i_{1}}\otimes \ldots \otimes
e_{i_{n}}\in \left( \mathbb{R}^{d}\right) ^{\otimes n}
\end{eqnarray*}%
where $\left\{ e_{1},\ldots ,e_{d}\right\} $ denotes the Euclidean basis in $%
\mathbb{R}^{d}$ and $\left( s,t\right) \in \Delta $. The canonical lift $%
S_{N}\left( x\right) \colon \Delta \rightarrow T^{N}\left( \mathbb{R}%
^{d}\right) $ is defined by%
\begin{equation*}
\pi _{n}\left( S_{N}\left( x\right) _{s,t}\right) =\left\{ 
\begin{array}{ccc}
\mathbf{x}_{s,t}^{n} & \text{if} & n\in \left\{ 1,\ldots ,N\right\} \\ 
1 & \text{if} & n=0.%
\end{array}%
\right.
\end{equation*}%
It is well know (as a consequence of Chen's theorem) that $S_{N}\left(
x\right) $ is a multiplicative functional. Actually, one can show that $%
S_{N}\left( x\right) $ takes values in the smaller set $G^{N}\left( \mathbb{R%
}^{d}\right) \subset T^{N}\left( \mathbb{R}^{d}\right) $ defined by%
\begin{equation*}
G^{N}\left( \mathbb{R}^{d}\right) =\left\{ S_{N}\left( x\right) _{0,1}:x\in
C^{1-var}\left( \left[ 0,1\right] ,\mathbb{R}^{d}\right) \right\}
\end{equation*}%
which is still a group with $\otimes $. If $\mathbf{x},\mathbf{y\colon }%
\Delta \rightarrow T^{N}\left( \mathbb{R}^{d}\right) $ are multiplicative
functionals and $p\geq 1$ we set 
\begin{equation*}
\rho _{p-var}\left( \mathbf{x},\mathbf{y}\right) :=\max_{n=1,\ldots
,N}\sup_{\left( t_{i}\right) \in \left[ 0,1\right] }\left(
\sum_{i}\left\vert \mathbf{x}_{t_{i},t_{i+1}}^{n}-\mathbf{y}%
_{t_{i},t_{i+1}}^{n}\right\vert ^{p/n}\right) ^{n/p}.
\end{equation*}%
This generalizes the $p$-variation distance induced by the usual $p$%
-variation semi-norm%
\begin{equation*}
\left\vert x\right\vert _{p-var;\left[ s,t\right] }=\left( \sup_{\left(
t_{i}\right) \subset \left[ s,t\right] }\sum_{i}\left\vert
x_{t_{i+1}}-x_{t_{i}}\right\vert ^{p}\right) ^{1/p}
\end{equation*}%
for paths $x\colon \left[ 0,1\right] \rightarrow \mathbb{R}^{d}$. The Lie
group $G^{N}\left( \mathbb{R}^{d}\right) $ admits a natural norm $\left\Vert
\cdot \right\Vert $, called the \emph{Carnot-Caratheodory norm} (cf. \cite[%
Chapter 7]{FV10}). If $\mathbf{x\colon }\Delta \rightarrow G^{N}\left( 
\mathbb{R}^{d}\right) $, we set%
\begin{equation*}
\left\Vert \mathbf{x}\right\Vert _{p-var;\left[ s,t\right] }=\left(
\sup_{\left( t_{i}\right) \subset \left[ s,t\right] }\sum_{i}\left\Vert 
\mathbf{x}_{t_{i},t_{i+1}}\right\Vert ^{p}\right) ^{1/p}.
\end{equation*}

\begin{definition}
The space $C_{o}^{0,p-var}\left( \left[ 0,1\right] ,G^{N}\left( \mathbb{R}%
^{d}\right) \right) $\textbf{\ }is defined as the set of continuous paths $%
\mathbf{x\colon }\Delta \rightarrow G^{N}\left( \mathbb{R}^{d}\right) $ for
which there exists a sequence of smooth paths $x_{k}\colon \left[ 0,1\right]
\rightarrow $ $\mathbb{R}^{d}$ such that $\rho _{p-var}\left( \mathbf{x}%
,S_{N}\left( x_{k}\right) \right) \rightarrow 0~$for $k\rightarrow \infty $.
If $N=\left[ p\right] =\max \left\{ n\in \mathbb{N}:n<p\right\} $ we call
this the \emph{space of (geometric) }$p$\emph{-rough paths}.
\end{definition}

It is clear by definition that every $p$-rough path is also a multiplicative
functional. By Lyon's First Theorem (or Extension Theorem, see \cite[Theorem
2.2.1]{L98} or \cite[Theorem 9.5]{FV10}) every $p$-rough path $\mathbf{x}$
has a unique lift to a path in $G^{N}\left( \mathbb{R}^{d}\right) $ for $%
N\geq \left[ p\right] $. We denote this lift by $S_{N}(\mathbf{x)}$ and call
it the \emph{Lyons lift}. For a $p$-rough path $\mathbf{x}$, we will also
use the notation%
\begin{equation*}
\mathbf{x}_{s,t}^{n}=\pi _{n}\left( S_{N}\left( \mathbf{x}\right)
_{s,t}\right)
\end{equation*}%
for $N\geq n$. Note that this is consistent with our former definition in
the case where $x$ had finite variation. We will always use small letters
for paths $x$ and capital letters for stochastic processes $X$. The same
notation introduced here will also be used for stochastic processes.

\begin{definition}
A function $\omega \colon \Delta \rightarrow \mathbb{R}^{+}$ is called a $%
\left( 1D\right) $ control if it is continuous and superadditive, i.e. if
for all $s<u<t$ one has%
\begin{equation*}
\omega \left( s,u\right) +\omega \left( u,t\right) \leq \omega \left(
s,t\right) .
\end{equation*}
\end{definition}

If $x\colon \left[ 0,1\right] \rightarrow \mathbb{R}^{d}$ is a continuous
path with finite $p$-variation, one can show that%
\begin{equation*}
\left( s,t\right) \mapsto V_{p}\left( x,\left[ s,t\right] \right)
^{p}:=\left\vert x\right\vert _{p-var;\left[ s,t\right] }^{p}
\end{equation*}%
is continuous and superadditive, hence defines a $1D$-control function.
Unfortunately, this is not the case for higher dimensions. Recall Definition %
\ref{def_2D_grid_variation}. If $f\colon \left[ 0,1\right] ^{2}\rightarrow 
\mathbb{R}$ has finite $p$-variation,%
\begin{equation*}
\left( s,t\right) ,\left( u,v\right) \mapsto V_{p}\left( f,\left[ s,t\right]
\times \left[ u,v\right] \right) ^{p}
\end{equation*}%
in general fails to be superadditive (cf. \cite{FV11}). Therefore, we will
need a second definition. If $A=\left[ s,t\right] \times \left[ u,v\right] $
is a rectangle in $\left[ 0,1\right] ^{2}$, we will use the notation $%
f\left( A\right) :=f\left( 
\begin{array}{c}
s,t \\ 
u,v%
\end{array}%
\right) $. We call two rectangles \emph{essentially disjoint} if their
intersection is empty or degenerate. A partition $\Pi $ of a rectangle $%
R\subset \left[ 0,1\right] ^{2}$ is a finite set of essentially disjoint
rectangles whose union is $R$. The family of all such partitions is denoted
by $\mathcal{P}\left( R\right) $.

\begin{definition}
\label{def_2D_controls}A function $\omega \colon \Delta \times \Delta
\rightarrow \mathbb{R}^{+}$ is called a $\left( 2D\right) $ control if it is
continuous, zero on degenerate rectangles and super-additive in the sense
that for all rectangles $R\subset \left[ 0,1\right] ^{2}$,%
\begin{equation*}
\sum_{i=1}^{n}\omega \left( R_{i}\right) \leq \omega \left( R\right)
\end{equation*}%
whenever $\{R_{i}:i=1,\ldots ,n\}\in \mathcal{P}\left( R\right) $. $\omega $
is called \emph{symmetric} if $\omega \left( \left[ s,t\right] \times \left[
u,v\right] \right) =\omega \left( \left[ u,v\right] \times \left[ s,t\right]
\right) $ holds for all $s<t$ and $u<v$. If $f\colon \left[ 0,1\right]
^{2}\rightarrow B$ is a continuous function, we say that its $p$-variation
is controlled by $\omega $ if $\left\vert f\left( R\right) \right\vert
^{p}\leq \omega \left( R\right) $ holds for all rectangles $R\subset \lbrack
0,1]^{2}$.
\end{definition}

It is easy to see that if $\omega $ is a $2D$ control, $\left( s,t\right)
\mapsto \omega \left( \left[ s,t\right] ^{2}\right) $ defines a $1D$-control.

\begin{definition}
For $f\colon \left[ 0,1\right] ^{2}\rightarrow \mathbb{R}$, $R\subset
\lbrack 0,1]^{2}$ a rectangle and $p\geq 1$ we define%
\begin{equation*}
\left\vert f\right\vert _{p-var;R}:=\sup_{\Pi \in \mathcal{P}\left( R\right)
}\left( \sum_{A\in \Pi }\left\vert f\left( A\right) \right\vert ^{p}\right)
^{1/p}.
\end{equation*}%
If $\left\vert f\right\vert _{p-var;\left[ 0,1\right] ^{2}}<\infty $ we say
that $f$ has finite controlled $p$-variation.
\end{definition}

The difference of $2D$ $p$-variation introduced in Definition \ref%
{def_2D_grid_variation} and \emph{controlled} $p$-variation is that in the
former, one only takes the supremum over grid-like partitions whereas in the
latter, one takes the supremum over all partitions of the rectangle. By
superadditivity, the existence of a control $\omega $ which controls the $p$%
-variation of $f$ implies that $f$ has finite controlled $p$-variation and $%
\left\vert f\right\vert _{p-var;R}\leq \omega \left( R\right) ^{1/p}$. In
this case, we can always assume w.l.o.g. that $\omega $ is symmetric,
otherwise we just substitute $\omega $ by its symmetrization $\omega _{\text{%
sym}}$ given by%
\begin{equation*}
\omega _{\text{sym}}\left( \left[ s,t\right] \times \left[ u,v\right]
\right) =\omega \left( \left[ s,t\right] \times \left[ u,v\right] \right)
+\omega \left( \left[ u,v\right] \times \left[ s,t\right] \right) .
\end{equation*}%
The connection between finite variation and finite controlled $p$-variation
is summarized in the following theorem.

\begin{theorem}
\label{theorem_comp_contr_p_var}Let $f\colon \left[ 0,1\right]
^{2}\rightarrow \mathbb{R}$ be continuous and $R\subset \left[ 0,1\right]
^{2}$ be a rectangle.

\begin{enumerate}
\item We have 
\begin{equation*}
V_{1}\left( f,R\right) =\left\vert f\right\vert _{1-var;R}.
\end{equation*}

\item For any $p\geq 1$ and $\epsilon >0$ there is a constant $C=C\left(
p,\epsilon \right) $ such that%
\begin{equation*}
\frac{1}{C}\left\vert f\right\vert _{\left( p+\epsilon \right) -var;R}\leq
V_{p-var}\left( f,R\right) \leq \left\vert f\right\vert _{p-var;R}.
\end{equation*}

\item If $f$ has finite controlled $p$-variation, then%
\begin{equation*}
R\mapsto \left\vert f\right\vert _{p-var;R}^{p}
\end{equation*}%
is a $2D$-control. In particular, there exists a $2D$-control $\omega $ such
that for all rectangles $R\subset \left[ 0,1\right] ^{2}$ we have $%
\left\vert f\left( R\right) \right\vert ^{p}\leq \omega \left( R\right) $,
i.e. $\omega $ controls the $p$-variation of $f$.
\end{enumerate}
\end{theorem}

\begin{proof}
\cite[ Theorem 1]{FV11}.
\end{proof}

In the following, unless mentioned otherwise, $X$ will always be a Gaussian
process as in Theorem \ref{theorem_main01_intro} and $\mathbf{X}$ denotes
the natural Gaussian rough path. We will need the following Proposition:

\begin{proposition}
\label{prop_moments_lp}Let $X$ be as in Theorem \ref{theorem_main01_intro}
and assume that $\omega $ controls the $\rho $-variation of the covariance
of $X$, $\rho \in \lbrack 1,2)$. Then for every $n\in \mathbb{N}$ there is a
constant $C\left( n\right) =C\left( n,\rho \right) $ such that%
\begin{equation*}
\left\vert \mathbf{X}_{s,t}^{n}\right\vert _{L^{2}}\leq C\left( n\right)
\omega \left( \left[ s,t\right] ^{2}\right) ^{\frac{n}{2\rho }}
\end{equation*}%
for any $s<t$.
\end{proposition}

\begin{proof}
For $n=1,2,3$ this is proven in \cite[Proposition 15.28]{FV10}. For $n\geq 4$
and fixed $s<t$, we set $\tilde{X}_{\tau }:=\frac{1}{\omega \left( \left[ s,t%
\right] ^{2}\right) ^{\frac{1}{2\rho }}}X_{s+\tau \left( t-s\right) }$. Then 
$\left\vert R_{\tilde{X}}\right\vert _{\rho -var;\left[ 0,1\right] }^{\rho
}\leq 1=:K$ and by the standard (deterministic) estimates for the Lyons lift,%
\begin{equation*}
\frac{\left\vert \mathbf{X}_{s,t}^{n}\right\vert ^{1/n}}{\omega \left( \left[
s,t\right] ^{2}\right) ^{\frac{1}{2\rho }}}\leq c_{1}\left\Vert S_{n}\left( 
\mathbf{\tilde{X}}\right) \right\Vert _{p-var;\left[ 0,1\right] }\leq
c_{2}\left( n,p\right) \left\Vert \mathbf{\tilde{X}}\right\Vert _{p-var;%
\left[ 0,1\right] }
\end{equation*}%
for any $p\in \left( 2\rho ,4\right) $. Now we take the $L^{2}$-norm on both
sides. From \cite[Theorem 15.33]{FV10} we know that $\left\vert \left\Vert 
\mathbf{\tilde{X}}\right\Vert _{p-var;\left[ 0,1\right] }\right\vert
_{L^{2}} $ is bounded by a constant only depending on $p,\rho $ and $K$
which shows the claim.

Alternatively (and more in the spirit of the forthcoming arguments), one
performs an induction similar (but easier) as in the proof of Proposition %
\ref{prop_key_estimate_higher_levels}.
\end{proof}

\section{Iterated integrals and the shuffle algebra\label%
{section_it_int_and_shuffle}}

Let $x=\left( x^{1},\ldots ,x^{d}\right) \colon \left[ 0,1\right]
\rightarrow \mathbb{R}^{d}$ be a path of finite variation. Forming finite
linear combinations of iterated integrals of the form%
\begin{equation*}
\int_{\Delta _{0,1}^{n}}dx^{i_{1}}\ldots dx^{i_{n}},\quad i_{1},\ldots
,i_{n}\in \left\{ 1,\ldots ,d\right\} ,n\in \mathbb{N}
\end{equation*}%
defines a vector space over $\mathbb{R}$. In this section, we will see that
this vector space is also an algebra where the product is given simply by
taking the usual multiplication. Moreover, we will describe precisely how
the product of two iterated integrals looks like.

\subsection{The shuffle algebra}

Let $A$ be a set which we will call from now on the alphabet. In the
following, we will only consider the finite alphabet $A=\left\{ a,b,\ldots
\right\} =\left\{ a_{1},a_{2},\ldots ,a_{d}\right\} =\left\{ 1,\ldots
,d\right\} $. We denote by $A^{\ast }$ the set of words composed by the
letters of $A$, hence $w=a_{i_{1}}a_{i_{2}}\ldots a_{i_{n}},~a_{i_{j}}\in A$%
. The empty word is denoted by $e$. $A^{+}$ is the set of non-empty words.
The length of the word is denoted by $\left\vert w\right\vert $ and $%
\left\vert w\right\vert _{a}$ denotes the number of occurrences of the
letter $a$. We denote by $\mathbb{R}\left\langle A\right\rangle $ the vector
space of noncommutative polynomials on $A$ over $\mathbb{R}$, hence every $%
P\in \mathbb{R}\left\langle A\right\rangle $ is a linear combination of
words in $A^{\ast }$ with coefficients in $\mathbb{R}$. $(P,w)$ denotes the
coefficient in $P$ of the word $w$. Hence every polynomial $P$ can be
written as%
\begin{equation*}
P=\sum_{w\in A^{\ast }}(P,w)w
\end{equation*}%
and the sum is finite since the $(P,w)$ are non-zero only for a finite set
of words $w$. We define the degree of $P$ as%
\begin{equation*}
\deg \left( P\right) =\max \left\{ \left\vert w\right\vert ~;~\left(
P,w\right) \neq 0\right\} .
\end{equation*}%
A polynomial is called \emph{homogeneous} if all monomials have the same
degree. We want to define a product on $\mathbb{R}\left\langle
A\right\rangle $. Since a polynomial is determined by its coefficients on
each word, we can define the product $PQ$ of $P$ and $Q$ by%
\begin{equation*}
(PQ,w)=\sum_{w=uv}(P,u)(Q,v).
\end{equation*}%
Note that this definition coincides with the usual multiplication in a
(noncommutative) polynomial ring. We call this product the \emph{%
concatenation product}\textit{\ }and the algebra $\mathbb{R}\left\langle
A\right\rangle $ endowed with this product the \emph{concatenation algebra}.

There is another product on $\mathbb{R}\left\langle A\right\rangle $ which
will be of special interest for us. We need some notation first. Given a
word $w=a_{i_{1}}a_{i_{2}}\ldots a_{i_{n}}$ and a subsequence $%
U=(j_{1},j_{2},\ldots ,j_{k})$ of $\left( i_{1},\ldots ,i_{n}\right) $, we
denote by $w(U)$ the word $a_{j_{1}}a_{j_{2}}\ldots a_{j_{k}}$ and we call $%
w(U)$ a \textit{subword} of $w$. If $w,u,v$ are words and if $w$ has length $%
n$, we denote by $\left( 
\begin{array}{c}
w \\ 
u\quad v%
\end{array}%
\right) $ the number of subsequences $U$ of $(1,\ldots ,n)$ such that $%
w(U)=u $ and $w(U^{c})=v$.

\begin{definition}
The (homogeneous) polynomial%
\begin{equation*}
u\ast v=\sum_{w\in A^{\ast }}\left( 
\begin{array}{c}
w \\ 
u\quad v%
\end{array}%
\right) w
\end{equation*}%
\newline
is called the \emph{shuffle product} of $u$ and $v$. By linearity we extend
it to a product on $\mathbb{R}\left\langle A\right\rangle $.
\end{definition}

In order to proof our main result, we want to use some sort of induction
over the length of the words. Therefore, the following definition will be
useful.

\begin{definition}
If $U$ is a set of words of the same length, we call a subset $\left\{
w_{1},\ldots ,w_{k}\right\} $ of $U$ a \emph{generating set for }$U$ if for
every word $w\in U$ there is a polynomial $R$ and real numbers $\lambda
_{1},\ldots ,\lambda _{k}$ such that%
\begin{equation*}
w=\sum_{j=1}^{k}\lambda _{j}w_{j}+R
\end{equation*}%
where $R$ is of the form $R=\sum_{u,v\in A^{+}}\mu _{u,v}u\ast v$ for real
numbers $\mu _{u,v}$.
\end{definition}

\begin{definition}
We say that a word $w$ is\emph{\ composed by} $a_{1}^{n_{1}},\ldots
,a_{d}^{n_{d}}$ if $w\in \left\{ a_{1},\ldots ,a_{d}\right\} ^{\ast }$ and $%
\left\vert w\right\vert _{a_{i}}=n_{i}$ for $i=1,\ldots ,d$, hence every
letter appears in the word with the given multiplicity.
\end{definition}

The aim now is to find a (possibly small) generating set for the set of all
words composed by some given letters. The next definition introduces a
special class of words which will be important for us.

\begin{definition}
Let $A$ be totally ordered and put on $A^{\ast }$ the alphabetical order. If 
$w$ is a word such that whenever $w=uv$ for $u,v\in A^{+}$ one has $u<v$,
then $w$ is called a \emph{Lyndon word}.
\end{definition}

\begin{proposition}
\label{prop_generating_sets}

\begin{enumerate}
\item For the set $\{$words composed by $a,a,b\}$ a generating set is given
by $\{aab\}$.

\item For the set $\{$words composed by $a,a,a,b\}$ a generating set is
given by $\{aaab\}$.

\item For the set $\{$words composed by $a,a,b,b\}$ a generating set is
given by $\{aabb\}$.

\item For the set $\{$words composed by $a,a,b,c\}$ a generating set is
given by $\{aabc,aacb,baac\}$.
\end{enumerate}
\end{proposition}

\begin{proof}
Consider the alphabet $A=\left\{ a,b,c\right\} $. We choose the order $a<b<c$%
. A general theorem states that every word $w$ has a unique decreasing
factorization into Lyndon words, i.e. $w=l_{1}^{i_{1}}\ldots l_{k}^{i_{k}}$
where $l_{1}>\ldots >l_{k}$ are Lyndon words and $i_{1},\ldots ,i_{k}\geq 1$
(see \cite[Theorem 5.1 and Corollary 4.7]{R93}), and the formula%
\begin{equation*}
\frac{1}{i_{1}!\ldots i_{k}!}l_{1}^{\ast i_{1}}\ast \ldots \ast l_{k}^{\ast
i_{k}}=w+\sum_{u<w}\alpha _{u}u
\end{equation*}%
holds, where $\alpha _{u}$ are some natural integers (see again \cite[%
Theorem 6.1]{R93}). By repeatedly applying this formula for the words in the
sum on the right hand side, it follows that a generating set for each of the
sets in $\left( 1\right) $ to $\left( 4\right) $ is given exactly by the
Lyndon words composed by these letters. One can easily show that indeed $aab$%
, $aaab$ and $aabb$ are the only Lyndon words composed by the corresponding
letters. The Lyndon words composed by $a,a,b,c$ are $\left\{
aabc,abac,aacb\right\} $ which therefore is a generating set for $\{$words
composed by $a,a,b,c\}$. From the shuffle identity%
\begin{equation*}
abac=baac+aabc+aacb-b\ast aac
\end{equation*}%
it follows that also $\{aabc,aacb,baac\}$ generates this set.
\end{proof}

\subsection{The connection to iterated integrals}

Let $x=(x^{1},\ldots ,x^{d})\colon \left[ 0,1\right] \rightarrow \mathbb{R}%
^{d}$ be a path of finite variation and fix $s<t\in \left[ 0,1\right] $. For
a word $w=\left( a_{i_{1}}\ldots a_{i_{n}}\right) \in A^{\ast }$, $A=\left\{
1,\ldots ,d\right\} $ we define%
\begin{equation*}
\mathbf{x}^{w}=\left\{ 
\begin{array}{ccc}
\int_{\Delta _{s,t}^{n}}dx^{i_{1}}\ldots dx^{i_{n}} & \text{if} & w\in A^{+}
\\ 
1 & \text{if} & w=e%
\end{array}%
\right. .
\end{equation*}

Let $\left( \mathbb{R}\left\langle A\right\rangle ,+,\ast \right) $ be the
shuffle algebra over the alphabet $A$. We define a map $\Phi \colon \mathbb{R%
}\left\langle A\right\rangle \rightarrow \mathbb{R}$ by $\Phi \left(
w\right) =\mathbf{x}_{s,t}^{w}$ and extend it linearly to polynomials $P\in 
\mathbb{R}\left\langle A\right\rangle $. The key observation is the
following:

\begin{theorem}
\label{isomorphism_shufflealgebra}$\Phi $ is an algebra homomorphism from
the shuffle algebra $\left( \mathbb{R}\left\langle A\right\rangle ,+,\ast
\right) $ to $\left( \mathbb{R},+,\cdot \right) $.
\end{theorem}

\begin{proof}
\cite{R93}, Corollary 3.5.
\end{proof}

The next proposition shows that we can restrict ourselves in showing the
desired estimates only for the iterated integrals which generate the others.

\begin{proposition}
\label{proposition_key_shuffle} Let $\left( X,Y\right) =\left(
X^{1},Y^{1},\ldots ,X^{d},Y^{d}\right) $ be a Gaussian process on $\left[ 0,1%
\right] $ with paths of finite variation. Let $A=\left\{ 1,\ldots ,d\right\} 
$ be the alphabet, let $U$ be a set of words of length $n$ and $V=\left\{
w_{1,}\ldots ,w_{k}\right\} $ be a generating set for $U$. Let $\omega $ be
a control, $\rho ,\gamma \geq 1$ constants and $s<t\in \left[ 0,1\right] $.
Assume that there are constants $C=C\left( \left\vert w\right\vert \right) $
such that 
\begin{equation*}
\left\vert \mathbf{X}_{s,t}^{w}\right\vert _{L^{2}}\leq C\left( \left\vert
w\right\vert \right) \omega \left( s,t\right) ^{\frac{\left\vert
w\right\vert }{2\rho }}\quad \text{and\quad }\left\vert \mathbf{Y}%
_{s,t}^{w}\right\vert _{L^{2}}\leq C\left( \left\vert w\right\vert \right)
\omega \left( s,t\right) ^{\frac{\left\vert w\right\vert }{2\rho }}
\end{equation*}%
holds for every word $w\in A^{\ast }$ with $\left\vert w\right\vert \leq n-1$%
. Assume also that for some $\epsilon >0$%
\begin{equation*}
\left\vert \mathbf{X}_{s,t}^{w}-\mathbf{Y}_{s,t}^{w}\right\vert _{L^{2}}\leq
C\left( \left\vert w\right\vert \right) \epsilon \omega \left( s,t\right) ^{%
\frac{1}{2\gamma }}\omega \left( s,t\right) ^{\frac{\left\vert w\right\vert
-1}{2\rho }}
\end{equation*}%
holds for every word $w$ with $\left\vert w\right\vert \leq n-1$ and $w\in V$%
. Then there is a constant $\tilde{C}$ which depends on the constants $C$,
on $n$ and on $d$ such that%
\begin{equation*}
\left\vert \mathbf{X}_{s,t}^{w}-\mathbf{Y}_{s,t}^{w}\right\vert _{L^{2}}\leq 
\tilde{C}\epsilon \omega \left( s,t\right) ^{\frac{1}{2\gamma }}\omega
\left( s,t\right) ^{\frac{n-1}{2\rho }}
\end{equation*}%
holds for every $w\in U$.
\end{proposition}

\begin{remark}
We could account for the factor $\omega \left( s,t\right) ^{\frac{1}{2\gamma 
}}$ in $\epsilon $ here but the present form is how we shall use this
proposition later on.
\end{remark}

\begin{proof}
Consider a copy $\bar{A}$ of $A$. If $a\in A$, we denote by $\bar{a}$ the
corresponding letter in $\bar{A}$. If $w=a_{i_{1}}\ldots a_{i_{n}}\in
A^{\ast }$, we define $\bar{w}=\bar{a}_{i_{1}}\ldots \bar{a}_{i_{n}}\in
A^{\ast }$ and in the same way we define $\bar{P}\in \mathbb{R}\left\langle 
\bar{A}\right\rangle $ for $P\in \mathbb{R}\left\langle A\right\rangle $.
Now we consider $\mathbb{R}\left\langle A\dot{\cup}\bar{A}\right\rangle $
equipped with the usual shuffle product. Define $\Psi \colon \mathbb{R}%
\left\langle A\dot{\cup}\bar{A}\right\rangle \rightarrow \mathbb{R}$ by%
\begin{equation*}
\Psi \left( w\right) =\int_{\Delta _{s,t}^{n}}dZ^{b_{i_{1}}}\ldots
dZ^{b_{i_{n}}}
\end{equation*}%
for a word $w=b_{i_{1}}\ldots b_{i_{n}}$ where%
\begin{equation*}
Z^{b_{j}}=\left\{ 
\begin{array}{ccc}
X^{a_{j}} & \text{for} & b_{j}=a_{j} \\ 
Y^{\bar{a}_{j}} & \text{for} & b_{j}=\bar{a}_{j}%
\end{array}%
\right.
\end{equation*}%
and extend this definition linearly. By Theorem \ref%
{isomorphism_shufflealgebra}, we know that $\Psi $ is an algebra
homomorphism. Take $w\in U$. By assumption, we know that there is a vector $%
\lambda =\left( \lambda _{1},\ldots ,\lambda _{k}\right) $ such that%
\begin{equation*}
w-\bar{w}=\sum_{j=1}^{k}\lambda _{j}\left( w_{j}-\bar{w}_{j}\right) +R-\bar{R%
}
\end{equation*}%
where $R$ is of the form $R=\sum_{u,v\in A^{+},\left\vert u\right\vert
+\left\vert v\right\vert =n}\mu _{u,v}\,u\ast v$ with real numbers $\mu
_{u,v}$. Applying $\Psi $ and taking the $L^{2}$ norm yields%
\begin{eqnarray*}
\left\vert \mathbf{X}_{s,t}^{w}-\mathbf{Y}_{s,t}^{w}\right\vert _{L^{2}}
&\leq &\sum_{l=1}^{k}\left\vert \lambda _{j}\right\vert \left\vert \mathbf{X}%
_{s,t}^{w_{j}}-\mathbf{Y}_{s,t}^{w_{j}}\right\vert _{L^{2}}+\left\vert \Psi
\left( R-\bar{R}\right) \right\vert _{L^{2}} \\
&\leq &c_{1}\epsilon \omega \left( s,t\right) ^{\frac{1}{2\gamma }}\omega
\left( s,t\right) ^{\frac{n-1}{2\rho }}+\left\vert \Psi \left( R-\bar{R}%
\right) \right\vert _{L^{2}}.
\end{eqnarray*}%
Now,%
\begin{equation*}
R-\bar{R}=\sum_{u,v}\mu _{u,v}\left( u\ast v-\bar{u}\ast \bar{v}\right)
=\sum_{u,v}\mu _{u,v}\left( u-\bar{u}\right) \ast v+\mu _{u,v}\bar{u}\ast
\left( v-\bar{v}\right) .
\end{equation*}%
Applying $\Psi $ and taking the $L^{2}$ norm gives then%
\begin{eqnarray*}
\left\vert \Psi \left( R-\bar{R}\right) \right\vert _{L^{2}} &\leq
&\sum_{u,v}\left\vert \mu _{u,v}\right\vert \left\vert \left( \mathbf{X}%
_{s,t}^{u}-\mathbf{Y}_{s,t}^{u}\right) \mathbf{X}_{s,t}^{v}\right\vert
_{L^{2}}+\left\vert \mu _{u,v}\right\vert \left\vert \mathbf{Y}%
_{s,t}^{u}\left( \mathbf{X}_{s,t}^{v}-\mathbf{Y}_{s,t}^{v}\right)
\right\vert _{L^{2}} \\
&\leq &\sum_{u,v}c_{2}\left( \left\vert \mathbf{X}_{s,t}^{u}-\mathbf{Y}%
_{s,t}^{u}\right\vert _{L^{2}}\left\vert \mathbf{X}_{s,t}^{v}\right\vert
_{L^{2}}+\left\vert \mathbf{Y}_{s,t}^{u}\right\vert _{L^{2}}\left\vert 
\mathbf{X}_{s,t}^{v}-\mathbf{Y}_{s,t}^{v}\right\vert _{L^{2}}\right) \\
&\leq &\sum_{u,v}c_{3}\epsilon \omega \left( s,t\right) ^{\frac{1}{2\gamma }%
}\omega \left( s,t\right) ^{\frac{\left\vert v\right\vert +\left\vert
u\right\vert -1}{2\rho }} \\
&\leq &c_{4}\epsilon \omega \left( s,t\right) ^{\frac{1}{2\gamma }}\omega
\left( s,t\right) ^{\frac{n-1}{2\rho }}
\end{eqnarray*}%
where we used equivalence of $L^{q}$-norms in the Wiener Chaos (cf. \cite[%
Proposition 15.19 and Theorem D.8]{FV10}). Putting all together shows the
assertion.
\end{proof}

\section{Multidimensional Young-integration and grid-controls\label%
{section_multidim_young}}

Let $f\colon \left[ 0,1\right] ^{n}\rightarrow \mathbb{R}$ be a continuous
function. If $s_{1}<t_{1},\ldots ,s_{n}<t_{n}$ and $u_{1},\ldots ,u_{n}$ are
elements in $\left[ 0,1\right] $, we make the following recursive definition:%
\begin{eqnarray*}
f\left( 
\begin{array}{c}
s_{1},t_{1} \\ 
u_{2} \\ 
\vdots \\ 
u_{n}%
\end{array}%
\right) &:&=f\left( 
\begin{array}{c}
t_{1} \\ 
u_{2} \\ 
\vdots \\ 
u_{n}%
\end{array}%
\right) -f\left( 
\begin{array}{c}
s_{1} \\ 
u_{2} \\ 
\vdots \\ 
u_{n}%
\end{array}%
\right) \quad \text{and} \\
f\left( 
\begin{array}{c}
s_{1},t_{1} \\ 
\vdots \\ 
s_{k-1},t_{k-1} \\ 
s_{k},t_{k} \\ 
u_{k+1} \\ 
\vdots \\ 
u_{n}%
\end{array}%
\right) &:&=f\left( 
\begin{array}{c}
s_{1},t_{1} \\ 
\vdots \\ 
s_{k-1},t_{k-1} \\ 
t_{k} \\ 
u_{k+1} \\ 
\vdots \\ 
u_{n}%
\end{array}%
\right) -f\left( 
\begin{array}{c}
s_{1},t_{1} \\ 
\vdots \\ 
s_{k-1},t_{k-1} \\ 
s_{k} \\ 
u_{k+1} \\ 
\vdots \\ 
u_{n}%
\end{array}%
\right) .
\end{eqnarray*}%
We will also use the simpler notation%
\begin{equation*}
f\left( R\right) =f\left( 
\begin{array}{c}
s_{1},t_{1} \\ 
\vdots \\ 
s_{n},t_{n}%
\end{array}%
\right)
\end{equation*}%
for the rectangle $R=\left[ s_{1},t_{1}\right] \times \ldots \times \left[
s_{n},t_{n}\right] \subset \left[ 0,1\right] ^{n}$. Note that for $n=2$ this
is consistent with our initial definition of $f\left( 
\begin{array}{c}
s_{1},t_{1} \\ 
s_{2},t_{2}%
\end{array}%
\right) $. If $f,g\colon \left[ 0,1\right] ^{n}\rightarrow \mathbb{R}$ are
continuous functions, the $n$-dimensional Young-integral is defined by%
\begin{eqnarray*}
&&\int_{\left[ s_{1},t_{1}\right] \times \ldots \times \left[ s_{n},t_{n}%
\right] }f\left( x_{1},\ldots ,x_{n}\right) \,dg\left( x_{1},\ldots
,x_{n}\right) \\
&:&=\lim_{\left\vert D_{1}\right\vert ,\ldots ,\left\vert D_{n}\right\vert
\rightarrow 0}\sum_{\substack{ \left( t_{i_{1}}^{1}\right) \subset D_{1}  \\ %
\vdots  \\ \left( t_{i_{n}}^{n}\right) \subset D_{n}}}f\left(
t_{i_{1}}^{1},\ldots ,t_{i_{n}}^{n}\right) g\left( 
\begin{array}{c}
t_{i_{1}}^{1},t_{i_{1}+1}^{1} \\ 
\vdots \\ 
t_{i_{n}}^{n},t_{i_{n}+1}^{n}%
\end{array}%
\right)
\end{eqnarray*}%
if this limit exists. Take $p\geq 1$. The $n$-dimensional $p$-variation of $%
f $ is defined by%
\begin{equation*}
V_{p}\left( f,\left[ s_{1},t_{1}\right] \times \ldots \times \left[
s_{n},t_{n}\right] \right) =\left( \sup_{\substack{ D_{1}\subset \left[
s_{1},t_{1}\right]  \\ \vdots  \\ D_{n}\subset \left[ s_{n},t_{n}\right] }}%
\sum_{\substack{ \left( t_{i_{1}}^{1}\right) \subset D_{1}  \\ \vdots  \\ %
\left( t_{i_{n}}^{n}\right) \subset D_{n}}}\left\vert f\left( 
\begin{array}{c}
t_{i_{1}}^{1},t_{i_{1}+1}^{1} \\ 
\vdots \\ 
t_{i_{n}}^{n},t_{i_{n}+1}^{n}%
\end{array}%
\right) \right\vert ^{p}\right) ^{1/p}
\end{equation*}%
and if $V_{p}\left( f,\left[ 0,1\right] ^{n}\right) <\infty $ we say that $f$
has finite ($n$-dimensional) $p$-variation. The fundamental theorem is the
following:

\begin{theorem}
Assume that $f$ has finite $p$-variation and $g$ finite $q$-variation where $%
\frac{1}{p}+\frac{1}{q}>1$. Then the joint Young-integral below exists and
there is a constant $C=C\left( p,q\right) $ such that%
\begin{eqnarray*}
&&\left\vert \int_{\left[ s_{1},t_{1}\right] \times \ldots \times \left[
s_{n},t_{n}\right] }f\left( 
\begin{array}{c}
s_{1},u_{1} \\ 
\vdots \\ 
s_{n},u_{n}%
\end{array}%
\right) \,dg\left( u_{1},\ldots ,u_{n}\right) \right\vert \\
&\leq &CV_{p}\left( f,\left[ s_{1},t_{1}\right] \times \ldots \times \left[
s_{n},t_{n}\right] \right) V_{q}\left( g,\left[ s_{1},t_{1}\right] \times
\ldots \times \left[ s_{n},t_{n}\right] \right) .
\end{eqnarray*}
\end{theorem}

\begin{proof}
\cite{T02}, Theorem 1.2 (c).
\end{proof}

We will mainly consider the case $n=2$, but we will also need $n=3$ and $4$
later on. In particular, the discussion of level $n=4$ will require us to
work with $4D$ grid control functions which we now introduce. With no extra
complication we make the following general definition.

\begin{definition}[$n$-dimensional grid control]
A map $\tilde{\omega}\colon \underbrace{\Delta \times \ldots \times \Delta }%
_{\text{n-times}}\rightarrow \mathbb{R}^{+}$ is called a $n$-$D$\emph{\
grid-control} if it is continuous and partially super-additive, i.e. for all 
$\left( s_{1},t_{1}\right) ,\ldots ,\left( s_{n},t_{n}\right) \in \Delta $
and $s_{i}<u_{i}<t_{i}$ we have%
\begin{eqnarray*}
&&\tilde{\omega}\left( \left[ s_{1},t_{1}\right] \times \ldots \times \left[
s_{i},u_{i}\right] \times \ldots \times \left[ s_{n},t_{n}\right] \right) +%
\tilde{\omega}\left( \left[ s_{1},t_{1}\right] \times \ldots \times \left[
u_{i},t_{i}\right] \times \ldots \times \left[ s_{n},t_{n}\right] \right) \\
&\leq &\tilde{\omega}\left( \left[ s_{1},t_{1}\right] \times \ldots \times %
\left[ s_{i},t_{i}\right] \times \ldots \times \left[ s_{n},t_{n}\right]
\right)
\end{eqnarray*}%
for every $i=1,\ldots ,n$. $\tilde{\omega}$ is called symmetric if%
\begin{equation*}
\tilde{\omega}\left( \left[ s_{1},t_{1}\right] \times \ldots \times \left[
s_{n},t_{n}\right] \right) =\tilde{\omega}\left( \left[ s_{\sigma \left(
1\right) },t_{\sigma \left( 1\right) }\right] \times \ldots \times \left[
s_{\sigma \left( n\right) },t_{\sigma \left( n\right) }\right] \right)
\end{equation*}%
holds for every $\sigma \in S_{n}$.
\end{definition}

The point of this definition is that $\left\vert f\left( A\right)
\right\vert ^{p}\leq \tilde{\omega}\left( A\right) $ for every rectangle $%
A\subset \left[ 0,1\right] ^{n}$ implies that $V_{p}\left( f,R\right)
^{p}\leq \tilde{\omega}\left( R\right) $ for every rectangle $R\subset \left[
0,1\right] ^{n}$. Note that a 2D control in the sense of Definition \ref%
{def_2D_controls} is automatically a $2D$ grid-control. The following
immediate properties will be used in Section \ref{section_n=4} with $m=n=2$.

\begin{lemma}
\begin{enumerate}
\item The restriction of a $\left( m+n\right) $-dimensional grid-control to $%
m$ arguments is a $m$-dimensional grid-control.

\item The product of a $m$- and a $n$-dimensional grid-control is a $(m+n)$%
-dimensional grid-control.
\end{enumerate}
\end{lemma}

\subsection{Iterated $2D$-integrals}

In the $1$-dimensional case, the classical Young-theory allows to define
iterated integrals of functions with finite $p$-variation where $p<2$.
There, the superadditivity of $\left( s,t\right) \mapsto \left\vert \cdot
\right\vert _{p-var;\left[ s,t\right] }^{p}$ played an essential role. We
will see that Theorem \ref{theorem_comp_contr_p_var} can be used to define
and estimate iterated $2D$-integrals. This will play an important role in
Section \ref{section_main_estimates} when we estimate the $L^{2}$-norm of
iterated integrals of Gaussian processes.

\begin{lemma}
\label{lemma_kernel_iter_2D}Let $f,g\colon \left[ 0,1\right] ^{2}\rightarrow 
\mathbb{R}$ be continuous where $f$ has finite $p$-variation and $g$ finite
controlled $q$-variation with $p^{-1}+q^{-1}>1$. Let $\left( s,t\right) \in
\Delta $ and assume that $f\left( s,\cdot \right) =f\left( \cdot ,s\right)
=0 $. Define $\Phi \colon \left[ s,t\right] ^{2}\rightarrow \mathbb{R}$ by%
\begin{equation*}
\Phi \left( u,v\right) =\int_{\left[ s,u\right] \times \left[ s,v\right]
}f\,dg.
\end{equation*}%
Then there is a constant $C=C\left( p,q\right) $ such that%
\begin{equation*}
V_{q-var}\left( \Phi ;\left[ s,t\right] ^{2}\right) \leq C\left( p,q\right)
V_{p-var}\left( f;\left[ s,t\right] ^{2}\right) \left\vert g\right\vert
_{q-var;\left[ s,t\right] ^{2}}.
\end{equation*}
\end{lemma}

\begin{proof}
\begin{enumerate}
\item Let $t_{i}<t_{i+1}$ and $\tilde{t}_{j}<\tilde{t}_{j+1}$. Then,%
\begin{equation*}
\Phi 
\begin{pmatrix}
t_{i},t_{i+1} \\ 
\tilde{t}_{j},\tilde{t}_{j+1}%
\end{pmatrix}%
=\int_{\left[ t_{i},t_{i+1}\right] \times \left[ \tilde{t}_{j},\tilde{t}%
_{j+1}\right] }f\,dg.
\end{equation*}%
Now let $t_{i}<u<t_{i+1}$ and $\tilde{t}_{j}<v<\tilde{t}_{j+1}$. Then one has%
\begin{equation*}
f%
\begin{pmatrix}
t_{i},u \\ 
\tilde{t}_{j},v%
\end{pmatrix}%
=f\left( u,v\right) -f\left( t_{i},v\right) -f\left( u,\tilde{t}_{j}\right)
+f\left( t_{i},\tilde{t}_{j}\right) .
\end{equation*}%
Therefore,%
\begin{eqnarray*}
\left\vert \Phi 
\begin{pmatrix}
t_{i},t_{i+1} \\ 
\tilde{t}_{j},\tilde{t}_{j+1}%
\end{pmatrix}%
\right\vert &\leq &\left\vert \int_{\left[ t_{i},t_{i+1}\right] \times \left[
\tilde{t}_{j},\tilde{t}_{j+1}\right] }f%
\begin{pmatrix}
t_{i},u \\ 
\tilde{t}_{j},v%
\end{pmatrix}%
\,dg\left( u,v\right) \right\vert +\left\vert \int_{\left[ t_{i},t_{i+1}%
\right] \times \left[ \tilde{t}_{j},\tilde{t}_{j+1}\right] }f\left(
t_{i},v\right) \,dg\left( u,v\right) \right\vert \\
&&+\left\vert \int_{\left[ t_{i},t_{i+1}\right] \times \left[ \tilde{t}_{j},%
\tilde{t}_{j+1}\right] }f\left( u,\tilde{t}_{j}\right) \,dg\left( u,v\right)
\right\vert +\left\vert \int_{\left[ t_{i},t_{i+1}\right] \times \left[ 
\tilde{t}_{j},\tilde{t}_{j+1}\right] }f\left( t_{i},\tilde{t}_{j}\right)
\,dg\left( u,v\right) \right\vert
\end{eqnarray*}%
For the first integral we use Young $2D$-estimates to see that%
\begin{eqnarray*}
&&\left\vert \int_{\left[ t_{i},t_{i+1}\right] \times \left[ \tilde{t}_{j},%
\tilde{t}_{j+1}\right] }f%
\begin{pmatrix}
t_{i},u \\ 
\tilde{t}_{j},v%
\end{pmatrix}%
\,dg\left( u,v\right) \right\vert \\
&\leq &c_{1}\left( p,q\right) V_{p}\left( f,\left[ t_{i},t_{i+1}\right]
\times \left[ \tilde{t}_{j},\tilde{t}_{j+1}\right] \right) V_{q}\left( g,%
\left[ t_{i},t_{i+1}\right] \times \left[ \tilde{t}_{j},\tilde{t}_{j+1}%
\right] \right) \\
&\leq &c_{1}\left( p,q\right) V_{p}\left( f,\left[ s,t\right] ^{2}\right)
\left\vert g\right\vert _{q-var;\left[ t_{i},t_{i+1}\right] \times \left[ 
\tilde{t}_{j},\tilde{t}_{j+1}\right] }
\end{eqnarray*}%
For the second, one has by a Young $1D$-estimate%
\begin{eqnarray*}
\left\vert \int_{\left[ t_{i},t_{i+1}\right] \times \left[ \tilde{t}_{j},%
\tilde{t}_{j+1}\right] }f\left( t_{i},v\right) \,dg\left( u,v\right)
\right\vert &=&\left\vert \int_{\left[ \tilde{t}_{j},\tilde{t}_{j+1}\right]
}f\left( t_{i},v\right) \,d\left( g\left( t_{i+1},v\right) -g\left(
t_{i},v\right) \right) \right\vert \\
&\leq &c_{2}\sup_{u\in \left[ s,t\right] }\left\vert f\left( u,\cdot \right)
\right\vert _{p-var;\left[ s,t\right] }\left\vert g\right\vert _{q-var;\left[
t_{i},t_{i+1}\right] \times \left[ \tilde{t}_{j},\tilde{t}_{j+1}\right] }.
\end{eqnarray*}%
Similarly,%
\begin{equation*}
\left\vert \int_{\left[ t_{i},t_{i+1}\right] \times \left[ \tilde{t}_{j},%
\tilde{t}_{j+1}\right] }f\left( u,\tilde{t}_{j}\right) \,dg\left( u,v\right)
\right\vert \leq c_{2}\sup_{v\in \left[ s,t\right] }\left\vert f\left( \cdot
,v\right) \right\vert _{p-var;\left[ s,t\right] }\left\vert g\right\vert
_{q-var;\left[ t_{i},t_{i+1}\right] \times \left[ \tilde{t}_{j},\tilde{t}%
_{j+1}\right] }.
\end{equation*}%
Finally,%
\begin{equation*}
\left\vert \int_{\left[ t_{i},t_{i+1}\right] \times \left[ \tilde{t}_{j},%
\tilde{t}_{j+1}\right] }f\left( t_{i},\tilde{t}_{j}\right) \,dg\left(
u,v\right) \right\vert =\left\vert f\left( t_{i},\tilde{t}_{j}\right)
\right\vert \left\vert g%
\begin{pmatrix}
t_{i},t_{i+1} \\ 
\tilde{t}_{j},\tilde{t}_{j+1}%
\end{pmatrix}%
\right\vert \leq \left\vert f\right\vert _{\infty ;\left[ s,t\right]
}\left\vert g\right\vert _{q-var;\left[ t_{i},t_{i+1}\right] \times \left[ 
\tilde{t}_{j},\tilde{t}_{j+1}\right] }.
\end{equation*}%
Putting all together, we get%
\begin{eqnarray*}
&&\left\vert \Phi 
\begin{pmatrix}
t_{i},t_{i+1} \\ 
\tilde{t}_{j},\tilde{t}_{j+1}%
\end{pmatrix}%
\right\vert ^{q} \\
&\leq &c_{3}\left( V_{p}\left( f,\left[ s,t\right] ^{2}\right) +\sup_{u\in %
\left[ s,t\right] }\left\vert f\left( u,\cdot \right) \right\vert _{p-var;%
\left[ s,t\right] }+\sup_{v\in \left[ s,t\right] }\left\vert f\left( \cdot
,v\right) \right\vert _{p-var;\left[ s,t\right] }+\left\vert f\right\vert
_{\infty ;\left[ s,t\right] }\right) ^{q} \\
&&\times \left\vert g\right\vert _{q-var;\left[ t_{i},t_{i+1}\right] \times %
\left[ \tilde{t}_{j},\tilde{t}_{j+1}\right] }^{q}.
\end{eqnarray*}%
Take a partition $D\subset \left[ s,t\right] $ and $u\in \left[ s,t\right] $%
. Then%
\begin{equation*}
\sum_{t_{i}\in D}\left\vert f\left( u,t_{i+1}\right) -f\left( u,t_{i}\right)
\right\vert ^{p}=\sum_{t_{i}\in D}\left\vert f%
\begin{pmatrix}
s,u \\ 
t_{i},t_{i+1}%
\end{pmatrix}%
\right\vert ^{p}\leq V_{p}\left( f,\left[ s,t\right] ^{2}\right) ^{p}
\end{equation*}%
and hence%
\begin{equation*}
\sup_{u\in \left[ s,t\right] }\left\vert f\left( u,\cdot \right) \right\vert
_{p-var;\left[ s,t\right] }\leq V_{p}\left( f,\left[ s,t\right] ^{2}\right) .
\end{equation*}%
The same way one obtains%
\begin{equation*}
\sup_{v\in \left[ s,t\right] }\left\vert f\left( \cdot ,v\right) \right\vert
_{p-var;\left[ s,t\right] }\leq V_{p}\left( f,\left[ s,t\right] ^{2}\right) .
\end{equation*}%
Finally, for $u,v\in \left[ s,t\right] $,%
\begin{equation*}
\left\vert f\left( u,v\right) \right\vert =\left\vert f%
\begin{pmatrix}
s,u \\ 
s,v%
\end{pmatrix}%
\right\vert \leq V_{p}\left( f,\left[ s,t\right] ^{2}\right)
\end{equation*}%
and therefore $\left\vert f\right\vert _{\infty ;\left[ s,t\right] }\leq
V_{p}\left( f,\left[ s,t\right] ^{2}\right) $. Putting everything together,
we end up with%
\begin{equation*}
\left\vert \Phi 
\begin{pmatrix}
t_{i},t_{i+1} \\ 
\tilde{t}_{j},\tilde{t}_{j+1}%
\end{pmatrix}%
\right\vert ^{q}\leq c_{4}V_{p}\left( f,\left[ s,t\right] ^{2}\right)
^{q}\left\vert g\right\vert _{q-var;\left[ t_{i},t_{i+1}\right] \times \left[
\tilde{t}_{j},\tilde{t}_{j+1}\right] }^{q}.
\end{equation*}%
Hence for every partition $D,\tilde{D}\subset \left[ s,t\right] $ one gets,
using superadditivity of $\left\vert g\right\vert _{q-var}^{q}$,%
\begin{eqnarray*}
\sum_{t_{i}\in D,\tilde{t}_{j}\in \tilde{D}}\left\vert \Phi 
\begin{pmatrix}
t_{i},t_{i+1} \\ 
\tilde{t}_{j},\tilde{t}_{j+1}%
\end{pmatrix}%
\right\vert ^{q} &\leq &c_{4}V_{p}\left( f,\left[ s,t\right] ^{2}\right)
^{q}\sum_{t_{i}\in D,\tilde{t}_{j}\in \tilde{D}}\left\vert g\right\vert
_{q-var;\left[ t_{i},t_{i+1}\right] \times \left[ \tilde{t}_{j},\tilde{t}%
_{j+1}\right] }^{q} \\
&\leq &c_{4}V_{p}\left( f,\left[ s,t\right] ^{2}\right) ^{q}\left\vert
g\right\vert _{q-var;\left[ s,t\right] ^{2}}^{q}.
\end{eqnarray*}%
Passing to the supremum over all partitions shows the assertion.
\end{enumerate}
\end{proof}

This lemma allows us to define iterated $2D$-integrals. Let $f,g_{1},\ldots
,g_{n}\colon \left[ 0,1\right] ^{2}\rightarrow \mathbb{R}$. An iterated $2D$%
-integral is given by $\int_{\Delta _{s,t}^{1}\times \Delta _{s^{\prime
},t^{\prime }}^{1}}f\,dg_{1}=\int_{\left[ s,t\right] \times \left[ s^{\prime
},t^{\prime }\right] }f\left( u,v\right) \,dg_{1}\left( u,v\right) $ for $%
n=1 $ and recursively defined by 
\begin{equation*}
\int_{\Delta _{s,t}^{n}\times \Delta _{s^{\prime },t^{\prime
}}^{n}}f\,dg_{1}\ldots dg_{n}:=\int_{\left[ s,t\right] \times \left[
s^{\prime },t^{\prime }\right] }\left( \int_{\Delta _{s,u}^{n-1}\times
\Delta _{s^{\prime },v}^{n-1}}f\,dg_{1}\ldots dg_{n-1}\right) \,dg_{n}\left(
u,v\right)
\end{equation*}%
for $n\geq 2$.

\begin{proposition}
\label{prop_rho_var_iter_2D}Let $f,g_{1},g_{2},\ldots \colon \left[ 0,1%
\right] ^{2}\rightarrow \mathbb{R}$ and $p,q_{1},q_{2},\ldots $ be real
numbers such that $p^{-1}+q_{1}^{-1}>1$ and $q_{i}^{-1}+q_{i+1}^{-1}>1$ for
every $i\geq 1$. Assume that $f$ has finite $p$-variation and $g_{i}$ has
finite $q_{i}$-variation for $i=1,2,\ldots $ and that for $\left( s,t\right)
\in \Delta $ we have $f\left( s,\cdot \right) =f\left( \cdot ,s\right) =0$.
Then for every $n\in \mathbb{N}$ there is a constant $C=C\left(
p,q_{1},\ldots ,q_{n}\right) $ such that%
\begin{equation*}
\left\vert \int_{\Delta _{s,t}^{n}\times \Delta _{s,t}^{n}}f\,dg_{1}\ldots
dg_{n}\right\vert \leq CV_{p}\left( f,\left[ s,t\right] ^{2}\right)
V_{q_{1}}\left( g_{1},\left[ s,t\right] ^{2}\right) \ldots V_{q_{n}}\left(
g_{n},\left[ s,t\right] ^{2}\right) .
\end{equation*}
\end{proposition}

\begin{proof}
Define $\Phi ^{\left( n\right) }\left( u,v\right) =\int_{\Delta
_{s,u}^{n}\times \Delta _{s,v}^{n}}f\,dg_{1}\ldots dg_{n}$. We will show a
stronger result; namely that for every $n\in \mathbb{N}$ and $q_{n}^{\prime
}>q_{n}$ there is a constant $C=C\left( p,q_{1},\ldots ,q_{n},q_{n}^{\prime
}\right) $ such that%
\begin{equation*}
V_{q_{n}^{\prime }}\left( \Phi ^{\left( n\right) },\left[ s,t\right]
^{2}\right) \leq CV_{p}\left( f,\left[ s,t\right] ^{2}\right)
V_{q_{1}}\left( g_{1},\left[ s,t\right] ^{2}\right) \ldots V_{q_{n}}\left(
g_{n},\left[ s,t\right] ^{2}\right) .
\end{equation*}%
To do so, let $\tilde{q}_{1},\tilde{q}_{2},\ldots $be a sequence of real
numbers such that $\tilde{q}_{j}>q_{j}$ and $\frac{1}{\tilde{q}_{j-1}}+\frac{%
1}{\tilde{q}_{j}}>1$ for every $j=1,2,\ldots $ where we set $\tilde{q}_{0}=p$%
. We make an induction over $n$. For $n=1$, we have $\tilde{q}_{1}>q_{1}$
and $\frac{1}{p}+\frac{1}{\tilde{q}_{1}}>1$, hence from Theorem \ref%
{theorem_comp_contr_p_var} we know that $g_{1}$ has finite controlled $%
\tilde{q}_{1}$-variation and Lemma \ref{lemma_kernel_iter_2D} gives us%
\begin{equation*}
V_{\tilde{q}_{1}}\left( \Phi ^{\left( 1\right) };\left[ s,t\right]
^{2}\right) \leq c_{1}V_{p}\left( f;\left[ s,t\right] ^{2}\right) \left\vert
g_{1}\right\vert _{\tilde{q}_{1};\left[ s,t\right] ^{2}}\leq
c_{2}V_{p}\left( f;\left[ s,t\right] ^{2}\right) V_{q_{1}}\left( g_{1};\left[
s,t\right] ^{2}\right) .
\end{equation*}%
W.l.o.g, we may assume that $q_{1}^{\prime }>\tilde{q}_{1}>q_{1}$, otherwise
we choose $\tilde{q}_{1}$ smaller in the beginning. From $V_{q_{1}^{\prime
}}\left( \Phi ^{\left( 1\right) };\left[ s,t\right] ^{2}\right) \leq V_{%
\tilde{q}_{1}}\left( \Phi ^{\left( 1\right) };\left[ s,t\right] ^{2}\right) $
the assertion follows for $n=1$. Now take $n\in \mathbb{N}$. Note that%
\begin{equation*}
\Phi ^{\left( n\right) }\left( u,v\right) =\int_{\left[ s,u\right] \times %
\left[ s,v\right] }\Phi ^{\left( n-1\right) }\,dg_{n}
\end{equation*}%
and clearly $\Phi ^{\left( n-1\right) }\left( s,\cdot \right) =\Phi ^{\left(
n-1\right) }\left( \cdot ,s\right) =0$. We can use Lemma \ref%
{lemma_kernel_iter_2D} again to see that%
\begin{eqnarray*}
V_{\tilde{q}_{n}}\left( \Phi ^{\left( n\right) },\left[ s,t\right]
^{2}\right) &\leq &c_{3}V_{\tilde{q}_{n-1}}\left( \Phi ^{\left( n-1\right) };%
\left[ s,t\right] ^{2}\right) \left\vert g_{n}\right\vert _{\tilde{q}%
_{n}-var;\left[ s,t\right] ^{2}} \\
&\leq &c_{4}V_{\tilde{q}_{n-1}}\left( \Phi ^{\left( n-1\right) };\left[ s,t%
\right] ^{2}\right) V_{q_{n}}\left( g_{n};\left[ s,t\right] ^{2}\right) .
\end{eqnarray*}%
Using our induction hypothesis shows the result for $\tilde{q}_{n}$. By
choosing $\tilde{q}_{n}$ smaller in the beginning if necessary, we may
assume that $q_{n}^{\prime }>\tilde{q}_{n}$ and the assertion follows.
\end{proof}

\section{The main estimates\label{section_main_estimates}}

In the following section, $\left( X,Y\right) =\left( X^{1},Y^{1},\ldots
,X^{d},Y^{d}\right) $ will always denote a centred continuous Gaussian
process where $\left( X^{i},Y^{i}\right) $ and $\left( X^{j},Y^{j}\right) $
are independent for $i\neq j$. We will also assume that the $\rho $%
-variation of $R_{\left( X,Y\right) }$ is finite for a $\rho <2$ and
controlled by a symmetric $2D$-control $\omega $ (this in particular implies
that the $\rho $-variation of $R_{X},R_{Y}$ and $R_{X-Y}$ is controlled by $%
\omega $, see \cite[Section 15.3.2]{FV10}). Let $\gamma >\rho $ such that $%
\frac{1}{\rho }+\frac{1}{\gamma }>1$. The aim of this section is to show
that for every $n\in \mathbb{N}$ there are constants $C\left( n\right) $
such that$\footnote{%
We prefer to write it in this notation instead of writing $\omega \left( %
\left[ s,t\right] ^{2}\right) ^{\frac{1}{2\gamma }+\frac{n-1}{2\rho }}$ to
emphasize the different roles of the two terms. The first term will play no
particular role and just comes from interpolation whereas the second one
will be crucial when doing the induction step from lower to higher levels in
Proposition \ref{prop_key_estimate_higher_levels}.}$%
\begin{equation}
\left\vert \mathbf{X}_{s,t}^{n}-\mathbf{Y}_{s,t}^{n}\right\vert
_{L^{2}\left( \left( \mathbb{R}^{d}\right) ^{\otimes n}\right) }\leq C\left(
n\right) \epsilon \omega \left( \left[ s,t\right] ^{2}\right) ^{\frac{1}{%
2\gamma }}\omega \left( \left[ s,t\right] ^{2}\right) ^{\frac{n-1}{2\rho }%
}\quad \text{for every }s<t  \label{eqn_key_estimate}
\end{equation}%
where $\epsilon ^{2}=V_{\infty }\left( R_{X-Y},\left[ s,t\right] ^{2}\right)
^{1-\rho /\gamma }$(see Definition \ref{def_V_infty} below for the exact
definition of $V_{\infty }$). Equivalently, we might show $\left( \ref%
{eqn_key_estimate}\right) $ coordinate-wise, i.e. proving that the same
estimate holds for $\left\vert \mathbf{X}^{w}-\mathbf{Y}^{w}\right\vert
_{L^{2}\left( \mathbb{R}\right) }$ for every word $w$ formed by the alphabet 
$A=\left\{ 1,\ldots ,d\right\} $. In some special cases, i.e. if a word $w$
has a very simple structure, we can do this directly using multidimensional
Young integration. This is done in Subsection \ref{subsection_special_cases}%
. Subsection \ref{subsection_lower_levels} shows $\left( \ref%
{eqn_key_estimate}\right) $ for $n=1,2,3,4$ coordinate-wise, using the
shuffle algebra structure for iterated integrals and multidimensional Young
integration. In Subsection \ref{subsection_higher_levels}, we show $\left( %
\ref{eqn_key_estimate}\right) $ coordinate-free for all $n>4$, using an
induction argument very similar to the one Lyon's used for proving the
Extension Theorem (cf. \cite{L98}).

We start with giving a $2$-dimensional analogue for the one-dimensional
interpolation inequality.

\begin{definition}
\label{def_V_infty}If $f\colon \left[ 0,1\right] ^{2}\rightarrow B$ is a
continuous function in a Banach space and $\left( s,t\right) \times \left(
u,v\right) \in \Delta \times \Delta $ we set%
\begin{equation*}
V_{\infty }\left( f,\left[ s,t\right] \times \left[ u,v\right] \right)
=\sup_{A\subset \left[ s,t\right] \times \left[ u,v\right] }\left\vert
f\left( A\right) \right\vert .
\end{equation*}
\end{definition}

\begin{lemma}
\label{lemma_2D_interpolation}For $\gamma >\rho \geq 1$ we have the
interpolation inequality%
\begin{equation*}
V_{\gamma -var}\left( f,\left[ s,t\right] \times \left[ u,v\right] \right)
\leq V_{\infty }\left( f,\left[ s,t\right] \times \left[ u,v\right] \right)
^{1-\rho /\gamma }V_{\rho -var}\left( f,\left[ s,t\right] \times \left[ u,v%
\right] \right) ^{\rho /\gamma }
\end{equation*}%
for all $\left( s,t\right) ,\left( u,v\right) \in \Delta $.
\end{lemma}

\begin{proof}
Exactly as $1D$-interpolation, see \cite[Proposition 5.5]{FV10}.
\end{proof}

\subsection{Some special cases\label{subsection_special_cases}}

If $Z\colon \left[ 0,1\right] \rightarrow \mathbb{R}$ is a process with
smooth sample paths, we will use the notation%
\begin{equation*}
\mathbf{Z}_{s,t}^{\left( n\right) }=\int_{\Delta _{s,t}^{n}}dZ\ldots dZ
\end{equation*}%
for $s<t$.

\begin{lemma}
\label{lemma_rho_var_same_letters} Let $X\colon \left[ 0,1\right]
\rightarrow \mathbb{R}$ be a centred Gaussian process with continuous paths
of finite variation and assume that the $\rho $-variation of the covariance $%
R_{X}$ is controlled by a $2D$-control $\omega $. For fixed $s<t$, define%
\begin{equation*}
f\left( u,v\right) =E\left( \mathbf{X}_{s,u}^{\left( n\right) }\mathbf{X}%
_{s,v}^{\left( n\right) }\right) .
\end{equation*}%
Then there is a constant $C=C\left( \rho ,n\right) $ such that%
\begin{equation*}
V_{\rho }\left( f,\left[ s,t\right] ^{2}\right) \leq C\omega \left( \left[
s,t\right] ^{2}\right) ^{\frac{n}{\rho }}.
\end{equation*}
\end{lemma}

\begin{proof}
Let $t_{i}<t_{i+1}$, $\tilde{t}_{j}<\tilde{t}_{j+1}$. Then%
\begin{equation*}
f%
\begin{pmatrix}
t_{i},t_{i+1} \\ 
\tilde{t}_{j},\tilde{t}_{j+1}%
\end{pmatrix}%
=E\left( \left( \mathbf{X}_{s,t_{i+1}}^{\left( n\right) }-\mathbf{X}%
_{s,t_{i}}^{\left( n\right) }\right) \left( \mathbf{X}_{s,\tilde{t}%
_{j+1}}^{\left( n\right) }-\mathbf{X}_{s,\tilde{t}_{j}}^{\left( n\right)
}\right) \right) .
\end{equation*}%
We know that $\mathbf{X}^{\left( n\right) }=\frac{\left( X\right) ^{n}}{n!}$%
. From the identity%
\begin{equation*}
b^{n}-a^{n}=\left( b-a\right) \left( a^{n-1}+a^{n-2}b+\ldots +\ldots
ab^{n-2}+b^{n-1}\right)
\end{equation*}%
we deduce that%
\begin{equation*}
f%
\begin{pmatrix}
t_{i},t_{i+1} \\ 
\tilde{t}_{j},\tilde{t}_{j+1}%
\end{pmatrix}%
=\frac{1}{\left( n!\right) ^{2}}\sum_{k,l=0}^{n-1}E\left(
X_{t_{i},t_{i+1}}X_{\tilde{t}_{j},\tilde{t}_{j+1}}\left(
X_{s,t_{i+1}}\right) ^{n-1-k}\left( X_{s,t_{i}}\right) ^{k}\left( X_{s,%
\tilde{t}_{j+1}}\right) ^{n-1-l}\left( X_{s,\tilde{t}_{j}}\right)
^{l}\right) .
\end{equation*}%
We want to apply Wick's formula now (cf. \cite[Theorem 1.28]{J97}). If $Z,%
\tilde{Z}\in \left\{ X_{s,t_{i+1}},X_{s,t_{i}},X_{s,\tilde{t}_{j+1}},X_{s,%
\tilde{t}_{j}}\right\} $ we know that%
\begin{eqnarray*}
\left\vert E\left( X_{t_{i},t_{i+1}}Z\right) \right\vert ^{\rho } &\leq
&\omega \left( \left[ t_{i},t_{i+1}\right] \times \left[ s,t\right] \right)
\\
\left\vert E\left( X_{t_{i},t_{i+1}}X_{\tilde{t}_{j},\tilde{t}_{j+1}}\right)
\right\vert ^{\rho } &\leq &\omega \left( \left[ t_{i},t_{i+1}\right] \times %
\left[ \tilde{t}_{j},\tilde{t}_{j+1}\right] \right) \\
\left\vert E\left( Z\tilde{Z}\right) \right\vert ^{\rho } &\leq &\omega
\left( \left[ s,t\right] ^{2}\right)
\end{eqnarray*}%
and the same holds for $X_{\tilde{t}_{j},\tilde{t}_{j+1}}$. Now take two
partitions $D,\tilde{D}\in \left[ 0,1\right] $. Then, by Wick's formula and
the estimates above,%
\begin{eqnarray*}
\sum_{t_{i}\in D,\tilde{t}_{j}\in \tilde{D}}\left\vert f%
\begin{pmatrix}
t_{i},t_{i+1} \\ 
\tilde{t}_{j},\tilde{t}_{j+1}%
\end{pmatrix}%
\right\vert ^{\rho } &\leq &c_{1}\left( \rho ,n\right) \omega \left( \left[
s,t\right] ^{2}\right) ^{n-2}\sum_{t_{i}\in D,\tilde{t}_{j}\in \tilde{D}%
}\omega \left( \left[ t_{i},t_{i+1}\right] \times \left[ s,t\right] \right)
\omega \left( \left[ \tilde{t}_{j},\tilde{t}_{j+1}\right] \times \left[ s,t%
\right] \right) \\
&&+c_{2}\left( \rho ,n\right) \omega \left( \left[ s,t\right] ^{2}\right)
^{n-1}\sum_{t_{i}\in D,\tilde{t}_{j}\in \tilde{D}}\omega \left( \left[
t_{i},t_{i+1}\right] \times \left[ \tilde{t}_{j},\tilde{t}_{j+1}\right]
\right) \\
&\leq &c_{3}\omega \left( \left[ s,t\right] ^{2}\right) ^{n}.
\end{eqnarray*}
\end{proof}

\begin{lemma}
\label{lemma_diff_allthesame}Let $\left( X,Y\right) $ be a centred Gaussian
process in $\mathbb{R}^{2}$ with continuous paths of finite variation.
Assume that the $\rho $-variation of $R_{\left( X,Y\right) }$ is controlled
by a $2D$-control $\omega $ for $\rho <2$ and take $\gamma >\rho $. Then for
every $n\in \mathbb{N}$ there is a constant $C=C\left( n\right) $ such that%
\begin{equation*}
\left\vert \mathbf{X}_{s,t}^{\left( n\right) }-\mathbf{Y}_{s,t}^{\left(
n\right) }\right\vert _{L^{2}}\leq C\left( n\right) \epsilon \omega \left( %
\left[ s,t\right] ^{2}\right) ^{\frac{1}{2\gamma }}\omega \left( \left[ s,t%
\right] ^{2}\right) ^{\frac{n-1}{2\rho }}
\end{equation*}%
for any $s<t$ where $\epsilon ^{2}=V_{\infty }\left( R_{X-Y},\left[ s,t%
\right] ^{2}\right) ^{1-\rho /\gamma }$.
\end{lemma}

\begin{proof}
By induction. For $n=1$ we simply have from Lemma \ref%
{lemma_2D_interpolation}%
\begin{eqnarray*}
\left\vert X_{s,t}-Y_{s,t}\right\vert _{L^{2}}^{2} &=&E\left[ \left(
X_{s,t}-Y_{s,t}\right) \left( X_{s,t}-Y_{s,t}\right) \right] \leq V_{\gamma
-var}\left( R_{X-Y},\left[ s,t\right] ^{2}\right) \\
&\leq &\epsilon ^{2}V_{\rho -var}\left( R_{X-Y},\left[ s,t\right]
^{2}\right) ^{\rho /\gamma }\leq \epsilon ^{2}\omega \left( \left[ s,t\right]
^{2}\right) ^{\frac{1}{\gamma }}
\end{eqnarray*}%
For $n\in \mathbb{N}$ we use the identity%
\begin{equation*}
\mathbf{X}_{s,t}^{\left( n\right) }-\mathbf{Y}_{s,t}^{\left( n\right) }=%
\frac{1}{n}\left( X_{s,t}\mathbf{X}_{s,t}^{\left( n-1\right) }-Y_{s,t}%
\mathbf{Y}_{s,t}^{\left( n-1\right) }\right)
\end{equation*}%
and hence%
\begin{eqnarray*}
\left\vert \mathbf{X}_{s,t}^{\left( n\right) }-\mathbf{Y}_{s,t}^{\left(
n\right) }\right\vert _{L^{2}} &\leq &c_{1}\left( \left\vert
X_{s,t}-Y_{s,t}\right\vert _{L^{2}}\left\vert \mathbf{X}_{s,t}^{\left(
n-1\right) }\right\vert _{L^{2}}+\left\vert \mathbf{X}_{s,t}^{\left(
n-1\right) }-\mathbf{Y}_{s,t}^{\left( n-1\right) }\right\vert
_{L^{2}}\left\vert Y_{s,t}\right\vert _{L^{2}}\right) \\
&\leq &c_{2}\epsilon \omega \left( \left[ s,t\right] ^{2}\right) ^{\frac{1}{%
2\gamma }}\omega \left( \left[ s,t\right] ^{2}\right) ^{\frac{n-1}{2\rho }}.
\end{eqnarray*}
\end{proof}

Assume that $\left( Z^{1},Z^{2}\right) $ is a centred, continuous Gaussian
process in $\mathbb{R}^{2}$ with smooth sample paths and that both
components are independent. Then (at least formally, cf. \cite{FV10AIHP}),%
\begin{eqnarray}
\left\vert \int_{0}^{1}Z_{0,u}^{1}\,dZ_{u}^{2}\right\vert _{L^{2}}^{2} &=&E 
\left[ \left( \int_{0}^{1}Z_{0,u}^{1}\,dZ_{u}^{2}\right) ^{2}\right] =E\left[
\int_{\left[ 0,1\right] ^{2}}Z_{0,u}^{1}Z_{0,v}^{1}\,dZ^{2}\,dZ_{v}^{2}%
\right]  \label{eqn_L2_norm_into_integral1} \\
&=&\int_{\left[ 0,1\right] ^{2}}E\left[ Z_{0,u}^{1}Z_{0,v}^{1}\right] \,dE%
\left[ Z_{u}^{2}Z_{v}^{2}\right] =\int_{\left[ 0,1\right] ^{2}}R_{Z^{1}}%
\begin{pmatrix}
0 & \cdot \\ 
0 & \cdot%
\end{pmatrix}%
\,dR_{Z^{2}}  \label{eqn_L2_norm_into_integral2}
\end{eqnarray}%
where the integrals in the second row are $2D$ Young-integrals (to make this
rigorous, one uses that the integrals are a.s. limits of Riemann sums and
that a.s. convergence implies convergence in $L^{1}$ in the (inhomogeneous)
Wiener chaos). These kinds of computations together with our estimates for $%
2D$ Young-integrals will be heavily used from now on.

\begin{lemma}
\label{lemma_diff_alldifferent}Let $\left( X,Y\right) =\left(
X^{1},Y^{1},\ldots ,X^{d},Y^{d}\right) $ be a centred Gaussian process with
continuous paths of finite variation where $\left( X^{i},Y^{i}\right) \ $and 
$\left( X^{j},Y^{j}\right) $ are independent for $i\neq j$. Assume that the $%
\rho $-variation of $R_{\left( X,Y\right) }$ is controlled by a $2D$-control 
$\omega $ for $\rho <2$. Let $w$ be a word of the form $w=$ $i_{1}\cdots
i_{n}$ where $i_{1},\ldots ,i_{n}\in \left\{ 1,\ldots ,d\right\} $ are all
distinct. Take $\gamma >\rho $ such that $\frac{1}{\rho }+\frac{1}{\gamma }%
>1 $. Then there is a constant $C=C\left( \rho ,\gamma ,n\right) $ such that%
\begin{equation*}
\left\vert \mathbf{X}_{s,t}^{w}-\mathbf{Y}_{s,t}^{w}\right\vert _{L^{2}}\leq
C\left( n\right) \epsilon \omega \left( \left[ s,t\right] ^{2}\right) ^{%
\frac{1}{2\gamma }}\omega \left( \left[ s,t\right] ^{2}\right) ^{\frac{n-1}{%
2\rho }}
\end{equation*}%
for any $s<t$ where $\epsilon ^{2}=V_{\infty }\left( R_{X-Y},\left[ s,t%
\right] ^{2}\right) ^{1-\rho /\gamma }$.
\end{lemma}

\begin{proof}
By the triangle inequality,%
\begin{eqnarray*}
\left\vert \mathbf{X}_{s,t}^{w}-\mathbf{Y}_{s,t}^{w}\right\vert _{L^{2}}
&=&\left\vert \int_{\Delta _{s,t}^{n}}\,dX^{i_{1}}\ldots
dX^{i_{n}}-\int_{\Delta _{s,t}^{n}}\,dY^{i_{1}}\ldots dY^{i_{n}}\right\vert
_{L^{2}} \\
&\leq &\sum_{k=1}^{n}\left\vert \int_{\Delta _{s,t}^{n}}\,dY^{i_{1}}\ldots
dY^{i_{k-1}}\,d\left( X^{i_{k}}-Y^{i_{k}}\right) \,dX^{i_{k+1}}\ldots
dX^{i_{n}}\right\vert _{L^{2}}.
\end{eqnarray*}%
From independence, Proposition \ref{prop_rho_var_iter_2D} and Lemma \ref%
{lemma_2D_interpolation} 
\begin{eqnarray*}
&&\left\vert \int_{\Delta _{s,t}^{n}}\,dY^{i_{1}}\ldots
dY^{i_{k-1}}\,d\left( X^{i_{k}}-Y^{i_{k}}\right) \,dX^{i_{k+1}}\ldots
dX^{i_{n}}\right\vert _{L^{2}}^{2} \\
&=&\int_{\Delta _{s,t}^{n}\times \Delta _{s,t}^{n}}\,dR_{Y^{i_{1}}}\ldots
dR_{Y^{i_{k-1}}}\,dR_{X^{i_{k}}-Y^{i_{k}}}\,dR_{X^{i_{k+1}}}\ldots
dR_{X^{i_{n}}} \\
&\leq &c_{1}V_{\rho }\left( R_{Y^{i_{1}}},\left[ s,t\right] ^{2}\right)
\ldots V_{\rho }\left( R_{Y^{i_{k-1}}},\left[ s,t\right] ^{2}\right)
V_{\gamma }\left( R_{X^{i_{k}}-Y^{i_{k}}},\left[ s,t\right] ^{2}\right) \\
&&\times V_{\rho }\left( R_{X^{i_{k+1}}},\left[ s,t\right] ^{2}\right)
\ldots V_{\rho }\left( R_{X^{i_{n}}},\left[ s,t\right] ^{2}\right) \\
&\leq &c_{1}V_{\gamma }\left( R_{X-Y},\left[ s,t\right] ^{2}\right) \omega
\left( \left[ s,t\right] ^{2}\right) ^{\frac{n-1}{\rho }}\leq c_{1}\epsilon
^{2}\omega \left( \left[ s,t\right] ^{2}\right) ^{\frac{1}{\gamma }}\omega
\left( \left[ s,t\right] ^{2}\right) ^{\frac{n-1}{\rho }}.
\end{eqnarray*}%
The first inequality above is an immediate generalization of the
calculations made in $\left( \ref{eqn_L2_norm_into_integral1}\right) $ and $%
\left( \ref{eqn_L2_norm_into_integral2}\right) $. Note that the respective
random terms are not only pairwise but mutually independent here since we
are dealing with a Gaussian process $\left( X,Y\right) $. Interchanging the
limits is allowed since convergence in probability implies convergence in $%
L^{p}$, any $p>0$, in the Wiener chaos.
\end{proof}

\subsection{Lower levels\label{subsection_lower_levels}}

\subsubsection{$n=1,2$}

\begin{proposition}
\label{prop_main_estimates_n1_n2}Let $\left( X,Y\right) $, $\omega $, $\rho $
and $\gamma $ as in Lemma \ref{lemma_diff_alldifferent}. Then there are
constants $C\left( 1\right) ,C\left( 2\right) $ which depend on $\rho $ and $%
\gamma $ such that%
\begin{equation*}
\left\vert \mathbf{X}_{s,t}^{n}-\mathbf{Y}_{s,t}^{n}\right\vert _{L^{2}}\leq
C\left( n\right) \epsilon \omega \left( \left[ s,t\right] ^{2}\right) ^{%
\frac{1}{2\gamma }}\omega \left( \left[ s,t\right] ^{2}\right) ^{\frac{n-1}{%
2\rho }}
\end{equation*}%
holds for $n=1,2$ and every $\left( s,t\right) \in \Delta $ where $\epsilon
^{2}=V_{\infty }\left( R_{X-Y},\left[ s,t\right] ^{2}\right) ^{1-\rho
/\gamma }$.
\end{proposition}

\begin{proof}
The coordinate-wise estimates are just special cases of Lemma \ref%
{lemma_diff_allthesame} and Lemma \ref{lemma_diff_alldifferent}.
\end{proof}

\subsubsection{$n=3$}

\begin{proposition}
\label{prop_main_estimates_n3}Let $\left( X,Y\right) $, $\omega $, $\rho $
and $\gamma $ as in Lemma \ref{lemma_diff_alldifferent}. Then there is a
constant $C\left( 3\right) $ which depends on $\rho $ and $\gamma $ such that%
\begin{equation*}
\left\vert \mathbf{X}_{s,t}^{3}-\mathbf{Y}_{s,t}^{3}\right\vert _{L^{2}}\leq
C\left( 3\right) \epsilon \omega \left( \left[ s,t\right] ^{2}\right) ^{%
\frac{1}{2\gamma }}\omega \left( \left[ s,t\right] ^{2}\right) ^{\frac{2}{%
2\rho }}
\end{equation*}%
holds for every $\left( s,t\right) \in \Delta $ where $\epsilon
^{2}=V_{\infty }\left( R_{X-Y},\left[ s,t\right] ^{2}\right) ^{1-\rho
/\gamma }$.
\end{proposition}

\begin{proof}
We have to show the estimate for $\mathbf{X}^{i,j,k}-\mathbf{Y}^{i,j,k}$
where $i,j,k\in \left\{ 1,\ldots ,d\right\} $. From Proposition \ref%
{proposition_key_shuffle} and \ref{prop_generating_sets} it follows that it
is enough to show the estimate for $\mathbf{X}^{w}-\mathbf{Y}^{w}$ where%
\begin{equation*}
w\in \left\{ iii,ijk,iij:i,j,k\in \left\{ 1,\ldots ,d\right\} ~\text{distinct%
}\right\} \text{.}
\end{equation*}%
The cases $w=iii$ and $w=ijk$ are special cases of Lemma \ref%
{lemma_diff_allthesame} and Lemma \ref{lemma_diff_alldifferent}. The rest of
this section is devoted to show the estimate for $w=iij$.
\end{proof}

\begin{lemma}
\label{lemma_half_estimates} Let $\left( X,Y\right) \colon \left[ 0,1\right]
\rightarrow \mathbb{R}^{2}$ be a centred Gaussian process and consider%
\begin{equation*}
f\left( u,v\right) =E\left( \left( X_{u}-Y_{u}\right) X_{v}\right) .
\end{equation*}%
Assume that the $\rho $-variation of $R_{\left( X,Y\right) }$ is controlled
by a $2D$-control $\omega $ where $\rho \geq 1$. Let $s<t$ and consider a
rectangle $\left[ \sigma ,\tau \right] \times \left[ \sigma ^{\prime },\tau
^{\prime }\right] \subset \left[ s,t\right] ^{2}$. Let $\gamma >\rho $. Then%
\begin{equation*}
V_{\gamma -var}\left( f,\left[ \sigma ,\tau \right] \times \left[ \sigma
^{\prime },\tau ^{\prime }\right] \right) \leq \epsilon \omega \left( \left[
s,t\right] ^{2}\right) ^{1/2\left( 1/\rho -1/\gamma \right) }\omega \left( %
\left[ \sigma ,\tau \right] \times \left[ \sigma ^{\prime },\tau ^{\prime }%
\right] \right) ^{1/\gamma }
\end{equation*}%
where $\epsilon ^{2}=V_{\infty }\left( R_{X-Y},\left[ s,t\right] ^{2}\right)
^{1-\rho /\gamma }$.
\end{lemma}

\begin{proof}
Let $u<v$ and $u^{\prime }<v^{\prime }\in \left[ s,t\right] $. Then%
\begin{eqnarray*}
\left\vert E\left( \left( X_{u,v}-Y_{u,v}\right) X_{u^{\prime },v^{\prime
}}\right) \right\vert &\leq &\left\vert X_{u,v}-Y_{u,v}\right\vert
_{L^{2}}\left\vert X_{u^{\prime },v^{\prime }}\right\vert _{L^{2}} \\
&\leq &V_{\infty }\left( R_{X-Y},\left[ s,t\right] ^{2}\right) ^{1/2}V_{\rho
-var}\left( R_{\left( X,Y\right) },\left[ s,t\right] ^{2}\right) ^{1/2}
\end{eqnarray*}

and hence%
\begin{equation*}
\sup_{u<v,u^{\prime }<v^{\prime }}\left\vert E\left( \left(
X_{u,v}-Y_{u,v}\right) X_{u^{\prime },v^{\prime }}\right) \right\vert \leq
V_{\infty }\left( R_{X-Y},\left[ s,t\right] ^{2}\right) ^{1/2}\omega \left( %
\left[ s,t\right] ^{2}\right) ^{\frac{1}{2\rho }}.
\end{equation*}%
Now take a partition $D$ of $\left[ \sigma ,\tau \right] $ and a partition $%
\tilde{D}$ of $\left[ \sigma ^{\prime },\tau ^{\prime }\right] $. Then%
\begin{eqnarray*}
&&\sum_{t_{i}\in D,\tilde{t}_{j}\in \tilde{D}}\left\vert E\left( \left(
X_{t_{i},t_{i+1}}-Y_{t_{i},t_{i+1}}\right) X_{\tilde{t}_{j},\tilde{t}%
_{j+1}}\right) \right\vert ^{\gamma } \\
&\leq &\sup_{u<v,u^{\prime }<v^{\prime }}\left\vert E\left( \left(
X_{u,v}-Y_{u,v}\right) X_{u^{\prime },v^{\prime }}\right) \right\vert
^{\gamma -\rho }\sum_{t_{i}\in D,\tilde{t}_{j}\in \tilde{D}}\left\vert
E\left( \left( X_{t_{i},t_{i+1}}-Y_{t_{i},t_{i+1}}\right) X_{\tilde{t}_{j},%
\tilde{t}_{j+1}}\right) \right\vert ^{\rho } \\
&\leq &V_{\infty }\left( R_{X-Y},\left[ s,t\right] ^{2}\right) ^{1/2\left(
\gamma -\rho \right) }\omega \left( \left[ s,t\right] ^{2}\right)
^{1/2\left( \gamma /\rho -1\right) }\omega \left( \left[ \sigma ,\tau \right]
\times \left[ \sigma ^{\prime },\tau ^{\prime }\right] \right)
\end{eqnarray*}%
and taking the supremum over all partitions shows the result.
\end{proof}

\begin{lemma}
\label{lemma_rhovar_diff_n2}Let $\left( X,Y\right) \colon \left[ 0,1\right]
\rightarrow \mathbb{R}^{2}$ be a centred Gaussian process with continuous
paths of finite variation. Assume that the $\rho $-variation of $R_{\left(
X,Y\right) }$ is controlled by a $2D$-control $\omega $ where $\rho \geq 1$.
Consider the function%
\begin{equation*}
g\left( u,v\right) =E\left[ \left( \mathbf{X}_{s,u}^{\left( 2\right) }-%
\mathbf{Y}_{s,u}^{\left( 2\right) }\right) \left( \mathbf{X}_{s,v}^{\left(
2\right) }-\mathbf{Y}_{s,v}^{\left( 2\right) }\right) \right] .
\end{equation*}%
Then for every $\gamma >\rho $ there is a constant $C=C\left( \rho ,\gamma
\right) $ such that%
\begin{equation*}
V_{\gamma -var}\left( g,\left[ s,t\right] ^{2}\right) \leq C\epsilon
^{2}\omega \left( \left[ s,t\right] ^{2}\right) ^{1/\gamma +1/\rho }
\end{equation*}%
holds for every $\left( s,t\right) \in \Delta $ where $\epsilon
^{2}=V_{\infty }\left( R_{X-Y},\left[ s,t\right] ^{2}\right) ^{1-\rho
/\gamma }$.
\end{lemma}

\begin{proof}
Let $u<v$ and $u^{\prime }<v^{\prime }$. Then%
\begin{eqnarray*}
g%
\begin{pmatrix}
u,v \\ 
u^{\prime },v^{\prime }%
\end{pmatrix}
&=&E\left[ \left( \left( \mathbf{X}_{s,v}^{\left( 2\right) }-\mathbf{X}%
_{s,u}^{\left( 2\right) }\right) -\left( \mathbf{Y}_{s,v}^{\left( 2\right) }-%
\mathbf{Y}_{s,u}^{\left( 2\right) }\right) \right) \left( \left( \mathbf{X}%
_{s,v^{\prime }}^{\left( 2\right) }-\mathbf{X}_{s,u^{\prime }}^{\left(
2\right) }\right) -\left( \mathbf{Y}_{s,v^{\prime }}^{\left( 2\right) }-%
\mathbf{Y}_{s,u^{\prime }}^{\left( 2\right) }\right) \right) \right] \\
&=&\frac{1}{2^{2}}E\left[ \left( \left( X_{s,v}^{2}-X_{s,u}^{2}\right)
-\left( Y_{s,v}^{2}-Y_{s,u}^{2}\right) \right) \left( \left( X_{s,v^{\prime
}}^{2}-X_{s,u^{\prime }}^{2}\right) -\left( Y_{s,v^{\prime
}}^{2}-Y_{s,u^{\prime }}^{2}\right) \right) \right] .
\end{eqnarray*}%
Now,%
\begin{eqnarray*}
\left( X_{s,v}^{2}-X_{s,u}^{2}\right) -\left( Y_{s,v}^{2}-Y_{s,u}^{2}\right)
&=&X_{u,v}\left( X_{s,u}+X_{s,v}\right) -Y_{u,v}\left( Y_{s,u}+Y_{s,v}\right)
\\
&=&X_{u,v}\left( X_{s,u}-Y_{s,u}\right) +\left( X_{u,v}-Y_{u,v}\right)
Y_{s,u} \\
&&+X_{u,v}\left( X_{s,v}-Y_{s,v}\right) +\left( X_{u,v}-Y_{u,v}\right)
Y_{s,v}.
\end{eqnarray*}%
The same way one gets%
\begin{eqnarray*}
\left( X_{s,v^{\prime }}^{2}-X_{s,u^{\prime }}^{2}\right) -\left(
Y_{s,v^{\prime }}^{2}-Y_{s,u^{\prime }}^{2}\right) &=&X_{u^{\prime
},v^{\prime }}\left( X_{s,u^{\prime }}-Y_{s,u^{\prime }}\right) +\left(
X_{u^{\prime },v^{\prime }}-Y_{u^{\prime },v^{\prime }}\right)
Y_{s,u^{\prime }} \\
&&+X_{u^{\prime },v^{\prime }}\left( X_{s,v^{\prime }}-Y_{s,v^{\prime
}}\right) +\left( X_{u^{\prime },v^{\prime }}-Y_{u^{\prime },v^{\prime
}}\right) Y_{s,v^{\prime }}.
\end{eqnarray*}%
Now we expand the product of both sums and take expectation. For the first
term we obtain, using the Wick formula and Lemma \ref{lemma_half_estimates},%
\begin{eqnarray*}
&&\left\vert E\left( X_{u,v}\left( X_{s,u}-Y_{s,u}\right) X_{u^{\prime
},v^{\prime }}\left( X_{s,u^{\prime }}-Y_{s,u^{\prime }}\right) \right)
\right\vert \\
&\leq &\left\vert E\left( X_{u,v}X_{u^{\prime },v^{\prime }}\right) E\left[
\left( X_{s,u}-Y_{s,u}\right) \left( X_{s,u^{\prime }}-Y_{s,u^{\prime
}}\right) \right] \right\vert \\
&&+\left\vert E\left[ X_{u,v}\left( X_{s,u^{\prime }}-Y_{s,u^{\prime
}}\right) \right] E\left[ X_{u^{\prime },v^{\prime }}\left(
X_{s,u}-Y_{s,u}\right) \right] \right\vert \\
&&+\left\vert E\left[ X_{u^{\prime },v^{\prime }}\left( X_{s,u^{\prime
}}-Y_{s,u^{\prime }}\right) \right] E\left[ X_{u,v}\left(
X_{s,u}-Y_{s,u}\right) \right] \right\vert \\
&\leq &V_{\rho -var}\left( R_{\left( X,Y\right) },\left[ u,v\right] \times %
\left[ u^{\prime },v^{\prime }\right] \right) V_{\gamma -var}\left( R_{X-Y},%
\left[ s,t\right] ^{2}\right) \\
&&+2V_{\gamma -var}\left( R_{\left( X,X-Y\right) },\left[ u,v\right] \times %
\left[ s,t\right] \right) V_{\gamma -var}\left( R_{\left( X,X-Y\right) },%
\left[ u^{\prime },v^{\prime }\right] \times \left[ s,t\right] \right) \\
&\leq &\epsilon ^{2}\omega \left( \left[ u,v\right] \times \left[ u^{\prime
},v^{\prime }\right] \right) ^{1/\rho }\omega \left( \left[ s,t\right]
^{2}\right) ^{1/\gamma } \\
&&+2\epsilon ^{2}\omega \left( \left[ s,t\right] ^{2}\right) ^{1/\rho
-1/\gamma }\omega \left( \left[ u,v\right] \times \left[ s,t\right] \right)
^{1/\gamma }\omega \left( \left[ u^{\prime },v^{\prime }\right] \times \left[
s,t\right] \right) ^{1/\gamma }.
\end{eqnarray*}%
Now take two partitions $D,\tilde{D}$ of $\left[ s,t\right] $. With our
calculations above,%
\begin{eqnarray*}
&&\sum_{t_{i}\in D,\tilde{t}_{j}\in \tilde{D}}\left\vert E\left(
X_{t_{i},t_{i+1}}\left( X_{s,t_{i}}-Y_{s,t_{i}}\right) X_{\tilde{t}_{j},%
\tilde{t}_{j+1}}\left( X_{s,\tilde{t}_{j}}-Y_{s,\tilde{t}_{j}}\right)
\right) \right\vert ^{\gamma } \\
&\leq &c_{1}\epsilon ^{2\gamma }\omega \left( \left[ s,t\right] ^{2}\right)
\sum_{t_{i}\in D,\tilde{t}_{j}\in \tilde{D}}\omega \left( \left[
t_{i},t_{i+1}\right] \times \left[ \tilde{t}_{j},\tilde{t}_{j+1}\right]
\right) ^{\gamma /\rho } \\
&&+c_{2}\epsilon ^{2\gamma }\omega \left( \left[ s,t\right] ^{2}\right)
^{\gamma /\rho -1}\sum_{t_{i}\in D,\tilde{t}_{j}\in \tilde{D}}\omega \left( %
\left[ t_{i},t_{i+1}\right] \times \left[ s,t\right] \right) \omega \left( %
\left[ \tilde{t}_{j},\tilde{t}_{j+1}\right] \times \left[ s,t\right] \right)
\\
&\leq &c_{3}\epsilon ^{2\gamma }\left( \omega \left( \left[ s,t\right]
^{2}\right) \omega \left( \left[ s,t\right] ^{2}\right) ^{\gamma /\rho
}+\omega \left( \left[ s,t\right] ^{2}\right) ^{\gamma /\rho -1}\omega
\left( \left[ s,t\right] ^{2}\right) ^{2}\right) .
\end{eqnarray*}%
The other terms are treated exactly the same way. Taking the supremum over
all partitions shows the result.
\end{proof}

The next corollary completes the proof of Proposition \ref%
{prop_main_estimates_n3}.

\begin{corollary}
Let $\left( X,Y\right) $, $\omega $, $\rho $ and $\gamma $ as in Lemma \ref%
{lemma_diff_alldifferent}. Then there is a constant $C=C\left( \rho ,\gamma
\right) $ such that%
\begin{equation*}
\left\vert \mathbf{X}_{s,t}^{i,i,j}-\mathbf{Y}_{s,t}^{i,i,j}\right\vert
_{L^{2}}\leq C\epsilon \omega \left( \left[ s,t\right] ^{2}\right) ^{\frac{1%
}{2\gamma }}\omega \left( \left[ s,t\right] ^{2}\right) ^{\frac{2}{2\rho }}
\end{equation*}%
holds for every $\left( s,t\right) \in \Delta $ and $i\neq j$ where $%
\epsilon ^{2}=V_{\infty }\left( R_{X-Y},\left[ s,t\right] ^{2}\right)
^{1-\rho /\gamma }$.
\end{corollary}

\begin{proof}
From the triangle inequality,%
\begin{equation*}
\left\vert \mathbf{X}_{s,t}^{i,i,j}-\mathbf{Y}_{s,t}^{i,i,j}\right\vert
_{L^{2}}\leq \left\vert \int_{\left[ s,t\right] }\left( \mathbf{X}%
_{s,u}^{i,i}-\mathbf{Y}_{s,u}^{i,i}\right) \,dY_{u}^{j}\right\vert
_{L^{2}}+\left\vert \int_{\left[ s,t\right] }\mathbf{Y}_{s,u}^{i,i}\,d\left(
X^{j}-Y^{j}\right) _{u}\right\vert _{L^{2}}.
\end{equation*}%
For the first integral, we use independence to move the expectation inside
the integral as seen in the proof of Lemma \ref{lemma_diff_alldifferent},
then we use $2D$ Young integration and Lemma \ref{lemma_rhovar_diff_n2} to
obtain the desired estimate. The second integral is estimated in the same
way using Lemma \ref{lemma_rho_var_same_letters}.
\end{proof}

\subsubsection{$n=4\label{section_n=4}$}

\begin{proposition}
\label{prop_main_estimates_n4}Let $\left( X,Y\right) $, $\omega $, $\rho $
and $\gamma $ as in Lemma \ref{lemma_diff_alldifferent}. Then there is a
constant $C\left( 4\right) $ which depends on $\rho $ and $\gamma $ such that%
\begin{equation*}
\left\vert \mathbf{X}_{s,t}^{4}-\mathbf{Y}_{s,t}^{4}\right\vert _{L^{2}}\leq
C\left( 4\right) \epsilon \omega \left( \left[ s,t\right] ^{2}\right) ^{%
\frac{1}{2\gamma }}\omega \left( \left[ s,t\right] ^{2}\right) ^{\frac{3}{%
2\rho }}
\end{equation*}%
holds for every $\left( s,t\right) \in \Delta $ where $\epsilon
^{2}=V_{\infty }\left( R_{X-Y},\left[ s,t\right] ^{2}\right) ^{1-\rho
/\gamma }$.
\end{proposition}

\begin{proof}
From Proposition \ref{proposition_key_shuffle} and \ref{prop_generating_sets}
one sees that it is enough to show the estimate for $\mathbf{X}^{w}-\mathbf{Y%
}^{w}$ where%
\begin{equation*}
w\in \left\{ iiii,ijkl,iijj,iiij,iijk,jiik:i,j,k,l\in \left\{ 1,\ldots
,d\right\} ~\text{distinct}\right\} .
\end{equation*}%
The cases $w=iiii$ and $w=ijkl$ are special cases of Lemma \ref%
{lemma_diff_allthesame} and Lemma \ref{lemma_diff_alldifferent}. Hence it
remains to show the estimate for%
\begin{equation*}
w\in \left\{ iijj,iiij,iijk,jiik:i,j,k\in \left\{ 1,\ldots ,d\right\} ~\text{%
pairwise distinct}\right\} \text{.}
\end{equation*}%
This is the content of the remaining section.
\end{proof}

\begin{lemma}
Let $\left( X,Y\right) $, $\omega $, $\rho $ and $\gamma $ as in Lemma \ref%
{lemma_diff_alldifferent}. Then there is a constant $C=C\left( \rho ,\gamma
\right) $ such that%
\begin{equation*}
\left\vert \mathbf{X}_{s,t}^{i,i,j,k}-\mathbf{Y}_{s,t}^{i,i,j,k}\right\vert
_{L^{2}}\leq C\epsilon \omega \left( \left[ s,t\right] ^{2}\right) ^{\frac{1%
}{2\gamma }}\omega \left( \left[ s,t\right] ^{2}\right) ^{\frac{3}{2\rho }}
\end{equation*}%
holds for every $\left( s,t\right) \in \Delta $ where $i,j,k$ are distinct
and $\epsilon ^{2}=V_{\infty }\left( R_{X-Y},\left[ s,t\right] ^{2}\right)
^{1-\rho /\gamma }$.
\end{lemma}

\begin{proof}
From the triangle inequality,%
\begin{eqnarray*}
&&\left\vert \mathbf{X}_{s,t}^{i,i,j,k}-\mathbf{Y}_{s,t}^{i,i,j,k}\right%
\vert _{L^{2}} \\
&=&\left\vert \int_{\left\{ s<u<v<t\right\} }\mathbf{X}_{s,u}^{i,i}%
\,dX_{u}^{j}\,dX_{v}^{k}-\int_{\left\{ s<u<v<t\right\} }\mathbf{Y}%
_{s,u}^{i,i}\,dY_{u}^{j}\,dY_{v}^{k}\right\vert _{L^{2}} \\
&\leq &\left\vert \int_{\left\{ s<u<v<t\right\} }\left( \mathbf{X}%
_{s,u}^{i,i}-\mathbf{Y}_{s,u}^{i,i}\right)
\,dX_{u}^{j}\,dX_{v}^{k}\right\vert _{L^{2}}+\left\vert \int_{\left\{
s<u<v<t\right\} }\mathbf{Y}_{s,u}^{i,i}\,d\left( X^{j}-Y^{j}\right)
_{u}\,dX_{v}^{k}\right\vert _{L^{2}} \\
&&+\left\vert \int_{\left\{ s<u<v<t\right\} }\mathbf{Y}_{s,u}^{i,i}%
\,dY_{u}^{j}\,d\left( X^{k}-Y^{k}\right) _{v}\right\vert _{L^{2}}.
\end{eqnarray*}%
For the first integral, we use Proposition \ref{prop_rho_var_iter_2D} and
Lemma \ref{lemma_rhovar_diff_n2} to obtain%
\begin{eqnarray*}
\left\vert \int_{\left\{ s<u<v<t\right\} }\left( \mathbf{X}_{s,u}^{i,i}-%
\mathbf{Y}_{s,u}^{i,i}\right) \,dX_{u}^{j}\,dX_{v}^{k}\right\vert
_{L^{2}}^{2} &=&\int_{\Delta _{s,t}^{2}\times \Delta _{s,t}^{2}}E\left[
\left( \mathbf{X}_{s,\cdot }^{i,i}-\mathbf{Y}_{s,\cdot }^{i,i}\right) \left( 
\mathbf{X}_{s,\cdot }^{i,i}-\mathbf{Y}_{s,\cdot }^{i,i}\right) \right]
\,dR_{X^{j}}\,dR_{X^{k}} \\
&\leq &c_{1}\epsilon ^{2}\omega \left( \left[ s,t\right] ^{2}\right)
^{1/\gamma +1/\rho }\omega \left( \left[ s,t\right] ^{2}\right) ^{2/\rho }.
\end{eqnarray*}%
For the other two integrals we also use Proposition \ref%
{prop_rho_var_iter_2D} together with Lemma \ref{lemma_rho_var_same_letters}
to obtain the same estimate.
\end{proof}

\begin{lemma}
\label{lemma_rhovar_diff_n3}Let $\left( X,Y\right) \colon \left[ 0,1\right]
\rightarrow \mathbb{R}^{2}$ be a centred Gaussian process with continuous
paths of finite variation. Assume that the $\rho $-variation of $R_{\left(
X,Y\right) }$ is controlled by a $2D$-control $\omega $ where $\rho \geq 1$.
Consider the function%
\begin{equation*}
g\left( u,v\right) =E\left[ \left( \mathbf{X}_{s,u}^{\left( 3\right) }-%
\mathbf{Y}_{s,u}^{\left( 3\right) }\right) \left( \mathbf{X}_{s,v}^{\left(
3\right) }-\mathbf{Y}_{s,v}^{\left( 3\right) }\right) \right] .
\end{equation*}%
Then for every $\gamma >\rho $ there is a constant $C=C\left( \rho ,\gamma
\right) $ such that%
\begin{equation*}
V_{\gamma -var}\left( g,\left[ s,t\right] ^{2}\right) \leq C\epsilon
^{2}\omega \left( \left[ s,t\right] ^{2}\right) ^{1/\gamma +2/\rho }
\end{equation*}%
holds for every $\left( s,t\right) \in \Delta $ where $\epsilon
^{2}=V_{\infty }\left( R_{X-Y},\left[ s,t\right] ^{2}\right) ^{\left( 1-\rho
/\gamma \right) }$.
\end{lemma}

\begin{proof}
Similar to the one of Lemma \ref{lemma_rhovar_diff_n2} applying again Wick's
formula.
\end{proof}

\begin{corollary}
Let $\left( X,Y\right) $, $\omega $, $\rho $ and $\gamma $ as in Lemma \ref%
{lemma_diff_alldifferent}. Then there is a constant $C=C\left( \rho ,\gamma
\right) $ such that%
\begin{equation*}
\left\vert \mathbf{X}_{s,t}^{i,i,i,j}-\mathbf{Y}_{s,t}^{i,i,i,j}\right\vert
_{L^{2}}\leq C\epsilon \omega \left( \left[ s,t\right] ^{2}\right) ^{\frac{1%
}{2\gamma }}\omega \left( \left[ s,t\right] ^{2}\right) ^{\frac{3}{2\rho }}
\end{equation*}%
holds for every $\left( s,t\right) \in \Delta $ and $i\neq j$ where $%
\epsilon ^{2}=V_{\infty }\left( R_{X-Y},\left[ s,t\right] ^{2}\right)
^{\left( 1-\rho /\gamma \right) }$.
\end{corollary}

\begin{proof}
The triangle inequality gives%
\begin{eqnarray*}
\left\vert \mathbf{X}_{s,t}^{i,i,i,j}-\mathbf{Y}_{s,t}^{i,i,i,j}\right\vert
_{L^{2}} &=&\left\vert \int_{\left[ s,t\right] }\mathbf{X}%
_{s,u}^{i,i,i}\,dX_{u}^{j}-\int_{\left[ s,t\right] }\mathbf{Y}%
_{s,u}^{i,i,i}\,dY_{u}^{j}\right\vert \\
&\leq &\left\vert \int_{\left[ s,t\right] }\left( \mathbf{X}_{s,u}^{i,i,i}-%
\mathbf{Y}_{s,u}^{i,i,i}\right) \,dX_{u}^{j}\right\vert _{L^{2}}+\left\vert
\int_{\left[ s,t\right] }\mathbf{Y}_{s,u}^{i,i,i}\,d\left(
X^{j}-Y^{j}\right) _{u}\right\vert _{L^{2}}.
\end{eqnarray*}%
For the first integral, we move the expectation inside the integral, use $2D$
Young integration and Lemma \ref{lemma_rhovar_diff_n3} to conclude the
estimate. The second integral is estimated the same way applying Lemma \ref%
{lemma_rho_var_same_letters}.
\end{proof}

It remains to show the estimates for $\mathbf{X}^{w}-\mathbf{Y}^{w}$ where $%
w\in \left\{ iijj,jiik\right\} $. We need to be a bit careful here for the
following reason: It is clear that $\mathbf{X}_{0,1}^{i,i,j}=\int_{\left[ 0,1%
\right] }\mathbf{X}_{u}^{i,i}\,dX_{u}^{j}$. One might expect that also $%
\mathbf{X}_{0,1}^{j,i,i}=\int_{\left[ 0,1\right] }X_{u}^{j}\,d\mathbf{X}%
_{u}^{i,i}$ holds, but this is not true in general. Indeed, just take $%
f\left( u\right) =g\left( u\right) =u$. Then%
\begin{equation*}
\int_{0}^{1}f\left( u\right) \,d\left( \int_{0}^{u}g\left( v\right)
\,dg\left( v\right) \right) =\frac{1}{2}\int_{0}^{1}u\,d\left( u^{2}\right)
=\int_{0}^{1}u^{2}\,du=\frac{1}{3}
\end{equation*}%
but%
\begin{equation*}
\int_{\Delta _{0,1}^{2}}f\left( u\right) \,dg\left( u\right) \,dg\left(
v\right) =\int_{\Delta _{0,1}^{3}}du_{1}\,du_{2}\,du_{3}=\frac{1}{6}.
\end{equation*}%
One the other hand, if $g$ is smooth, we can use Fubini to see that%
\begin{eqnarray*}
\int_{\Delta _{0,1}^{2}}f\left( u\right) \,dg\left( u\right) \,dg\left(
v\right) &=&\int_{\left[ 0,1\right] ^{2}}f\left( u\right) g^{\prime }\left(
u\right) g^{\prime }\left( v\right) 1_{\left\{ u<v\right\} }\,du\,dv \\
&=&\frac{1}{2}\int_{\left[ 0,1\right] ^{2}}f\left( u\right) g^{\prime
}\left( u\right) g^{\prime }\left( v\right) 1_{\left\{ u<v\right\} }\,du\,dv
\\
&&+\frac{1}{2}\int_{\left[ 0,1\right] ^{2}}f\left( v\right) g^{\prime
}\left( v\right) g^{\prime }\left( u\right) 1_{\left\{ v<u\right\} }\,du\,dv
\\
&=&\frac{1}{2}\int_{\left[ 0,1\right] ^{2}}\left( f\left( u\right)
1_{\left\{ u<v\right\} }+f\left( v\right) 1_{\left\{ v<u\right\} }\right)
g^{\prime }\left( u\right) g^{\prime }\left( v\right) \,du\,dv \\
&=&\frac{1}{2}\int_{\left[ 0,1\right] ^{2}}f\left( u\wedge v\right)
g^{\prime }\left( u\right) g^{\prime }\left( v\right) \,du\,dv \\
&=&\frac{1}{2}\int_{\left[ 0,1\right] ^{2}}f\left( u\wedge v\right)
\,d\left( g\left( u\right) g\left( v\right) \right)
\end{eqnarray*}%
where the last integral is a $2D$ Young integral. Hence we have seen that an
iterated $1D$-integral can be transformed into a usual $2D$-integral. We
will use this trick for the remaining estimates.

\begin{lemma}
\label{lemma_rho_var_invar_symm}Let $f\colon \left[ 0,1\right]
^{2}\rightarrow \mathbb{R}$ be a continuous function. Set%
\begin{equation*}
\bar{f}\left( u_{1},u_{2},v_{1},v_{2}\right) =f\left( u_{1}\wedge
u_{2},v_{1}\wedge v_{2}\right) .
\end{equation*}

\begin{enumerate}
\item Let $u_{1}<\tilde{u}_{1},u_{2}<\tilde{u}_{2},v_{1}<\tilde{v}_{1},v_{2}<%
\tilde{v}_{2}$ be all in $\left[ 0,1\right] $. Then%
\begin{equation*}
\bar{f}\left( 
\begin{array}{c}
u_{1},\tilde{u}_{1} \\ 
u_{2},\tilde{u}_{2} \\ 
v_{1},\tilde{v}_{1} \\ 
v_{2},\tilde{v}_{2}%
\end{array}%
\right) =f%
\begin{pmatrix}
u,\tilde{u} \\ 
v,\tilde{v}%
\end{pmatrix}%
\end{equation*}%
where we set%
\begin{eqnarray*}
\left[ u,\tilde{u}\right] &=&\left\{ 
\begin{array}{ccc}
\left[ u_{1},\tilde{u}_{1}\right] \cap \left[ u_{2},\tilde{u}_{2}\right] & 
\text{if} & \left[ u_{1},\tilde{u}_{1}\right] \cap \left[ u_{2},\tilde{u}_{2}%
\right] \neq \emptyset \\ 
\left[ 0,0\right] & \text{if} & \left[ u_{1},\tilde{u}_{1}\right] \cap \left[
u_{2},\tilde{u}_{2}\right] =\emptyset%
\end{array}%
\right. . \\
\left[ v,\tilde{v}\right] &=&\left\{ 
\begin{array}{ccc}
\left[ v_{1},\tilde{v}_{1}\right] \cap \left[ v_{2},\tilde{v}_{2}\right] & 
\text{if} & \left[ v_{1},\tilde{v}_{1}\right] \cap \left[ v_{2},\tilde{v}_{2}%
\right] \neq \emptyset \\ 
\left[ 0,0\right] & \text{if} & \left[ v_{1},\tilde{v}_{1}\right] \cap \left[
v_{2},\tilde{v}_{2}\right] =\emptyset%
\end{array}%
\right.
\end{eqnarray*}

\item For $s<t$, $\sigma <t$ and $p\geq 1$ we have%
\begin{equation*}
V_{p}\left( f,\left[ s,t\right] \times \left[ \sigma ,\tau \right] \right)
=V_{p}\left( \bar{f},\left[ s,t\right] ^{2}\times \left[ \sigma ,\tau \right]
^{2}\right) .
\end{equation*}
\end{enumerate}
\end{lemma}

\begin{proof}
\begin{enumerate}
\item By definition of the higher dimensional increments,%
\begin{eqnarray*}
\bar{f}\left( 
\begin{array}{c}
u_{1},\tilde{u}_{1} \\ 
u_{2},\tilde{u}_{2} \\ 
v_{1} \\ 
v_{2}%
\end{array}%
\right) &=&\bar{f}\left( 
\begin{array}{c}
\tilde{u}_{1} \\ 
\tilde{u}_{2} \\ 
v_{1} \\ 
v_{2}%
\end{array}%
\right) -\bar{f}\left( 
\begin{array}{c}
\tilde{u}_{1} \\ 
u_{2} \\ 
v_{1} \\ 
v_{2}%
\end{array}%
\right) -\bar{f}\left( 
\begin{array}{c}
u_{1} \\ 
\tilde{u}_{2} \\ 
v_{1} \\ 
v_{2}%
\end{array}%
\right) +\bar{f}\left( 
\begin{array}{c}
u_{1} \\ 
u_{2} \\ 
v_{1} \\ 
v_{2}%
\end{array}%
\right) \\
&=&f\left( \tilde{u}_{1}\wedge \tilde{u}_{2},v_{1}\wedge v_{2}\right)
-f\left( \tilde{u}_{1}\wedge u_{2},v_{1}\wedge v_{2}\right) \\
&&-f\left( u_{1}\wedge \tilde{u}_{2},v_{1}\wedge v_{2}\right) +f\left(
u_{1}\wedge u_{2},v_{1}\wedge v_{2}\right) .
\end{eqnarray*}%
By a case distinction, one sees that this is equal to $f\left( \tilde{u}%
,v_{1}\wedge v_{2}\right) -f\left( u,v_{1}\wedge v_{2}\right) $. One goes on
with%
\begin{eqnarray*}
\bar{f}\left( 
\begin{array}{c}
u_{1},\tilde{u}_{1} \\ 
u_{2},\tilde{u}_{2} \\ 
v_{1},\tilde{v}_{1} \\ 
v_{2},\tilde{v}_{2}%
\end{array}%
\right) &=&\bar{f}\left( 
\begin{array}{c}
u_{1},\tilde{u}_{1} \\ 
u_{2},\tilde{u}_{2} \\ 
\tilde{v}_{1} \\ 
\tilde{v}_{2}%
\end{array}%
\right) -\bar{f}\left( 
\begin{array}{c}
u_{1},\tilde{u}_{1} \\ 
u_{2},\tilde{u}_{2} \\ 
\tilde{v}_{1} \\ 
v_{2}%
\end{array}%
\right) -\bar{f}\left( 
\begin{array}{c}
u_{1},\tilde{u}_{1} \\ 
u_{2},\tilde{u}_{2} \\ 
v_{1} \\ 
\tilde{v}_{2}%
\end{array}%
\right) +\bar{f}\left( 
\begin{array}{c}
u_{1},\tilde{u}_{1} \\ 
u_{2},\tilde{u}_{2} \\ 
v_{1} \\ 
v_{2}%
\end{array}%
\right) \\
&=&h\left( \tilde{v}_{1}\wedge \tilde{v}_{2}\right) -h\left( \tilde{v}%
_{1}\wedge v_{2}\right) -h\left( v_{1}\wedge \tilde{v}_{2}\right) +h\left(
v_{1}\wedge v_{2}\right) \\
&=&h\left( \tilde{v}\right) -h\left( v\right)
\end{eqnarray*}%
where $h\left( \cdot \right) =f\left( \tilde{u},\cdot \right) -f\left(
u,\cdot \right) .$Hence%
\begin{equation*}
h\left( \tilde{v}\right) -h\left( v\right) =f\left( \tilde{u},\tilde{v}%
\right) -f\left( u,\tilde{v}\right) -f\left( \tilde{u},v\right) +f\left(
u,v\right) =f\left( 
\begin{array}{c}
u,\tilde{u} \\ 
v,\tilde{v}%
\end{array}%
\right) .
\end{equation*}

\item Let $D$ be a partition of $\left[ s,t\right] $ and $\tilde{D}$ a
partition of $\left[ \sigma ,\tau \right] $. Then by 1,%
\begin{equation*}
\sum_{t_{i}\in D,\tilde{t}\in \tilde{D}}\left\vert f\left( 
\begin{array}{c}
t_{i},t_{i+1} \\ 
\tilde{t}_{j},\tilde{t}_{j+1}%
\end{array}%
\right) \right\vert ^{p}=\sum_{t_{i}\in D,\tilde{t}\in \tilde{D}}\left\vert 
\bar{f}\left( 
\begin{array}{c}
t_{i},t_{i+1} \\ 
t_{i},t_{i+1} \\ 
\tilde{t}_{j},\tilde{t}_{j+1} \\ 
\tilde{t}_{j},\tilde{t}_{j+1}%
\end{array}%
\right) \right\vert ^{p}\leq \left( V_{p}\left( \bar{f},\left[ s,t\right]
^{2}\times \left[ \sigma ,\tau \right] ^{2}\right) \right) ^{p},
\end{equation*}%
hence $V_{p}\left( f,\left[ s,t\right] \times \left[ \sigma ,\tau \right]
\right) \leq V_{p}\left( \bar{f},\left[ s,t\right] ^{2}\times \left[ \sigma
,\tau \right] ^{2}\right) $. Now let $D_{1},D_{2}$ be partitions of $\left[
s,t\right] $ and $\tilde{D}_{1},\tilde{D}_{2}$ be partitions of $\left[
\sigma ,\tau \right] $. Set $D=D_{1}\cup D_{2}$, $\tilde{D}=\tilde{D}%
_{1}\cup \tilde{D}_{2}$. Then $D$ is a partition of $\left[ s,t\right] $ and 
$\tilde{D}$ a partition of $\left[ \sigma ,\tau \right] $ (see Figure 1
below). 
%TCIMACRO{%
%\TeXButton{TeX field}{\begin{figure}[H]
%\begin{center}
%\includegraphics[trim = 0mm 30mm 0mm 10mm, clip, scale=0.9]{Z:/Pictures/rates_gaussian_pic_01.pdf}
%\caption{}
%\end{center}
%\end{figure}} }%
%BeginExpansion
\begin{figure}[H]
\begin{center}
\includegraphics[trim = 0mm 30mm 0mm 10mm, clip, scale=0.9]{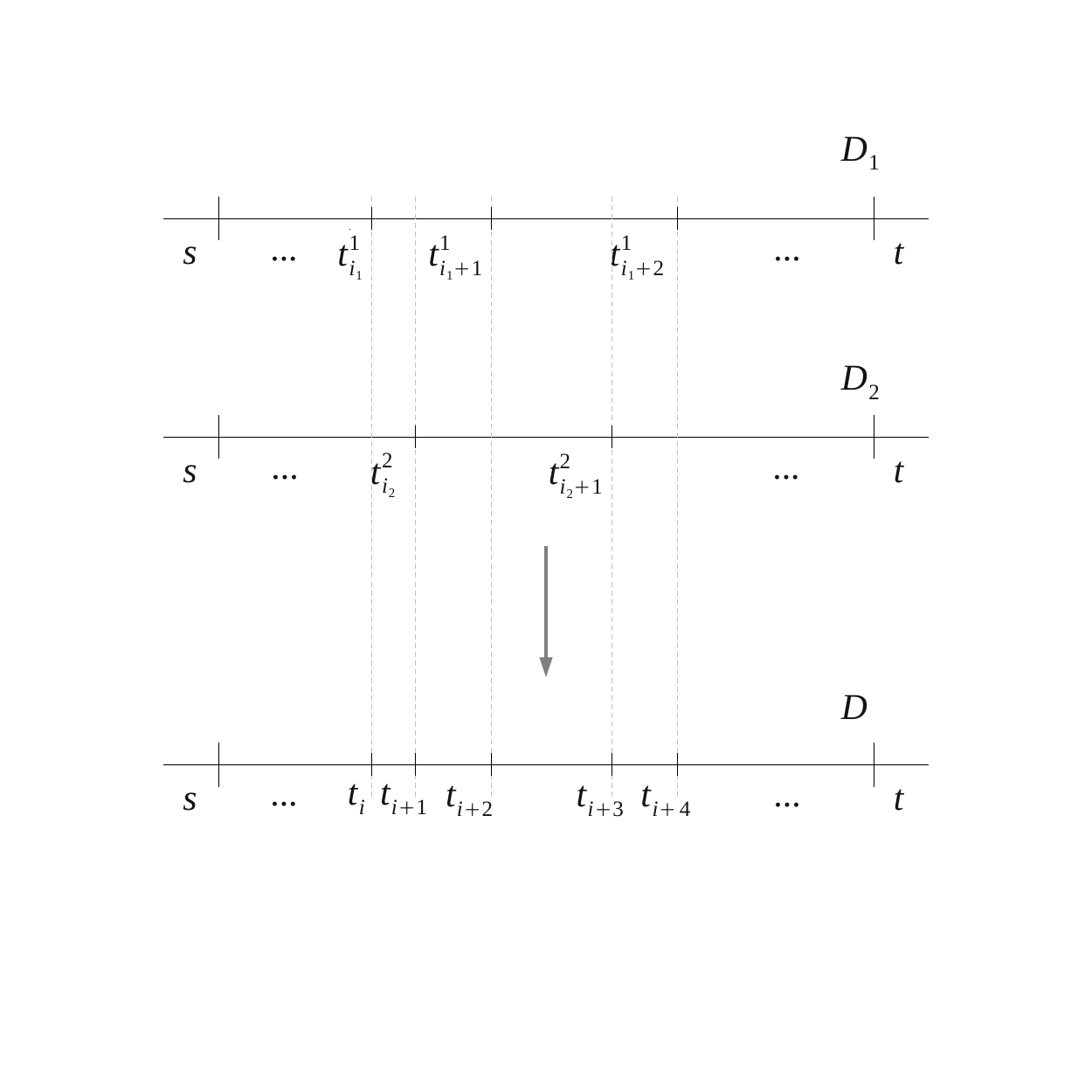}
\caption{}
\end{center}
\end{figure}
%EndExpansion
By (1),%
\begin{equation*}
\sum_{\substack{ t_{i_{1}}^{1}\in D_{1},t_{i_{2}}^{2}\in D_{2}  \\ \tilde{t}%
_{j_{1}}^{1}\in \tilde{D}_{1},\tilde{t}_{j_{2}}^{2}\in \tilde{D}_{2}}}%
\left\vert f\left( 
\begin{array}{c}
t_{i_{1}}^{1},t_{i_{1}+1}^{1} \\ 
t_{i_{2}}^{2},t_{i_{2}+1}^{2} \\ 
\tilde{t}_{j_{1}}^{1},\tilde{t}_{j_{1}+1}^{1} \\ 
\tilde{t}_{j_{2}}^{2},\tilde{t}_{j_{2}+1}^{2}%
\end{array}%
\right) \right\vert ^{p}=\sum_{t_{i}\in D,\tilde{t}\in \tilde{D}}\left\vert
f\left( 
\begin{array}{c}
t_{i},t_{i+1} \\ 
\tilde{t}_{j},\tilde{t}_{j+1}%
\end{array}%
\right) \right\vert ^{p}\leq \left( V_{p}\left( f,\left[ s,t\right] \times %
\left[ \sigma ,\tau \right] \right) \right) ^{p}
\end{equation*}%
and we also get $V_{p}\left( \bar{f},\left[ s,t\right] ^{2}\times \left[
\sigma ,\tau \right] ^{2}\right) \leq V_{p}\left( f,\left[ s,t\right] \times %
\left[ \sigma ,\tau \right] \right) $.
\end{enumerate}
\end{proof}

\begin{lemma}
\label{lemma_rho_var_4D_estimates}Let $\left( X,Y\right) \colon \left[ 0,1%
\right] \rightarrow \mathbb{R}^{2}$ be a centred Gaussian process with
continuous paths of finite variation and assume that $\omega $ is a
symmetric control which controls the $\rho $-variation of $R_{\left(
X,Y\right) }$ where $\rho \geq 1$. Take $\left( s,t\right) \in \Delta $, $%
\gamma >\rho $ and set $\epsilon ^{2}=V_{\infty }\left( R_{X-Y},\left[ s,t%
\right] ^{2}\right) ^{1-\rho /\gamma }$.

\begin{enumerate}
\item Set $f\left( u_{1},u_{2},v_{1},v_{2}\right) =E\left[
X_{u_{1}}X_{u_{2}}X_{v_{1}}X_{v_{2}}\right] $. Then there is a constant $%
C_{1}=C_{1}\left( \rho \right) $ and a symmetric $4D$ grid-control $\tilde{%
\omega}_{1}$ which controls the $\rho $-variation of $f$ and%
\begin{equation*}
V_{\rho }\left( f,\left[ s,t\right] ^{4}\right) \leq \tilde{\omega}%
_{1}\left( \left[ s,t\right] ^{4}\right) ^{1/\rho }=C_{1}\omega \left( \left[
s,t\right] ^{2}\right) ^{\frac{2}{\rho }}.
\end{equation*}

\item Set $\tilde{f}\left( u_{1},u_{2},v_{1},v_{2}\right) =E\left[ \mathbf{X}%
_{s,u_{1}\wedge u_{2}}^{\left( 2\right) }\mathbf{X}_{s,v_{1}\wedge
v_{2}}^{\left( 2\right) }\right] $. Then there is a constant $%
C_{2}=C_{2}\left( \rho \right) $ such that%
\begin{equation*}
V_{\rho }\left( \tilde{f},\left[ s,t\right] ^{4}\right) \leq C_{2}\omega
\left( \left[ s,t\right] ^{2}\right) ^{\frac{2}{\rho }}.
\end{equation*}

\item Set%
\begin{equation*}
g\left( u_{1},u_{2},v_{1},v_{2}\right) =E\left[ \left(
X_{u_{1}}X_{u_{2}}-Y_{u_{1}}Y_{u_{2}}\right) \left(
X_{v_{1}}X_{v_{2}}-Y_{v_{1}}Y_{v_{2}}\right) \right] .
\end{equation*}%
Then there is a constant $C_{3}=C_{3}\left( \rho ,\gamma \right) $ and a
symmetric $4D$ grid-control $\tilde{\omega}_{2}$ which controls the $\gamma $%
-variation of $g$ and%
\begin{equation*}
V_{\gamma }\left( g,\left[ s,t\right] ^{4}\right) \leq \tilde{\omega}%
_{2}\left( \left[ s,t\right] ^{4}\right) ^{1/\gamma }=C_{3}\epsilon
^{2}\omega \left( \left[ s,t\right] ^{2}\right) ^{1/\gamma +1/\rho }.
\end{equation*}

\item Set%
\begin{equation*}
\tilde{g}\left( u_{1},u_{2},v_{1},v_{2}\right) =E\left[ \left( \mathbf{X}%
^{\left( 2\right) }-\mathbf{Y}^{\left( 2\right) }\right) _{s,u_{1}\wedge
u_{2}}\left( \mathbf{X}^{\left( 2\right) }-\mathbf{Y}^{\left( 2\right)
}\right) _{s,v_{1}\wedge v_{2}}\right] .
\end{equation*}%
Then there is a constant $C_{4}=C_{4}\left( \rho ,\gamma \right) $ such that%
\begin{equation*}
V_{\gamma }\left( \tilde{g},\left[ s,t\right] ^{4}\right) \leq C_{4}\epsilon
^{2}\omega \left( \left[ s,t\right] ^{2}\right) ^{1/\gamma +1/\rho }.
\end{equation*}
\end{enumerate}
\end{lemma}

\begin{proof}
\begin{enumerate}
\item Let $u_{1}<\tilde{u}_{1}$, $u_{2}<\tilde{u}_{2}$, $v_{1}<\tilde{v}_{1}$%
, $v_{2}<\tilde{v}_{2}$. By the Wick-formula,%
\begin{eqnarray*}
&&\left\vert E\left[ X_{u_{1},\tilde{u}_{1}}X_{u_{2},\tilde{u}_{2}}X_{v_{1},%
\tilde{v}_{1}}X_{v_{2},\tilde{v}_{2}}\right] \right\vert ^{\rho } \\
&\leq &3^{\rho -1}\left\vert E\left[ X_{u_{1},\tilde{u}_{1}}X_{u_{2},\tilde{u%
}_{2}}\right] E\left[ X_{v_{1},\tilde{v}_{1}}X_{v_{2},\tilde{v}_{2}}\right]
\right\vert ^{\rho }+3^{\rho -1}\left\vert E\left[ X_{u_{1},\tilde{u}%
_{1}}X_{v_{1},\tilde{v}_{1}}\right] E\left[ X_{u_{2},\tilde{u}_{2}}X_{v_{2},%
\tilde{v}_{2}}\right] \right\vert ^{\rho } \\
&&+3^{\rho -1}\left\vert E\left[ X_{u_{1},\tilde{u}_{1}}X_{v_{2},\tilde{v}%
_{2}}\right] E\left[ X_{u_{2},\tilde{u}_{2}}X_{v_{1},\tilde{v}_{1}}\right]
\right\vert ^{\rho } \\
&\leq &3^{\rho -1}\omega \left( \left[ u_{1},\tilde{u}_{1}\right] \times %
\left[ u_{2},\tilde{u}_{2}\right] \right) \omega \left( \left[ v_{1},\tilde{v%
}_{1}\right] \times \left[ v_{2},\tilde{v}_{2}\right] \right) \\
&&+3^{\rho -1}\omega \left( \left[ u_{1},\tilde{u}_{1}\right] \times \left[
v_{1},\tilde{v}_{1}\right] \right) \omega \left( \left[ u_{2},\tilde{u}_{2}%
\right] \times \left[ v_{2},\tilde{v}_{2}\right] \right) \\
&&+3^{\rho -1}\omega \left( \left[ u_{1},\tilde{u}_{1}\right] \times \left[
v_{2},\tilde{v}_{2}\right] \right) \omega \left( \left[ u_{2},\tilde{u}_{2}%
\right] \times \left[ v_{1},\tilde{v}_{1}\right] \right) \\
&=&:\tilde{\omega}_{1}\left( \left[ u_{1},\tilde{u}_{1}\right] \times \left[
u_{2},\tilde{u}_{2}\right] \times \left[ v_{1},\tilde{v}_{1}\right] \times %
\left[ v_{2},\tilde{v}_{2}\right] \right) .
\end{eqnarray*}%
It is easy to see that $\tilde{\omega}_{1}$ is a symmetric grid-control and
that it fulfils the stated property.

\item A direct consequence of Lemma \ref{lemma_rho_var_same_letters} and
Lemma \ref{lemma_rho_var_invar_symm}.

\item We have%
\begin{equation*}
X_{u_{1}}X_{u_{2}}-Y_{u_{1}}Y_{u_{2}}=\left( X_{u_{1}}-Y_{u_{1}}\right)
X_{u_{2}}+Y_{u_{1}}\left( X_{u_{2}}-Y_{u_{2}}\right) .
\end{equation*}%
Hence for $u_{1}<\tilde{u}_{1},u_{2}<\tilde{u}_{2},v_{1}<\tilde{v}_{1},v_{2}<%
\tilde{v}_{2}$,%
\begin{eqnarray*}
\tilde{f}\left( 
\begin{array}{c}
u_{1},\tilde{u}_{1} \\ 
u_{2},\tilde{u}_{2} \\ 
v_{1},\tilde{v}_{1} \\ 
v_{2},\tilde{v}_{2}%
\end{array}%
\right) &=&E\left[ \left( X-Y\right) _{u_{1},\tilde{u}_{1}}X_{u_{2},\tilde{u}%
_{2}}\left( X-Y\right) _{v_{1},\tilde{v}_{1}}X_{v_{2},\tilde{v}_{2}}\right]
\\
&&+E\left[ Y_{u_{1},\tilde{u}_{1}}\left( X-Y\right) _{u_{2},\tilde{u}%
_{2}}\left( X-Y\right) _{v_{1},\tilde{v}_{1}}X_{v_{2},\tilde{v}_{2}}\right]
\\
&&+E\left[ \left( X-Y\right) _{u_{1},\tilde{u}_{1}}X_{u_{2},\tilde{u}%
_{2}}Y_{v_{1},\tilde{v}_{1}}\left( X-Y\right) _{v_{2},\tilde{v}_{2}}\right]
\\
&&+E\left[ Y_{u_{1},\tilde{u}_{1}}\left( X-Y\right) _{u_{2},\tilde{u}%
_{2}}Y_{v_{1},\tilde{v}_{1}}\left( X-Y\right) _{v_{2},\tilde{v}_{2}}\right]
\end{eqnarray*}%
For the first term we have, using Lemma \ref{lemma_half_estimates}, 
\begin{eqnarray*}
&&\left\vert E\left[ \left( X-Y\right) _{u_{1},\tilde{u}_{1}}X_{u_{2},\tilde{%
u}_{2}}\left( X-Y\right) _{v_{1},\tilde{v}_{1}}X_{v_{2},\tilde{v}_{2}}\right]
\right\vert ^{\gamma } \\
&\leq &3^{\gamma -1}\left\vert E\left[ \left( X-Y\right) _{u_{1},\tilde{u}%
_{1}}X_{u_{2},\tilde{u}_{2}}\right] \right\vert ^{\gamma }\left\vert E\left[
\left( X-Y\right) _{v_{1},\tilde{v}_{1}}X_{v_{2},\tilde{v}_{2}}\right]
\right\vert ^{\gamma } \\
&&+3^{\gamma -1}\left\vert E\left[ \left( X-Y\right) _{u_{1},\tilde{u}%
_{1}}\left( X-Y\right) _{v_{1},\tilde{v}_{1}}\right] \right\vert ^{\gamma
}\left\vert E\left[ X_{u_{2},\tilde{u}_{2}}X_{v_{2},\tilde{v}_{2}}\right]
\right\vert ^{\gamma } \\
&&+3^{\gamma -1}\left\vert E\left[ \left( X-Y\right) _{u_{1},\tilde{u}%
_{1}}X_{v_{2},\tilde{v}_{2}}\right] \right\vert ^{\gamma }\left\vert E\left[
X_{u_{2},\tilde{u}_{2}}\left( X-Y\right) _{v_{1},\tilde{v}_{1}}\right]
\right\vert ^{\gamma } \\
&\leq &3^{\gamma -1}\epsilon ^{2\gamma }\omega \left( \left[ s,t\right]
^{2}\right) ^{\frac{\gamma }{\rho }-1}\omega \left( \left[ u_{1},\tilde{u}%
_{1}\right] \times \left[ u_{2},\tilde{u}_{2}\right] \right) \omega \left( %
\left[ v_{1},\tilde{v}_{1}\right] \times \left[ v_{2},\tilde{v}_{2}\right]
\right) \\
&&+3^{\gamma -1}\epsilon ^{2\gamma }\omega \left( \left[ u_{1},\tilde{u}_{1}%
\right] \times \left[ v_{1},\tilde{v}_{1}\right] \right) \omega \left( \left[
u_{2},\tilde{u}_{2}\right] \times \left[ v_{2},\tilde{v}_{2}\right] \right)
^{\frac{\gamma }{\rho }} \\
&&+3^{\gamma -1}\epsilon ^{2\gamma }\omega \left( \left[ s,t\right]
^{2}\right) ^{\frac{\gamma }{\rho }-1}\omega \left( \left[ u_{1},\tilde{u}%
_{1}\right] \times \left[ v_{2},\tilde{v}_{2}\right] \right) \omega \left( %
\left[ u_{2},\tilde{u}_{2}\right] \times \left[ v_{1},\tilde{v}_{1}\right]
\right) \\
&\leq &3^{\gamma -1}\epsilon ^{2\gamma }\omega \left( \left[ s,t\right]
^{2}\right) ^{\frac{\gamma }{\rho }-1}\left( \omega \left( \left[ u_{1},%
\tilde{u}_{1}\right] \times \left[ u_{2},\tilde{u}_{2}\right] \right) \omega
\left( \left[ v_{1},\tilde{v}_{1}\right] \times \left[ v_{2},\tilde{v}_{2}%
\right] \right) \right. \\
&&+\omega \left( \left[ u_{1},\tilde{u}_{1}\right] \times \left[ v_{1},%
\tilde{v}_{1}\right] \right) \omega \left( \left[ u_{2},\tilde{u}_{2}\right]
\times \left[ v_{2},\tilde{v}_{2}\right] \right) \\
&&\left. +\omega \left( \left[ u_{1},\tilde{u}_{1}\right] \times \left[
v_{2},\tilde{v}_{2}\right] \right) \omega \left( \left[ u_{2},\tilde{u}_{2}%
\right] \times \left[ v_{1},\tilde{v}_{1}\right] \right) \right) \\
&=&:\tilde{\omega}\left( \left[ u_{1},\tilde{u}_{1}\right] \times \left[
u_{2},\tilde{u}_{2}\right] \times \left[ v_{1},\tilde{v}_{1}\right] \times %
\left[ v_{2},\tilde{v}_{2}\right] \right) .
\end{eqnarray*}%
$\tilde{\omega}$ is a symmetric grid-control and fulfils the stated
property. The other terms are treated in the same way.

\item Follows from Lemma \ref{lemma_rhovar_diff_n2} and Lemma \ref%
{lemma_rho_var_invar_symm}.
\end{enumerate}
\end{proof}

\begin{corollary}
Let $\left( X,Y\right) $, $\omega $, $\rho $ and $\gamma $ as in Lemma \ref%
{lemma_diff_alldifferent}. Then there is a constant $C=C\left( \rho ,\gamma
\right) $ such that%
\begin{equation*}
\left\vert \mathbf{X}_{s,t}^{i,i,j,j}-\mathbf{Y}_{s,t}^{i,i,j,j}\right\vert
_{L^{2}}\leq C\epsilon \omega \left( \left[ s,t\right] ^{2}\right) ^{\frac{1%
}{2\gamma }}\omega \left( \left[ s,t\right] ^{2}\right) ^{\frac{3}{2\rho }}
\end{equation*}%
holds for every $\left( s,t\right) \in \Delta $ and $i\neq j$ where $%
\epsilon ^{2}=V_{\infty }\left( R_{X-Y},\left[ s,t\right] ^{2}\right)
^{1-\rho /\gamma }$.
\end{corollary}

\begin{proof}
As seen before, we can use Fubini to obtain%
\begin{equation*}
\mathbf{X}_{s,t}^{i,i,j,j}=\int_{\Delta _{s,t}^{2}}\mathbf{X}%
_{s,u_{1}}^{i,i}\,dX_{u_{1}}^{j}\,dX_{u_{2}}^{j}=\frac{1}{2}\int_{\left[ s,t%
\right] ^{2}}\mathbf{X}_{s,u_{1}\wedge u_{2}}^{i,i}\,d\left(
X_{u_{1}}^{j}X_{u_{2}}^{j}\right)
\end{equation*}%
and hence%
\begin{eqnarray*}
\left\vert \mathbf{X}_{s,t}^{i,i,j,j}-\mathbf{Y}_{s,t}^{i,i,j,j}\right\vert
_{L^{2}} &\leq &\frac{1}{2}\left\vert \int_{\left[ s,t\right] ^{2}}\left( 
\mathbf{X}_{s,u_{1}\wedge u_{2}}^{i,i}-\mathbf{Y}_{s,u_{1}\wedge
u_{2}}^{i,i}\right) \,d\left( X_{u_{1}}^{j}X_{u_{2}}^{j}\right) \right\vert
_{L^{2}} \\
&&+\frac{1}{2}\left\vert \int_{\left[ s,t\right] ^{2}}\mathbf{Y}%
_{s,u_{1}\wedge u_{2}}^{i,i}\,d\left(
X_{u_{1}}^{j}X_{u_{2}}^{j}-Y_{u_{1}}^{j}Y_{u_{2}}^{j}\right) \right\vert
_{L^{2}}.
\end{eqnarray*}%
We use a Young $4D$-estimate and the estimates of Lemma \ref%
{lemma_rho_var_4D_estimates} to see that%
\begin{eqnarray*}
&&\left\vert \int_{\left[ s,t\right] ^{2}}\left( \mathbf{X}_{s,u_{1}\wedge
u_{2}}^{i,i}-\mathbf{Y}_{s,u_{1}\wedge u_{2}}^{i,i}\right) \,d\left(
X_{u_{1}}^{j}X_{u_{2}}^{j}\right) \right\vert _{L^{2}}^{2} \\
&=&\int_{\left[ s,t\right] ^{4}}E\left[ \left( \mathbf{X}_{s,u_{1}\wedge
u_{2}}^{i,i}-\mathbf{Y}_{s,u_{1}\wedge u_{2}}^{i,i}\right) \left( \mathbf{X}%
_{s,v_{1}\wedge v_{2}}^{i,i}-\mathbf{Y}_{s,v_{1}\wedge v_{2}}^{i,i}\right) %
\right] \,dE\left[ X_{u_{1}}^{j}X_{u_{2}}^{j}X_{v_{1}}^{j}X_{v_{2}}^{j}%
\right] \\
&\leq &c_{1}\epsilon ^{2}\omega \left( \left[ s,t\right] ^{2}\right)
^{1/\gamma }\omega \left( \left[ s,t\right] ^{2}\right) ^{3/\rho }.
\end{eqnarray*}%
The second term is estimated in the same way using again Lemma \ref%
{lemma_rho_var_4D_estimates}.
\end{proof}

\begin{lemma}
\label{lemma_ext_young_rho_var}Let $f\colon \lbrack 0,1]^{2}\rightarrow 
\mathbb{R}$ and $g\colon \lbrack 0,1]^{2}\times \lbrack 0,1]^{2}\rightarrow 
\mathbb{R}$ be continuous where $g$ is symmetric in the first and the last
two variables. Let $\left( s,t\right) \in \Delta $ and assume that $f\left(
s,\cdot \right) =f(\cdot ,s)=0$. Assume also that $f$ has finite $p$%
-variation and that the $q$-variation of $g$ is controlled by a symmetric $%
4D $ grid-control $\tilde{\omega}$ where $\frac{1}{p}+\frac{1}{q}>1$. Define%
\begin{equation*}
\Psi \left( u,v\right) =\int_{[s,u]^{2}\times \lbrack s,v]^{2}}f(u_{1}\wedge
u_{2},v_{1}\wedge v_{2})\,dg\left( u_{1},u_{2};v_{1},v_{2}\right)
\end{equation*}%
Then there is a constant $C=C\left( p,q\right) $ such that%
\begin{equation*}
V_{q}\left( \Psi ;\left[ s,t\right] ^{2}\right) \leq CV_{p}\left( f;\left[
s,t\right] ^{2}\right) \tilde{\omega}\left( \left[ s,t\right] ^{4}\right)
^{1/q}.
\end{equation*}
\end{lemma}

\begin{proof}
Set%
\begin{equation*}
\tilde{f}\left( u_{1},u_{2},v_{1},v_{2}\right) =f(u_{1}\wedge
u_{2},v_{1}\wedge v_{2}).
\end{equation*}%
Let $u<v$ and $u^{\prime }<v^{\prime }$. Note that%
\begin{eqnarray*}
&&1_{[s,v]^{2}\times \lbrack s,v^{\prime }]^{2}}-1_{[s,u]^{2}\times \lbrack
s,v^{\prime }]^{2}}-1_{[s,v]^{2}\times \lbrack s,u^{\prime
}]^{2}}+1_{[s,u]^{2}\times \lbrack s,u^{\prime }]^{2}} \\
&=&1_{\left( [s,v]^{2}\setminus \lbrack s,u]^{2}\right) \times \lbrack
s,v^{\prime }]^{2}}-1_{\left( [s,v]^{2}\setminus \lbrack s,u]^{2}\right)
\times \lbrack s,u^{\prime }]^{2}} \\
&=&1_{\left( [s,v]^{2}\setminus \lbrack s,u]^{2}\right) \times \left(
\lbrack s,v^{\prime }]^{2}\setminus \lbrack s,u^{\prime }]^{2}\right) }
\end{eqnarray*}%
If we take out the square $\left[ s,u\right] ^{2}$ of the larger square $%
\left[ s,v\right] ^{2}$, what is left is the union of three essentially
disjoint squares. More precisely,%
\begin{equation*}
\overline{\left[ s,v\right] ^{2}\setminus \left[ s,u\right] ^{2}}=\left[ u,v%
\right] ^{2}\cup \left( \left[ s,u\right] \times \left[ u,v\right] \right)
\cup \left( \left[ u,v\right] \times \left[ s,u\right] \right) .
\end{equation*}%
The same holds for $u^{\prime }$ and $v^{\prime }$. Hence,%
\begin{eqnarray*}
&&\left( \overline{[s,v]^{2}\setminus \lbrack s,u]^{2}}\right) \times \left( 
\overline{[s,v^{\prime }]^{2}\setminus \lbrack s,u^{\prime }]^{2}}\right) \\
&=&\left( [u,v]^{2}\cup \left( \lbrack s,u]\times \lbrack u,v]\right) \cup
\left( \lbrack u,v]\times \lbrack s,u]\right) \right) \\
&&\times \left( \lbrack u^{\prime },v^{\prime }]^{2}\cup \left( \lbrack
s,u^{\prime }]\times \lbrack u^{\prime },v^{\prime }]\right) \cup \left(
\lbrack u^{\prime },v^{\prime }]\times \lbrack s,u^{\prime }]\right) \right)
\\
&=&\left( [u,v]^{2}\times \lbrack u^{\prime },v^{\prime }]^{2}\right) \cup
\left( \lbrack u,v]^{2}\times \lbrack s,u^{\prime }]\times \lbrack u^{\prime
},v^{\prime }]\right) \cup \left( \lbrack u,v]^{2}\times \lbrack u^{\prime
},v^{\prime }]\times \lbrack s,u^{\prime }]\right) \\
&&\cup \left( \lbrack s,u]\times \lbrack u,v]\times \lbrack u^{\prime
},v^{\prime }]^{2}\right) \cup \left( \lbrack s,u]\times \lbrack u,v]\times
\lbrack s,u^{\prime }]\times \lbrack u^{\prime },v^{\prime }]\right) \\
&&\cup \left( \lbrack s,u]\times \lbrack u,v]\times \lbrack u^{\prime
},v^{\prime }]\times \lbrack s,u^{\prime }]\right) \\
&&\cup \left( \lbrack u,v]\times \lbrack s,u]\times \lbrack u^{\prime
},v^{\prime }]^{2}\right) \cup \left( \lbrack u,v]\times \lbrack s,u]\times
\lbrack s,u^{\prime }]\times \lbrack u^{\prime },v^{\prime }]\right) \\
&&\cup \left( \lbrack u,v]\times \lbrack s,u]\times \lbrack u^{\prime
},v^{\prime }]\times \lbrack s,u^{\prime }]\right)
\end{eqnarray*}%
and all these are unions of essentially disjoint sets. Using continuity and
the symmetry of $\tilde{f}$ and $g$ we have then%
\begin{eqnarray*}
\Psi 
\begin{pmatrix}
u,v \\ 
u^{\prime },v^{\prime }%
\end{pmatrix}
&=&\int_{\left( [s,v]^{2}\setminus \lbrack s,u]^{2}\right) \times \left(
\lbrack s,v^{\prime }]^{2}\setminus \lbrack s,u^{\prime }]^{2}\right) }%
\tilde{f}\,dg \\
&=&\int_{[u,v]^{2}\times \lbrack u^{\prime },v^{\prime }]^{2}}\tilde{f}%
\,dg+2\int_{[u,v]^{2}\times \lbrack s,u^{\prime }]\times \lbrack u^{\prime
},v^{\prime }]}\tilde{f}\,dg \\
&&+2\int_{[s,u]\times \lbrack u,v]\times \lbrack u^{\prime },v^{\prime
}]^{2}}\tilde{f}\,dg+4\int_{[s,u]\times \lbrack u,v]\times \lbrack
s,u^{\prime }]\times \lbrack u^{\prime },v^{\prime }]}\tilde{f}\,dg.
\end{eqnarray*}%
For the first integral we use Young $4D$-estimates. Since $\tilde{f}\left(
s,\cdot ,\cdot ,\cdot \right) =\ldots =\tilde{f}\left( \cdot ,\cdot ,\cdot
,s\right) =0$, we can proceed as in the proof of Lemma \ref%
{lemma_kernel_iter_2D} and use Lemma \ref{lemma_rho_var_invar_symm} to see
that%
\begin{eqnarray*}
\left\vert \int_{\lbrack u,v]^{2}\times \lbrack u^{\prime },v^{\prime }]^{2}}%
\tilde{f}\,dg\right\vert &\leq &c_{1}V_{p}\left( f,\left[ s,t\right]
^{2}\right) V_{q}\left( g,[u,v]^{2}\times \lbrack u^{\prime },v^{\prime
}]^{2}\right) \\
&\leq &c_{1}V_{p}\left( f,\left[ s,t\right] ^{2}\right) \tilde{\omega}\left(
[u,v]^{2}\times \lbrack u^{\prime },v^{\prime }]^{2}\right) ^{1/q}
\end{eqnarray*}%
For the second integral, we have%
\begin{eqnarray*}
&&\int_{[u,v]^{2}\times \lbrack s,u^{\prime }]\times \lbrack u^{\prime
},v^{\prime }]}\tilde{f}\,dg \\
&=&\int_{[u,v]^{2}\times \lbrack s,u^{\prime }]\times \lbrack u^{\prime
},v^{\prime }]}f(u_{1}\wedge u_{2},v_{1}\wedge v_{2})\,dg\left(
u_{1},u_{2};v_{1},v_{2}\right) \\
&=&\int_{[u,v]^{2}\times \lbrack s,u^{\prime }]}f(u_{1}\wedge
u_{2},v_{1})\,d \left[ g\left( u_{1},u_{2};v_{1},v^{\prime }\right) -g\left(
u_{1},u_{2};v_{1},u^{\prime }\right) \right]
\end{eqnarray*}%
We now use a Young $3D$-estimate to see that%
\begin{eqnarray*}
\left\vert \int_{\lbrack u,v]^{2}\times \lbrack s,u^{\prime }]\times \lbrack
u^{\prime },v^{\prime }]}\tilde{f}\,dg\right\vert &\leq &c_{2}V_{p}\left(
f\left( \cdot \wedge \cdot ,\cdot \right) ,\left[ s,t\right] ^{3}\right) \\
&&\times V_{q}\left( g\left( \cdot ,\cdot ;\cdot ,v^{\prime }\right)
-g\left( \cdot ,\cdot ;\cdot ,u^{\prime }\right) ,\left[ u,v\right]
^{2}\times \left[ s,u^{\prime }\right] \right)
\end{eqnarray*}%
As in Lemma \ref{lemma_rho_var_invar_symm}, one can show that $V_{p}\left(
f\left( \cdot \wedge \cdot ,\cdot \right) ,\left[ s,t\right] ^{3}\right)
=V_{p}\left( f,\left[ s,t\right] ^{2}\right) $. For $g$, we have%
\begin{eqnarray*}
V_{q}\left( g\left( \cdot ,\cdot ;\cdot ,v^{\prime }\right) -g\left( \cdot
,\cdot ;\cdot ,u^{\prime }\right) ,\left[ u,v\right] ^{2}\times \left[
s,u^{\prime }\right] \right) &\leq &V_{q}\left( g,\left[ u,v\right]
^{2}\times \left[ s,u^{\prime }\right] \times \left[ u^{\prime },v^{\prime }%
\right] \right) \\
&\leq &\tilde{\omega}\left( \left[ u,v\right] ^{2}\times \left[ s,t\right]
\times \left[ u^{\prime },v^{\prime }\right] \right) ^{1/q}.
\end{eqnarray*}%
Hence%
\begin{equation*}
\left\vert \int_{\lbrack u,v]^{2}\times \lbrack s,u^{\prime }]\times \lbrack
u^{\prime },v^{\prime }]}\tilde{f}\,dg\right\vert \leq c_{2}V_{p}\left( f,%
\left[ s,t\right] ^{2}\right) \tilde{\omega}\left( \left[ u,v\right]
^{2}\times \left[ s,t\right] \times \left[ u^{\prime },v^{\prime }\right]
\right) ^{1/q}.
\end{equation*}%
Similarly, using Young $3D$ and $2D$ estimates, we get%
\begin{equation*}
\left\vert \int_{\lbrack s,u]\times \lbrack u,v]\times \lbrack u^{\prime
},v^{\prime }]^{2}}\tilde{f}\,dg\right\vert \leq c_{3}V_{p}\left( f,\left[
s,t\right] ^{2}\right) \tilde{\omega}\left( \left[ s,t\right] \times \left[
u,v\right] \times \left[ u^{\prime },v^{\prime }\right] ^{2}\right) ^{1/q}
\end{equation*}%
and%
\begin{equation*}
\left\vert \int_{\lbrack s,u]\times \lbrack u,v]\times \lbrack s,u^{\prime
}]\times \lbrack u^{\prime },v^{\prime }]}\tilde{f}\,dg\right\vert \leq
c_{4}V_{p}\left( f,\left[ s,t\right] ^{2}\right) \tilde{\omega}\left( \left[
s,t\right] \times \left[ u,v\right] \times \lbrack s,t]\times \lbrack
u^{\prime },v^{\prime }]\right) ^{1/q}.
\end{equation*}%
Putting all together, using the symmetry of $\tilde{\omega}$ we have shown
that%
\begin{equation*}
\left\vert \Psi 
\begin{pmatrix}
u,v \\ 
u^{\prime },v^{\prime }%
\end{pmatrix}%
\right\vert ^{q}\leq c_{5}V_{p}\left( f,\left[ s,t\right] ^{2}\right) ^{q}%
\tilde{\omega}\left( \left[ u,v\right] \times \lbrack u^{\prime },v^{\prime
}]\times \left[ s,t\right] ^{2}\right) .
\end{equation*}%
Since $\tilde{\omega}_{2}\left( \left[ u,v\right] \times \lbrack u^{\prime
},v^{\prime }]\right) :=\tilde{\omega}\left( \left[ u,v\right] \times
\lbrack u^{\prime },v^{\prime }]\times \left[ s,t\right] ^{2}\right) $ is a $%
2D$ grid-control this shows the claim.
\end{proof}

We are now able to prove the remaining estimate.

\begin{corollary}
Let $\left( X,Y\right) $, $\omega $, $\rho $ and $\gamma $ as in Lemma \ref%
{lemma_diff_alldifferent}. Then there is a constant $C=C\left( \rho ,\gamma
\right) $ such that%
\begin{equation*}
\left\vert \mathbf{X}_{s,t}^{j,i,i,k}-\mathbf{Y}_{s,t}^{j,i,i,k}\right\vert
_{L^{2}}\leq C\epsilon \omega \left( \left[ s,t\right] ^{2}\right) ^{\frac{1%
}{2\gamma }}\omega \left( \left[ s,t\right] ^{2}\right) ^{\frac{3}{2\rho }}
\end{equation*}%
holds for every $\left( s,t\right) \in \Delta $ and $i,j,k$ pairwise
distinct where $\epsilon ^{2}=V_{\infty }\left( R_{X-Y},\left[ s,t\right]
^{2}\right) ^{1-\rho /\gamma }$.
\end{corollary}

\begin{proof}
From%
\begin{equation*}
\int_{\Delta _{s,w}^{2}}X_{s,u_{1}}^{j}\,dX_{u_{1}}^{i}\,dX_{u_{2}}^{i}=%
\frac{1}{2}\int_{\left[ s,w\right] ^{2}}X_{s,u_{1}\wedge u_{2}}^{j}\,d\left(
X_{u_{1}}^{i}X_{u_{2}}^{i}\right)
\end{equation*}%
we see that%
\begin{equation*}
\mathbf{X}_{s,t}^{j,i,i,k}=\frac{1}{2}\int_{s}^{t}\left( \int_{\left[ s,w%
\right] ^{2}}X_{s,u_{1}\wedge u_{2}}^{j}\,d\left(
X_{u_{1}}^{i}X_{u_{2}}^{i}\right) \right) \,dX_{w}^{k}.
\end{equation*}%
Hence%
\begin{eqnarray*}
&&\left\vert \mathbf{X}_{s,t}^{j,i,i,k}-\mathbf{Y}_{s,t}^{j,i,i,k}\right%
\vert _{L^{2}} \\
&\leq &\frac{1}{2}\left\vert \int_{s}^{t}\Psi _{1}\left( w\right)
\,dX_{w}^{k}\right\vert _{L^{2}}+\frac{1}{2}\left\vert \int_{s}^{t}\Psi
_{2}\left( w\right) \,dX_{w}^{k}\right\vert _{L^{2}}+\frac{1}{2}\left\vert
\int_{s}^{t}\Psi _{3}\left( w\right) \,d\left( X^{k}-Y^{k}\right)
_{w}\right\vert _{L^{2}}
\end{eqnarray*}%
where%
\begin{eqnarray*}
\Psi _{1}\left( w\right) &=&\int_{\left[ s,w\right] ^{2}}\left(
X_{s,u_{1}\wedge u_{2}}^{j}-Y_{s,u_{1}\wedge u_{2}}^{j}\right) \,d\left(
X_{u_{1}}^{i}X_{u_{2}}^{i}\right) \\
\Psi _{2}\left( w\right) &=&\int_{\left[ s,w\right] ^{2}}Y_{s,u_{1}\wedge
u_{2}}^{j}\,d\left(
X_{u_{1}}^{i}X_{u_{2}}^{i}-Y_{u_{1}}^{i}Y_{u_{2}}^{i}\right) \\
\Psi _{3}\left( w\right) &=&\int_{\left[ s,w\right] ^{2}}Y_{s,u_{1}\wedge
u_{2}}^{j}\,d\left( Y_{u_{1}}^{i}Y_{u_{2}}^{i}\right) .
\end{eqnarray*}%
We start with the first integral. From independence and Young $2D$-estimates,%
\begin{eqnarray*}
\left\vert \int_{s}^{t}\Psi _{1}\left( w\right) \,dX_{w}^{k}\right\vert
_{L^{2}}^{2} &=&\int_{\left[ s,t\right] ^{2}}E\left[ \Psi _{1}\left(
w_{1}\right) \Psi _{1}\left( w_{2}\right) \right] \,dE\left[
X_{w_{1}}^{k}X_{w_{2}}^{k}\right] \\
&\leq &c_{1}V_{\rho }\left( E\left[ \Psi _{1}\left( \cdot \right) \Psi
_{1}\left( \cdot \right) \right] ,\left[ s,t\right] ^{2}\right) V_{\rho
}\left( R_{X^{k}}\left[ s,t\right] ^{2}\right) .
\end{eqnarray*}%
Now,%
\begin{eqnarray*}
&&E\left[ \Psi _{1}\left( w_{1}\right) \Psi _{1}\left( w_{2}\right) \right]
\\
&=&\int_{\left[ s,w_{1}\right] ^{2}\times \left[ s,w_{2}\right] ^{2}}E\left[
\left( X_{s,u_{1}\wedge u_{2}}^{j}-Y_{s,u_{1}\wedge u_{2}}^{j}\right)
\,\left( X_{s,v_{1}\wedge v_{2}}^{j}-Y_{s,v_{1}\wedge v_{2}}^{j}\right) %
\right] dE\left[ X_{u_{1}}^{i}X_{u_{2}}^{i}X_{v_{1}}^{i}X_{v_{2}}^{i}\right]
.
\end{eqnarray*}%
In Lemma \ref{lemma_rho_var_4D_estimates} we have seen that the $\rho $%
-variation of $E\left[ X_{\cdot }^{i}X_{\cdot }^{i}X_{\cdot }^{i}X_{\cdot
}^{i}\right] $ is controlled by a symmetric grid-control $\tilde{\omega}_{1}$%
. Hence we can apply Lemma \ref{lemma_ext_young_rho_var} to conclude that%
\begin{eqnarray*}
V_{\rho }\left( E\left[ \Psi _{1}\left( \cdot \right) \Psi _{1}\left( \cdot
\right) \right] ,\left[ s,t\right] ^{2}\right) &\leq &c_{2}V_{\gamma }\left(
R_{X-Y};\left[ s,t\right] ^{2}\right) \tilde{\omega}_{1}\left( \left[ s,t%
\right] ^{4}\right) ^{1/\rho } \\
&\leq &c_{3}\epsilon ^{2}\omega \left( \left[ s,t\right] ^{2}\right)
^{1/\gamma }\omega \left( \left[ s,t\right] ^{2}\right) ^{2/\rho }.
\end{eqnarray*}%
Clearly, $V_{\rho }\left( R_{X^{k}}\left[ s,t\right] ^{2}\right) \leq \omega
\left( \left[ s,t\right] ^{2}\right) ^{1/\rho }$ and therefore%
\begin{equation*}
\left\vert \int_{s}^{t}\Psi _{1}\left( w\right) \,dX_{w}^{k}\right\vert
_{L^{2}}^{2}\leq c_{4}\epsilon ^{2}\omega \left( \left[ s,t\right]
^{2}\right) ^{1/\gamma }\omega \left( \left[ s,t\right] ^{2}\right) ^{3/\rho
}.
\end{equation*}%
Now we come to the second integral. From independence,%
\begin{eqnarray*}
\left\vert \int_{s}^{t}\Psi _{2}\left( w\right) \,dX_{w}^{k}\right\vert
_{L^{2}}^{2} &=&\int_{\left[ s,t\right] ^{2}}E\left[ \Psi _{2}\left(
w_{1}\right) \Psi _{2}\left( w_{2}\right) \right] \,dE\left[
X_{w_{1}}^{k}X_{w_{2}}^{k}\right] . \\
&\leq &c_{5}V_{\gamma }\left( E\left[ \Psi _{2}\left( \cdot \right) \Psi
_{2}\left( \cdot \right) \right] ,\left[ s,t\right] ^{2}\right) V_{\rho
}\left( R_{X^{k}}\left[ s,t\right] ^{2}\right) .
\end{eqnarray*}%
Now%
\begin{eqnarray*}
&&E\left[ \Psi _{2}\left( w_{1}\right) \Psi _{2}\left( w_{2}\right) \right]
\\
&=&\int_{\left[ s,w_{1}\right] ^{2}\times \left[ s,w_{2}\right] ^{2}}E\left[
Y_{s,u_{1}\wedge u_{2}}^{j}Y_{s,v_{1}\wedge v_{2}}^{j}\right] \,dE\left[
\left( X_{u_{1}}^{i}X_{u_{2}}^{i}-Y_{u_{1}}^{i}Y_{u_{2}}^{i}\right) \left(
X_{v_{1}}^{i}X_{v_{2}}^{i}-Y_{v_{1}}^{i}Y_{v_{2}}^{i}\right) \right] \\
&=&:\int_{\left[ s,w_{1}\right] ^{2}\times \left[ s,w_{2}\right] ^{2}}E\left[
Y_{s,u_{1}\wedge u_{2}}^{j}Y_{s,v_{1}\wedge v_{2}}^{j}\right] \,dg\left(
u_{1},u_{2},v_{1},v_{2}\right) .
\end{eqnarray*}%
In Lemma \ref{lemma_rho_var_4D_estimates} we have seen that the $4D$ $\gamma 
$-variation of $g$ is controlled by a symmetric $4D$ grid-control $\tilde{%
\omega}_{2}$ where%
\begin{equation*}
\tilde{\omega}_{2}\left( \left[ s,t\right] ^{4}\right) ^{1/\gamma
}=c_{6}\epsilon ^{2}\omega \left( \left[ s,t\right] ^{2}\right) ^{1/\rho
+1/\gamma }.
\end{equation*}%
Hence%
\begin{equation*}
V_{\gamma }\left( E\left[ \Psi _{2}\left( \cdot \right) \Psi _{2}\left(
\cdot \right) \right] ,\left[ s,t\right] ^{2}\right) \leq c_{7}V_{\rho
}\left( R_{Y^{j}};\left[ s,t\right] ^{2}\right) \tilde{\omega}_{2}\left( %
\left[ s,t\right] ^{4}\right) ^{1/\gamma }\leq c_{8}\epsilon ^{2}\omega
\left( \left[ s,t\right] ^{2}\right) ^{2/\rho +1/\gamma }.
\end{equation*}%
This gives us%
\begin{equation*}
\left\vert \int_{s}^{t}\Psi _{2}\left( w\right) \,dX_{w}^{k}\right\vert
_{L^{2}}^{2}\leq c_{9}\epsilon ^{2}\omega \left( \left[ s,t\right]
^{2}\right) ^{1/\gamma }\omega \left( \left[ s,t\right] ^{2}\right) ^{3/\rho
}.
\end{equation*}%
For the third integral we see again that%
\begin{eqnarray*}
\left\vert \int_{s}^{t}\Psi _{3}\left( w\right) \,d\left( X^{k}-Y^{k}\right)
_{w}\right\vert _{L^{2}}^{2} &=&\int_{\left[ s,t\right] ^{2}}E\left[ \Psi
_{3}\left( w_{1}\right) \Psi _{3}\left( w_{2}\right) \right] \,dE\left[
\left( X^{k}-Y^{k}\right) _{w_{1}}\left( X^{k}-Y^{k}\right) _{w_{2}}\right]
\\
&\leq &c_{10}V_{\rho }\left( E\left[ \Psi _{3}\left( \cdot \right) \Psi
_{3}\left( \cdot \right) \right] ,\left[ s,t\right] ^{2}\right) V_{\gamma
}\left( R_{X-Y},\left[ s,t\right] ^{2}\right) .
\end{eqnarray*}%
From%
\begin{equation*}
E\left[ \Psi _{3}\left( w_{1}\right) \Psi _{3}\left( w_{2}\right) \right]
=\int_{\left[ s,w_{1}\right] ^{2}\times \left[ s,w_{2}\right] ^{2}}E\left[
Y_{s,u_{1}\wedge u_{2}}^{j}Y_{s,v_{1}\wedge v_{2}}^{j}\right] \,dE\left[
Y_{u_{1}}^{i}Y_{u_{2}}^{i}Y_{v_{1}}^{i}Y_{v_{2}}^{i}\right]
\end{equation*}%
we see that we can apply Lemma \ref{lemma_ext_young_rho_var} to obtain%
\begin{equation*}
V_{\rho }\left( E\left[ \Psi _{3}\left( \cdot \right) \Psi _{3}\left( \cdot
\right) \right] ,\left[ s,t\right] ^{2}\right) \leq c_{11}V_{\rho }\left(
R_{Y^{j}};\left[ s,t\right] ^{2}\right) \omega \left( \left[ s,t\right]
^{2}\right) ^{2/\rho }\leq c_{11}\omega \left( \left[ s,t\right] ^{2}\right)
^{3/\rho }.
\end{equation*}%
Clearly, $V_{\gamma }\left( R_{X-Y},\left[ s,t\right] ^{2}\right) \leq
\epsilon ^{2}\omega \left( \left[ s,t\right] ^{2}\right) ^{1/\gamma }$ and
hence%
\begin{equation*}
\left\vert \int_{s}^{t}\Psi _{3}\left( w\right) \,d\left( X^{k}-Y^{k}\right)
_{w}\right\vert _{L^{2}}^{2}\leq c_{12}\epsilon ^{2}\omega \left( \left[ s,t%
\right] ^{2}\right) ^{1/\gamma }\omega \left( \left[ s,t\right] ^{2}\right)
^{3/\rho }
\end{equation*}%
which gives the claim.
\end{proof}

\begin{remark}
Even though Proposition \ref{prop_main_estimates_n1_n2}, \ref%
{prop_main_estimates_n3} and \ref{prop_main_estimates_n4} are only
formulated for Gaussian processes with sample paths of finite variation, the
estimate $\left( \ref{eqn_key_estimate}\right) $ is valid also for general
Gaussian rough paths for $n=1,2,3,4$. Indeed, this follows from the fact
that Gaussian rough paths are just defined as $L^{2}$ limits of smooth
paths, cf. \cite{FV10AIHP}.
\end{remark}

\subsection{Higher levels\label{subsection_higher_levels}}

Once we have shown our desired estimates for the first four levels, we can
use induction to obtain also the higher levels. This is done in the next
proposition.

\begin{proposition}
\label{prop_key_estimate_higher_levels}Let $X$ and $Y$ be Gaussian processes
as in Theorem \ref{theorem_main01_intro}. Let $\rho $, $\gamma $ be fixed
and $\omega $ be a control. Assume that there are constants $\tilde{C}=%
\tilde{C}\left( n\right) $ such that%
\begin{equation*}
\left\vert \mathbf{X}_{s,t}^{n}\right\vert _{L^{2}},\left\vert \mathbf{Y}%
_{s,t}^{n}\right\vert _{L^{2}}\leq \tilde{C}\left( n\right) \frac{\omega
\left( s,t\right) ^{\frac{n}{2\rho }}}{\beta \left( \frac{n}{2\rho }\right) !%
}
\end{equation*}%
holds for $n=1,\ldots ,\left[ 2\rho \right] $ and constants $C=C\left(
n\right) $ such that%
\begin{equation*}
\left\vert \mathbf{X}_{s,t}^{n}-\mathbf{Y}_{s,t}^{n}\right\vert _{L^{2}}\leq
C\left( n\right) \epsilon \omega \left( s,t\right) ^{\frac{1}{2\gamma }}%
\frac{\omega \left( s,t\right) ^{\frac{n-1}{2\rho }}}{\beta \left( \frac{n-1%
}{2\rho }\right) !}
\end{equation*}%
holds for $n=1,\ldots ,\left[ 2\rho \right] +1$ and every~$\left( s,t\right)
\in \Delta $. Here, $\epsilon >0$ and $\beta $ is a positive constant such
that%
\begin{equation*}
\beta \geq 4\rho \left( 1+2^{\left( \left[ 2\rho \right] +1\right) /2\rho
}\left( \zeta \left( \frac{\left[ 2\rho \right] +1}{2\rho }\right) -1\right)
\right)
\end{equation*}
where$\ \zeta $ is just the usual Riemann zeta function. Then for every $%
n\in \mathbb{N}$ there is a constant $C=C\left( n\right) $ such that%
\begin{equation*}
\left\vert \mathbf{X}_{s,t}^{n}-\mathbf{Y}_{s,t}^{n}\right\vert _{L^{2}}\leq
C\epsilon \omega \left( s,t\right) ^{\frac{1}{2\gamma }}\frac{\omega \left(
s,t\right) ^{\frac{n-1}{2\rho }}}{\beta \left( \frac{n-1}{2\rho }\right) !}
\end{equation*}%
holds for every~$\left( s,t\right) \in \Delta $.
\end{proposition}

\begin{proof}
From Proposition \ref{prop_moments_lp} we know that for every $n\in \mathbb{N%
}$ there are constants $\tilde{C}\left( n\right) $ such that%
\begin{equation*}
\left\vert \mathbf{X}_{s,t}^{n}\right\vert _{L^{2}},\left\vert \mathbf{Y}%
_{s,t}^{n}\right\vert _{L^{2}}\leq \tilde{C}\frac{\omega \left( s,t\right) ^{%
\frac{n}{2\rho }}}{\beta \left( \frac{n}{2\rho }\right) !}
\end{equation*}%
holds for all $s<t$. We will proof the assertion by induction over $n$. The
induction basis is fulfiled by assumption. Suppose that the statement is
true for $k=1,\ldots ,n$ where $n\geq \left[ 2\rho \right] +1$. We will show
the statement for $n+1$. Let $D=\left\{ s=t_{0}<t_{1}<\ldots
<t_{j}=t\right\} $ be any partition of $\left[ s,t\right] $. Set%
\begin{eqnarray*}
\mathbf{\bar{X}}_{s,t} &:&=\left( 1,\mathbf{X}_{s,t}^{1},\ldots ,\mathbf{X}%
_{s,t}^{n},0\right) \in T^{n+1}\left( \mathbb{R}^{d}\right) , \\
\mathbf{\bar{X}}_{s,t}^{D} &:&=\mathbf{\bar{X}}_{s,t_{1}}\otimes \ldots
\otimes \mathbf{\bar{X}}_{t_{j-1},t}
\end{eqnarray*}%
and the same for $\mathbf{Y}$. We know that $\lim_{\left\vert D\right\vert
\rightarrow 0}\mathbf{\bar{X}}_{s,t}^{D}=S_{n+1}\left( \mathbf{X}\right)
_{s,t}$ a.s. and the same holds for $\mathbf{Y}$ (indeed, this is just the
definition of the Lyons lift, cf. \cite[Theorem 2.2.1]{L98}). By
multiplicativity, $\pi _{k}\left( \mathbf{\bar{X}}_{s,t}^{D}\right) =\mathbf{%
X}_{s,t}^{k}$ for $k\leq n$. We will show that for any dissection $D$ we have%
\begin{equation*}
\left\vert \pi _{n+1}\left( \mathbf{\bar{X}}_{s,t}^{D}-\mathbf{\bar{Y}}%
_{s,t}^{D}\right) \right\vert _{L^{2}}\leq C\left( n+1\right) \epsilon
\omega \left( s,t\right) ^{\frac{1}{2\gamma }}\frac{\omega \left( s,t\right)
^{\frac{n}{2\rho }}}{\beta \left( \frac{n}{2\rho }\right) !}.
\end{equation*}%
We use the notation $\left( \mathbf{X}^{D}\right) ^{k}:=\pi _{k}\left( 
\mathbf{\bar{X}}^{D}\right) $. Assume that $j\geq 2$. Let $D^{\prime }$ be
the partition of $\left[ s,t\right] $ obtained by removing a point $t_{i}$
of the dissection $D$ for which%
\begin{equation*}
\omega \left( t_{i-1},t_{i+1}\right) \leq \left\{ 
\begin{array}{ccc}
\frac{2\omega \left( s,t\right) }{j-1} & \text{for} & j\geq 3 \\ 
\omega \left( s,t\right) & \text{for} & j=2%
\end{array}%
\right.
\end{equation*}%
holds (Lemma 2.2.1 in \cite{L98} shows that there is indeed such a point).
By the triangle inequality,%
\begin{equation*}
\left\vert \left( \mathbf{X}^{D}-\mathbf{Y}^{D}\right) ^{n+1}\right\vert
_{L^{2}}\leq \left\vert \left( \mathbf{X}^{D}-\mathbf{X}^{D^{\prime
}}\right) ^{n+1}-\left( \mathbf{Y}^{D}-\mathbf{Y}^{D^{\prime }}\right)
^{n+1}\right\vert _{L^{2}}+\left\vert \left( \mathbf{X}^{D^{\prime }}-%
\mathbf{Y}^{D^{\prime }}\right) ^{n+1}\right\vert _{L^{2}}.
\end{equation*}%
We estimate the first term on the right hand side. As seen in the proof of 
\cite[Theorem 2.2.1]{L98}, $\left( \mathbf{X}_{s,t}^{D}-\mathbf{X}%
_{s,t}^{D^{\prime }}\right) ^{n+1}=\sum_{l=1}^{n}\mathbf{X}%
_{t_{i-1},t_{i}}^{l}\mathbf{X}_{t_{i},t_{i+1}}^{n+1-l}$. Set $\mathbf{R}^{l}=%
\mathbf{Y}^{l}\mathbf{-X}^{l}$. Then%
\begin{eqnarray*}
&&\left( \mathbf{X}_{s,t}^{D}-\mathbf{X}_{s,t}^{D^{\prime }}\right)
^{n+1}-\left( \mathbf{Y}_{s,t}^{D}-\mathbf{Y}_{s,t}^{D^{\prime }}\right)
^{n+1} \\
&=&\sum_{l=1}^{n}\mathbf{X}_{t_{i-1},t_{i}}^{l}\mathbf{X}%
_{t_{i},t_{i+1}}^{n+1-l}-\left( \mathbf{X}_{t_{i-1},t_{i}}^{l}+\mathbf{R}%
_{t_{i-1},t_{i}}^{l}\right) \left( \mathbf{X}_{t_{i},t_{i+1}}^{n+1-l}+%
\mathbf{R}_{t_{i},t_{i+1}}^{n+1-l}\right) \\
&=&\sum_{l=1}^{n}-\mathbf{X}_{t_{i-1},t_{i}}^{l}\mathbf{R}%
_{t_{i},t_{i+1}}^{n+1-l}-\mathbf{R}_{t_{i-1},t_{i}}^{l}\mathbf{Y}%
_{t_{i},t_{i+1}}^{n+1-l}.
\end{eqnarray*}%
By the triangle inequality, equivalence of $L^{q}$-norms in the Wiener
Chaos, our moment estimate for $\mathbf{X}^{k}$ and $\mathbf{Y}^{k}$ and the
induction hypothesis,%
\begin{eqnarray*}
&&\left\vert \left( \mathbf{X}_{s,t}^{D}-\mathbf{X}_{s,t}^{D^{\prime
}}\right) ^{n+1}-\left( \mathbf{Y}_{s,t}^{D}-\mathbf{Y}_{s,t}^{D^{\prime
}}\right) ^{n+1}\right\vert _{L^{2}} \\
&\leq &c_{1}\left( n+1\right) \sum_{l=1}^{n}\left\vert \mathbf{X}%
_{t_{i-1},t_{i}}^{l}\right\vert _{L^{2}}\left\vert \mathbf{R}%
_{t_{i},t_{i+1}}^{n+1-l}\right\vert _{L^{2}}+\left\vert \mathbf{R}%
_{t_{i-1},t_{i}}^{l}\right\vert _{L^{2}}\left\vert \mathbf{Y}%
_{t_{i},t_{i+1}}^{n+1-l}\right\vert _{L^{2}} \\
&\leq &c_{2}\left( n+1\right) \sum_{l=1}^{n}\epsilon \omega \left(
t_{i},t_{i+1}\right) ^{\frac{1}{2\gamma }}\frac{\omega \left(
t_{i-1},t_{i}\right) ^{\frac{l}{2\rho }}}{\beta \left( \frac{l}{2\rho }%
\right) !}\frac{\omega \left( t_{i},t_{i+1}\right) ^{\frac{n-l}{2\rho }}}{%
\beta \left( \frac{n-l}{2\rho }\right) !} \\
&&+\epsilon \omega \left( t_{i-1},t_{i}\right) ^{\frac{1}{2\gamma }}\frac{%
\omega \left( t_{i-1},t_{i}\right) ^{\frac{l-1}{2\rho }}}{\beta \left( \frac{%
l-1}{2\rho }\right) !}\frac{\omega \left( t_{i},t_{i+1}\right) ^{\frac{n+1-l%
}{2\rho }}}{\beta \left( \frac{n+1-l}{2\rho }\right) !} \\
&\leq &2c_{2}\epsilon \omega \left( s,t\right) ^{\frac{1}{2\gamma }%
}\sum_{l=0}^{n}\frac{\omega \left( t_{i-1},t_{i}\right) ^{\frac{l}{2\rho }}}{%
\beta \left( \frac{l}{2\rho }\right) !}\frac{\omega \left(
t_{i},t_{i+1}\right) ^{\frac{n-l}{2\rho }}}{\beta \left( \frac{n-l}{2\rho }%
\right) !} \\
&=&\frac{4\rho }{\beta ^{2}}c_{2}\epsilon \omega \left( s,t\right) ^{\frac{1%
}{2\gamma }}\frac{1}{2\rho }\sum_{l=0}^{n}\frac{\omega \left(
t_{i-1},t_{i}\right) ^{\frac{l}{2\rho }}}{\left( \frac{l}{2\rho }\right) !}%
\frac{\omega \left( t_{i},t_{i+1}\right) ^{\frac{n-l}{2\rho }}}{\left( \frac{%
n-l}{2\rho }\right) !} \\
&\leq &4\rho c_{2}\epsilon \omega \left( s,t\right) ^{\frac{1}{2\gamma }}%
\frac{\omega \left( t_{i-1},t_{i+1}\right) ^{\frac{n}{2\rho }}}{\beta
^{2}\left( \frac{n}{2\rho }\right) !}
\end{eqnarray*}%
where we used the neo-classical inequality (cf. \cite{HH10}) and
superadditivity of the control function. Hence for $j\geq 3$,%
\begin{eqnarray*}
\left\vert \left( \mathbf{X}_{s,t}^{D}-\mathbf{X}_{s,t}^{D^{\prime }}\right)
^{n+1}-\left( \mathbf{Y}_{s,t}^{D}-\mathbf{Y}_{s,t}^{D^{\prime }}\right)
^{n+1}\right\vert _{L^{2}} &\leq &4\rho c_{2}\epsilon \omega \left(
s,t\right) ^{\frac{1}{2\gamma }}\frac{\omega \left( t_{i-1},t_{i+1}\right) ^{%
\frac{n}{2\rho }}}{\beta ^{2}\left( \frac{n}{2\rho }\right) !} \\
&\leq &\left( \frac{2}{j-1}\right) ^{\frac{n}{2\rho }}4\rho c_{2}\epsilon
\omega \left( s,t\right) ^{\frac{1}{2\gamma }}\frac{\omega \left( s,t\right)
^{\frac{n}{2\rho }}}{\beta ^{2}\left( \frac{n}{2\rho }\right) !}.
\end{eqnarray*}%
For $j=2$ we get%
\begin{equation*}
\left\vert \left( \mathbf{X}_{s,t}^{D}-\mathbf{X}_{s,t}^{D^{\prime }}\right)
^{n+1}-\left( \mathbf{Y}_{s,t}^{D}-\mathbf{Y}_{s,t}^{D^{\prime }}\right)
^{n+1}\right\vert _{L^{2}}\leq 4\rho c_{2}\epsilon \omega \left( s,t\right)
^{\frac{1}{2\gamma }}\frac{\omega \left( s,t\right) ^{\frac{n}{2\rho }}}{%
\beta ^{2}\left( \frac{n}{2\rho }\right) !}
\end{equation*}%
but then $D^{\prime }=\left\{ s,t\right\} $ and therefore $\left\vert \left( 
\mathbf{X}_{s,t}^{D^{\prime }}-\mathbf{Y}_{s,t}^{D^{\prime }}\right)
^{n+1}\right\vert _{L^{2}}=0$. Hence by successively dropping points we see
that%
\begin{equation*}
\left\vert \left( \mathbf{X}_{s,t}^{D}-\mathbf{Y}_{s,t}^{D}\right)
^{n+1}\right\vert _{L^{2}}\leq \left( 1+\sum_{j=3}^{\infty }\left( \frac{2}{%
j-1}\right) ^{\frac{n}{2\rho }}\right) 4\rho c_{2}\epsilon \omega \left(
s,t\right) ^{\frac{1}{2\gamma }}\frac{\omega \left( s,t\right) ^{\frac{n}{%
2\rho }}}{\beta ^{2}\left( \frac{n}{2\rho }\right) !}
\end{equation*}%
holds for all partitions $D$. Since $n\geq \left[ 2\rho \right] +1$,%
\begin{equation*}
\sum_{j=3}^{\infty }\left( \frac{2}{j-1}\right) ^{\frac{n}{2\rho }}\leq
\sum_{j=3}^{\infty }\left( \frac{2}{j-1}\right) ^{\frac{\left[ 2\rho \right]
+1}{2\rho }}\leq 2^{\frac{\left[ 2\rho \right] +1}{2\rho }}\left( \zeta
\left( \frac{\left[ 2\rho \right] +1}{2\rho }\right) -1\right)
\end{equation*}%
and thus%
\begin{equation*}
\left\vert \left( \mathbf{X}_{s,t}^{D}-\mathbf{Y}_{s,t}^{D}\right)
^{n+1}\right\vert _{L^{2}}\leq \frac{4\rho \left( 1+2^{\frac{\left[ 2\rho %
\right] +1}{2\rho }}\left( \zeta \left( \frac{\left[ 2\rho \right] +1}{2\rho 
}\right) -1\right) \right) }{\beta }c_{2}\epsilon \omega \left( s,t\right) ^{%
\frac{1}{2\gamma }}\frac{\omega \left( s,t\right) ^{\frac{n}{2\rho }}}{\beta
\left( \frac{n}{2\rho }\right) !}.
\end{equation*}%
By the choice of $\beta $, we get the uniform bound%
\begin{equation*}
\left\vert \left( \mathbf{X}_{s,t}^{D}-\mathbf{Y}_{s,t}^{D}\right)
^{n+1}\right\vert _{L^{2}}\leq c_{2}\epsilon \omega \left( s,t\right) ^{%
\frac{1}{2\gamma }}\frac{\omega \left( s,t\right) ^{\frac{n}{2\rho }}}{\beta
\left( \frac{n}{2\rho }\right) !}
\end{equation*}%
which holds for all partitions $D$. Noting that a.s. convergence implies
convergence in $L^{2}$ in the Wiener chaos, we obtain our claim by sending $%
|D|\rightarrow 0$.
\end{proof}

\begin{corollary}
\label{cor_diff_estimates}Let $\left( X,Y\right) $, $\omega $, $\rho $ and $%
\gamma $ as in Lemma \ref{lemma_diff_alldifferent}. Then for all $n\in 
\mathbb{N}$ there are constants $C=C\left( \rho ,\gamma ,n\right) $ such that%
\begin{equation*}
\left\vert \mathbf{X}_{s,t}^{n}-\mathbf{Y}_{s,t}^{n}\right\vert _{L^{2}}\leq
C\epsilon \omega \left( \lbrack s,t]^{2}\right) ^{\frac{1}{2\gamma }}\omega
\left( \lbrack s,t]^{2}\right) ^{\frac{n-1}{2\rho }}
\end{equation*}%
holds for every $\left( s,t\right) \in \Delta $ where $\epsilon
^{2}=V_{\infty }\left( R_{X-Y},\left[ 0,1\right] ^{2}\right) ^{1-\rho
/\gamma }$.
\end{corollary}

\begin{proof}
For $n=1,2,3,4$ this is the content of Proposition \ref%
{prop_main_estimates_n1_n2}, \ref{prop_main_estimates_n3} and \ref%
{prop_main_estimates_n4}. By making the constants larger if necessary, we
also get%
\begin{equation*}
\left\vert \mathbf{X}_{s,t}^{n}-\mathbf{Y}_{s,t}^{n}\right\vert _{L^{2}}\leq
c\left( n\right) \epsilon \omega \left( \lbrack s,t]^{2}\right) ^{\frac{1}{%
2\gamma }}\frac{\omega \left( \lbrack s,t]^{2}\right) ^{\frac{n-1}{2\rho }}}{%
\beta \left( \frac{n-1}{2\rho }\right) !}
\end{equation*}%
with $\beta $ chosen as in Proposition \ref{prop_key_estimate_higher_levels}%
. We have already seen that%
\begin{equation*}
\left\vert \mathbf{X}_{s,t}^{n}\right\vert _{L^{2}},\left\vert \mathbf{Y}%
_{s,t}^{n}\right\vert _{L^{2}}\leq \tilde{c}\left( n\right) \frac{\omega
\left( \lbrack s,t]^{2}\right) ^{\frac{n}{2\rho }}}{\beta \left( \frac{n}{%
2\rho }\right) !}
\end{equation*}%
holds for constants $\tilde{c}\left( n\right) $ where $n=1,2,3$. Since $\rho
<2$, we have $[2\rho ]+1\leq 4$. From Proposition \ref%
{prop_key_estimate_higher_levels} we can conclude that%
\begin{equation*}
\left\vert \mathbf{X}_{s,t}^{n}-\mathbf{Y}_{s,t}^{n}\right\vert _{L^{2}}\leq
c\left( n\right) \epsilon \omega \left( \lbrack s,t]^{2}\right) ^{\frac{1}{%
2\gamma }}\frac{\omega \left( \lbrack s,t]^{2}\right) ^{\frac{n-1}{2\rho }}}{%
\beta \left( \frac{n-1}{2\rho }\right) !}
\end{equation*}%
holds for every $n\in \mathbb{N}$ and constants $c\left( n\right) $. Setting 
$C\left( n\right) =\frac{c\left( n\right) }{\beta \left( \frac{n-1}{2\rho }%
\right) !}$ gives our claim.
\end{proof}

\section{Main result\label{section_main_result}}

Assume that $X$ is a Gaussian process as in Theorem \ref%
{theorem_main01_intro} with paths of finite $p$-variation. Consider a
sequence $\left( \Lambda _{k}\right) _{k\in \mathbb{N}}$ of continuous
operators%
\begin{equation*}
\Lambda _{k}\colon C^{p-var}\left( \left[ 0,1\right] ,\mathbb{R}\right)
\rightarrow C^{1-var}\left( \left[ 0,1\right] ,\mathbb{R}\right) .
\end{equation*}%
If $x=\left( x^{1},\ldots ,x^{d}\right) \in C^{p-var}\left( \left[ 0,1\right]
,\mathbb{R}^{d}\right) $, we will write $\Lambda _{k}\left( x\right) =\left(
\Lambda _{k}\left( x^{1}\right) ,\ldots ,\Lambda _{k}\left( x^{d}\right)
\right) $. Assume that $\Lambda _{k}$ fulfils the following conditions:

\begin{enumerate}
\item $\Lambda _{k}\left( x\right) \rightarrow x$ in the $\left\vert \cdot
\right\vert _{\infty }$-norm if $k\rightarrow \infty $ for every $x\in
C^{p-var}\left( \left[ 0,1\right] ,\mathbb{R}^{d}\right) .$

\item If $R_{X}$ has finite controlled $\rho $-variation, then, for some $%
C=C\left( \rho \right) $,%
\begin{equation*}
\sup_{k,l\in \mathbb{N}}\left\vert R_{\left( \Lambda _{k}\left( X\right)
,\Lambda _{l}\left( X\right) \right) }\right\vert _{\rho -var;\left[ 0,1%
\right] ^{2}}\leq C\left\vert R_{X}\right\vert _{\rho -var;\left[ 0,1\right]
^{2}}.
\end{equation*}
\end{enumerate}

Our main result is the following:

\begin{theorem}
\label{theorem_main01}Let $X$ be a Gaussian process as in Theorem \ref%
{theorem_main01_intro} for $\rho <2$ and $K\geq V_{\rho }\left( R_{X},\left[
0,1\right] ^{2}\right) $. Then there is an enhanced Gaussian process $%
\mathbf{X}$ with sample paths in $C^{0,p-var}\left( \left[ 0,1\right] ,G^{%
\left[ p\right] }\left( \mathbb{R}^{d}\right) \right) $ w.r.t. $\left(
\Lambda _{k}\right) _{k\in \mathbb{N}}$ where $p\in \left( 2\rho ,4\right) $%
, i.e.%
\begin{equation*}
\left\vert \rho _{p-var}\left( S_{\left[ p\right] }\left( \Lambda _{k}\left(
X\right) \right) ,\mathbf{X}\right) \right\vert _{L^{r}}\rightarrow 0
\end{equation*}%
for $k\rightarrow \infty $ and every $r\geq 1$. Moreover, choose $\gamma $
such that $\gamma >\rho $ and $\frac{1}{\gamma }+\frac{1}{\rho }>1$. Then
for $q>2\gamma $ and every $N\in \mathbb{N}$ there is a constant $C=C\left(
q,\rho ,\gamma ,K,N\right) $ such that%
\begin{equation*}
\left\vert \rho _{q-var}\left( S_{N}\left( \Lambda _{k}\left( X\right)
\right) ,S_{N}\left( \mathbf{X}\right) \right) \right\vert _{L^{r}}\leq
Cr^{N/2}\sup_{0\leq t\leq 1}\left\vert \Lambda _{k}\left( X\right)
_{t}-X_{t}\right\vert _{L^{2}\left( \mathbb{R}^{d}\right) }^{1-\frac{\rho }{%
\gamma }}
\end{equation*}%
holds for every $k\in \mathbb{N}$.
\end{theorem}

\begin{proof}
The first statement is a fundamental result about Gaussian rough paths, see 
\cite[Theorem 15.33]{FV10}. For the second, take $\delta >0$ and set%
\begin{equation*}
\gamma ^{\prime }=\left( 1+\delta \right) \gamma \quad \text{and\quad }\rho
^{\prime }=\left( 1+\delta \right) \rho .
\end{equation*}%
By choosing $\delta $ smaller if necessary we can assume that $\frac{1}{\rho
^{\prime }}+\frac{1}{\gamma ^{\prime }}>1$ and $q>2\gamma ^{\prime }$. Set%
\begin{equation*}
\omega _{k,l}\left( A\right) =\left\vert R_{\left( \Lambda _{k}\left(
X\right) ,\Lambda _{l}\left( X\right) \right) }\right\vert _{\rho ^{\prime
}-var;A}^{\rho ^{\prime }}
\end{equation*}%
for a rectangle $A\subset \left[ 0,1\right] ^{2}$ and 
\begin{equation*}
\epsilon _{k,l}=V_{\infty }\left( R_{\left( \Lambda _{k}\left( X\right)
-\Lambda _{l}\left( X\right) \right) },\left[ 0,1\right] ^{2}\right) ^{\frac{%
1}{2}-\frac{\rho ^{\prime }}{2\gamma ^{\prime }}}=V_{\infty }\left(
R_{\left( \Lambda _{k}\left( X\right) -\Lambda _{l}\left( X\right) \right) },%
\left[ 0,1\right] ^{2}\right) ^{\frac{1}{2}-\frac{\rho }{2\gamma }}.
\end{equation*}%
From Theorem \ref{theorem_comp_contr_p_var} we know that $\omega _{k,l}$ is
a $2D$ control function which controls the $\rho ^{\prime }$-variation of $%
R_{\left( \Lambda _{k}\left( X\right) ,\Lambda _{l}\left( X\right) \right) }$%
. From Corollary \ref{cor_diff_estimates} we can conclude that there is a
constant $c_{1}$ such that%
\begin{equation*}
\left\vert \pi _{n}\left( S_{N}\left( \Lambda _{k}\left( X\right) \right)
_{s,t}-S_{N}\left( \Lambda _{l}\left( X\right) \right) _{s,t}\right)
\right\vert _{L^{2}}\leq c_{1}\epsilon _{k,l}\omega _{k,l}\left( \left[ s,t%
\right] ^{2}\right) ^{\frac{1}{2\gamma ^{\prime }}}\omega _{k,l}\left( \left[
s,t\right] ^{2}\right) ^{\frac{n-1}{2\rho ^{\prime }}}
\end{equation*}%
holds for every $n=1,\ldots ,N$, $\left( s,t\right) \in \Delta $ and $k,l\in 
\mathbb{N}$. Now,%
\begin{eqnarray*}
\omega _{k,l}\left( \left[ s,t\right] ^{2}\right) ^{\frac{n-1}{2\rho
^{\prime }}} &=&\left( \frac{\omega _{k,l}\left( \left[ s,t\right]
^{2}\right) }{\omega _{k,l}\left( \left[ 0,1\right] ^{2}\right) }\right) ^{%
\frac{n-1}{2\rho ^{\prime }}}\omega _{k,l}\left( \left[ 0,1\right]
^{2}\right) ^{\frac{n-1}{2\rho ^{\prime }}} \\
&\leq &\omega _{k,l}\left( \left[ s,t\right] ^{2}\right) ^{\frac{n-1}{%
2\gamma ^{\prime }}}\omega _{k,l}\left( \left[ 0,1\right] ^{2}\right) ^{%
\frac{n-1}{2\rho ^{\prime }}-\frac{n-1}{2\gamma ^{\prime }}}.
\end{eqnarray*}%
From Theorem \ref{theorem_comp_contr_p_var} and our assumptions on the $%
\Lambda _{k}$ we know that 
\begin{equation*}
\omega _{k,l}\left( \left[ 0,1\right] ^{2}\right) ^{1/\rho ^{\prime }}\leq
c_{2}\left\vert R_{X}\right\vert _{\rho ^{\prime }-var;\left[ 0,1\right]
^{2}}\leq c_{3}V_{\rho }\left( R_{X},\left[ 0,1\right] ^{2}\right) \leq
c_{4}\left( \rho ,\rho ^{\prime },K\right) .
\end{equation*}%
holds uniformly over all $k,l$. Hence%
\begin{equation*}
\left\vert \pi _{n}\left( S_{N}\left( \Lambda _{k}\left( X\right) \right)
_{s,t}-S_{N}\left( \Lambda _{l}\left( X\right) \right) _{s,t}\right)
\right\vert _{L^{2}}\leq c_{5}\epsilon _{k,l}\omega _{k,l}\left( \left[ s,t%
\right] ^{2}\right) ^{\frac{n}{2\gamma ^{\prime }}}.
\end{equation*}%
Proposition \ref{prop_moments_lp} shows with the same argument that%
\begin{equation*}
\left\vert \pi _{n}\left( S_{N}\left( \Lambda _{k}\left( X\right) \right)
_{s,t}\right) \right\vert _{L^{2}}\leq c_{6}\omega _{k,l}\left( \left[ s,t%
\right] ^{2}\right) ^{\frac{n}{2\rho ^{\prime }}}\leq c_{7}\omega
_{k,l}\left( \left[ s,t\right] ^{2}\right) ^{\frac{n}{2\gamma ^{\prime }}}
\end{equation*}%
for every $k\in \mathbb{N}$ and the same holds for $S_{N}\left( \Lambda
_{l}\left( X\right) \right) _{s,t}$. From \cite[Proposition 15.24]{FV10} we
can conclude that there is a constant $c_{8}$ such that%
\begin{equation*}
\left\vert \rho _{q-var}\left( S_{N}\left( \Lambda _{k}\left( X\right)
\right) ,S_{N}\left( \Lambda _{l}\left( X\right) \right) \right) \right\vert
_{L^{r}}\leq c_{8}r^{N/2}\epsilon _{k,l}
\end{equation*}%
holds for all $k,l\in \mathbb{N}$. In particular, we have shown that $\left(
S_{N}\left( \Lambda _{k}\left( X\right) \right) \right) _{k\in \mathbb{N}}$
is a Cauchy sequence in $L^{r\text{ }}$and it is clear that the limit is
given by the Lyons lift $S_{N}\left( \mathbf{X}\right) $ of the enhanced
Gaussian process $\mathbf{X}$. Now fix $k\in \mathbb{N}$. For every $l\in 
\mathbb{N}$,%
\begin{eqnarray*}
\left\vert \rho _{q-var}\left( S_{N}\left( \Lambda _{k}\left( X\right)
\right) ,S_{N}\left( \mathbf{X}\right) \right) \right\vert _{L^{r}} &\leq
&\left\vert \rho _{q-var}\left( S_{N}\left( \Lambda _{k}\left( X\right)
\right) ,S_{N}\left( \Lambda _{l}\left( X\right) \right) \right) \right\vert
_{L^{r}} \\
&&+\left\vert \rho _{q-var}\left( S_{N}\left( \Lambda _{l}\left( X\right)
\right) ,S_{N}\left( \mathbf{X}\right) \right) \right\vert _{L^{r}} \\
&\leq &c_{8}r^{N/2}\epsilon _{k,l}+\left\vert \rho _{q-var}\left(
S_{N}\left( \Lambda _{l}\left( X\right) \right) ,S_{N}\left( \mathbf{X}%
\right) \right) \right\vert _{L^{r}}.
\end{eqnarray*}%
It is easy to see that%
\begin{equation*}
\epsilon _{k,l}\rightarrow V_{\infty }\left( R_{\left( \Lambda _{k}\left(
X\right) -X\right) },\left[ 0,1\right] ^{2}\right) ^{\frac{1}{2}-\frac{\rho 
}{2\gamma }}\quad \text{for }l\rightarrow \infty
\end{equation*}%
and since%
\begin{equation*}
\left\vert \rho _{q-var}\left( S_{N}\left( \Lambda _{l}\left( X\right)
\right) ,S_{N}\left( \mathbf{X}\right) \right) \right\vert
_{L^{r}}\rightarrow 0\quad \text{for }l\rightarrow \infty
\end{equation*}%
we can conclude that%
\begin{equation*}
\left\vert \rho _{q-var}\left( S_{N}\left( \Lambda _{k}\left( X\right)
\right) ,S_{N}\left( \mathbf{X}\right) \right) \right\vert _{L^{r}}\leq
c_{8}r^{N/2}V_{\infty }\left( R_{\left( \Lambda _{k}\left( X\right)
-X\right) },\left[ 0,1\right] ^{2}\right) ^{\frac{1}{2}-\frac{\rho }{2\gamma 
}}
\end{equation*}%
holds for every $k\in \mathbb{N}$. Finally, we have for $\left[ \sigma ,\tau %
\right] \times \left[ \sigma ^{\prime },\tau ^{\prime }\right] \subset \left[
0,1\right] ^{2}$%
\begin{equation*}
\left\vert R_{\left( \Lambda _{k}\left( X\right) -X\right) }\left( 
\begin{array}{c}
\sigma ,\tau \\ 
\sigma ^{\prime },\tau ^{\prime }%
\end{array}%
\right) \right\vert _{\mathbb{R}^{d\times d}}\leq 4\sup_{0\leq s<t\leq
1}\left\vert R_{\left( \Lambda _{k}\left( X\right) -X\right) }\left(
s,t\right) \right\vert _{\mathbb{R}^{d\times d}}
\end{equation*}%
and hence%
\begin{equation*}
V_{\infty }\left( R_{\left( \Lambda _{k}\left( X\right) -X\right) },\left[
0,1\right] ^{2}\right) \leq 4\sup_{0\leq s<t\leq 1}\left\vert R_{\left(
\Lambda _{k}\left( X\right) -X\right) }\left( s,t\right) \right\vert _{%
\mathbb{R}^{d\times d}}.
\end{equation*}%
Furthermore, for any $s<t$,%
\begin{equation*}
\left\vert R_{\left( \Lambda _{k}\left( X\right) -X\right) }\left(
s,t\right) \right\vert _{\mathbb{R}^{d\times d}}\leq \left\vert \Lambda
_{k}\left( X\right) _{s}-X_{s}\right\vert _{L^{2}\left( \mathbb{R}%
^{d}\right) }\left\vert \Lambda _{k}\left( X\right) _{t}-X_{t}\right\vert
_{L^{2}\left( \mathbb{R}^{d}\right) }\leq \sup_{0\leq t\leq 1}\left\vert
\Lambda _{k}\left( X\right) _{t}-X_{t}\right\vert _{L^{2}\left( \mathbb{R}%
^{d}\right) }^{2}
\end{equation*}%
and therefore%
\begin{equation*}
V_{\infty }\left( R_{\left( \Lambda _{k}\left( X\right) -X\right) },\left[
0,1\right] ^{2}\right) ^{\frac{1}{2}-\frac{\rho }{2\gamma }}\leq
c_{9}\sup_{0\leq t\leq 1}\left\vert \Lambda _{k}\left( X\right)
_{t}-X_{t}\right\vert _{L^{2}\left( \mathbb{R}^{d}\right) }^{1-\frac{\rho }{%
\gamma }}
\end{equation*}%
which shows the result.
\end{proof}

The next Theorem gives pathwise convergence rates for the Wong-Zakai error
for suitable approximations of the driving signal.

\begin{theorem}
\label{theorem_as_wong_zakai_rate}Let $X$ be as in Theorem \ref%
{theorem_main01_intro} for $\rho <2$, $K\geq V_{\rho }\left( R_{X},\left[ 0,1%
\right] ^{2}\right) $ and $X^{\left( k\right) }=\Lambda _{k}\left( X\right) $%
. Consider the SDEs%
\begin{eqnarray}
dY_{t} &=&V(Y_{t})\,d\mathbf{X}_{t},\quad Y_{0}\in \mathbb{R}^{n}
\label{eqn_RDE_GP_general} \\
dY_{t}^{\left( k\right) } &=&V(Y_{t}^{\left( k\right) })\,dX_{t}^{\left(
k\right) },\quad Y_{0}^{\left( k\right) }=Y_{0}\in \mathbb{R}^{n}
\label{eqn_RS_general}
\end{eqnarray}%
where $\left\vert V\right\vert _{Lip^{\theta }}\leq \nu <\infty $ for a $%
\theta >2\rho $. Assume that there is a constant $C_{1}$ and a sequence $%
\left( \epsilon _{k}\right) _{k\in \mathbb{N}}\subset \dbigcup\limits_{r\geq
1}l^{r}$ such that%
\begin{equation*}
\sup_{0\leq t\leq 1}\left\vert X_{t}^{\left( k\right) }-X_{t}\right\vert
_{L^{2}}^{2}\leq C_{1}\epsilon _{k}^{1/\rho }~\text{for all }k\in \mathbb{N}.
\end{equation*}%
Choose $\eta ,q$ such that%
\begin{equation*}
0\leq \eta <\min \left\{ \frac{1}{\rho }-\frac{1}{2},\frac{1}{2\rho }-\frac{1%
}{\theta }\right\} \quad \text{and\quad }q\in \left( \frac{2\rho }{1-2\rho
\eta },\theta \right) .
\end{equation*}%
Then both SDEs $\left( \ref{eqn_RDE_GP_general}\right) $ and $\left( \ref%
{eqn_RS_general}\right) $ have unique solutions $Y$ and $Y^{\left( k\right)
} $ and there is a finite random variable $C$ and a null set $M$ such that%
\begin{equation}
\left\vert Y^{\left( k\right) }\left( \omega \right) -Y\left( \omega \right)
\right\vert _{\infty ;\left[ 0,1\right] }\leq \left\vert Y^{\left( k\right)
}\left( \omega \right) -Y\left( \omega \right) \right\vert _{q-var;\left[ 0,1%
\right] }\leq C\left( \omega \right) \epsilon _{k}^{\eta }
\label{eqn_pathwise_wong_zakai}
\end{equation}%
holds for all $k\in \mathbb{N}$ and $\omega \in \Omega \setminus M$. The
random variable $C$ depends on $\rho ,q,\eta ,\nu ,\theta ,K,C_{1}$, the
sequence $\left( \epsilon _{k}\right) _{k\in \mathbb{N}}$ and the driving
process $X$ but not on the equation itself. The same holds for the set $M$.
\end{theorem}

\begin{remark}
Note that this means that we have \emph{universal rates}, i.e. the set $M$
and the random variable $C$ are valid for all starting points (and also
vector fields subject to a uniform $Lip^{\theta }$-bound). In particular,
our convergence rates apply to solutions viewed as $C^{l}$-diffeomorphisms
where $l=\left[ \theta -q\right] $, cf. \cite[Theorem 11.12]{FV10} and \cite%
{FR11}.
\end{remark}

\begin{proof}[Proof of Theorem \protect\ref{theorem_as_wong_zakai_rate}]
Note that $\gamma >\rho $ and $\frac{1}{\rho }+\frac{1}{\gamma }>1$ is
equivalent to $0<\frac{1}{2\rho }-\frac{1}{2\gamma }<\frac{1}{\rho }-\frac{1%
}{2}$. Hence there is a $\gamma _{0}>\rho $ such that $\eta =\frac{1}{2\rho }%
-\frac{1}{2\gamma _{0}}$ and $\frac{1}{\rho }+\frac{1}{\gamma _{0}}>1$.
Furthermore, $2\gamma _{0}=\frac{2\rho }{1-2\rho \eta }<q$. Choose $\gamma
_{1}>\gamma _{0}$ such that still $2\gamma _{1}<q$ and $\eta <\frac{1}{2\rho 
}-\frac{1}{2\gamma _{1}}<$ $\frac{1}{\rho }-\frac{1}{2}$ , hence $\frac{1}{%
\rho }+\frac{1}{\gamma _{1}}>1$ hold. Set $\alpha :=$ $\frac{1}{2\rho }-%
\frac{1}{2\gamma _{1}}-\eta >0$. From Theorem \ref{theorem_main01} we know
that for every $r\geq 1$ and $N\in \mathbb{N}$ there is a constant $c_{1}$
such that 
\begin{equation*}
\left\vert \rho _{q-var}\left( S_{N}(X^{\left( k\right) }),S_{N}\left( 
\mathbf{X}\right) \right) \right\vert _{L^{r}}\leq c_{1}r^{N/2}\sup_{0\leq
t\leq 1}\left\vert X_{t}^{\left( k\right) }-X_{t}\right\vert _{L^{2}}^{1-%
\frac{\rho }{\gamma }}\leq c_{2}r^{N/2}\epsilon _{k}^{\frac{1}{2\rho }-\frac{%
1}{2\gamma }}
\end{equation*}%
holds for every $k\in \mathbb{N}$. Hence%
\begin{equation*}
\left\vert \frac{\rho _{q-var}\left( S_{N}(X^{\left( k\right) }),S_{N}\left( 
\mathbf{X}\right) \right) }{\epsilon _{k}^{\eta }}\right\vert _{L^{r}}\leq
c_{2}r^{N/2}\epsilon _{k}^{\alpha }
\end{equation*}%
for every $k\in \mathbb{N}$. From the Markov inequality, for any $\delta >0$,

\begin{equation*}
\sum_{k=1}^{\infty }P\left[ \frac{\rho _{q-var}\left( S_{N}(X^{\left(
k\right) }),S_{N}\left( \mathbf{X}\right) \right) }{\epsilon _{k}^{\eta }}%
\geq \delta \right] \leq \frac{1}{\delta ^{r}}\sum_{k=1}^{\infty }\left\vert 
\frac{\rho _{q-var}\left( S_{N}(X^{\left( k\right) }),S_{N}\left( \mathbf{X}%
\right) \right) }{\epsilon _{k}^{\eta }}\right\vert _{L^{r}}^{r}\leq
c_{3}\sum_{k=1}^{\infty }\epsilon _{k}^{\alpha r}
\end{equation*}%
By assumption, we can choose $r$ large enough such that the series
converges. With Borel-Cantelli we can conclude that%
\begin{equation*}
\frac{\rho _{q-var}\left( S_{N}(X^{\left( k\right) }),S_{N}\left( \mathbf{X}%
\right) \right) }{\epsilon _{k}^{\eta }}\rightarrow 0
\end{equation*}%
outside a null set $M$ for $k\rightarrow \infty $. We set%
\begin{equation*}
C_{2}:=\sup_{k\in \mathbb{N}}\frac{\rho _{q-var}\left( S_{N}(X^{\left(
k\right) }),S_{N}\left( \mathbf{X}\right) \right) }{\epsilon _{k}^{\eta }}%
<\infty \quad \text{a.s.}
\end{equation*}%
Since $C_{2}$ is the supremum of $\mathcal{F}$-measurable random variables
it is itself $\mathcal{F}$-measurable. Now set $N=\left[ q\right] $ which
turns $\rho _{q-var}$ into a rough path metric. Note that since $\theta
>2\rho $, $\left( \ref{eqn_RDE_GP_general}\right) $ and $\left( \ref%
{eqn_RS_general}\right) $ have indeed unique solutions $Y$ and $Y^{\left(
k\right) }$. We substitute the driver $\mathbf{X}$ by $S_{N}(\mathbf{X})$
resp. $X^{\left( k\right) }$ by $S_{N}(X^{\left( k\right) })$ in the above
equations, now considered as RDEs in the $q$-rough paths space. Since $%
\theta >q$, both (RDE-) equations have again unique solutions and it is
clear that they coincide with $Y$ and $Y^{\left( k\right) }$. From%
\begin{equation*}
\rho _{q-var}\left( S_{N}(X^{\left( k\right) }),\mathbf{1}\right) \leq \rho
_{q-var}\left( S_{N}(X^{\left( k\right) }),S_{N}\left( \mathbf{X}\right)
\right) +\rho _{q-var}\left( S_{N}\left( \mathbf{X}\right) ,\mathbf{1}%
\right) \leq C_{1}+\rho _{q-var}\left( S_{N}\left( \mathbf{X}\right) ,%
\mathbf{1}\right)
\end{equation*}%
we see that for every $\omega \in \Omega \setminus M$ the $S_{N}(X^{\left(
k\right) }\left( \omega \right) )$ are uniformly bounded for all $k$ in the
topology given by the metric $\rho _{q-var}$. Thus we can apply local
Lipschitz-continuity of the It\={o}-Lyons map (see \cite[Theorem 10.26]{FV10}%
) to see that there is a random variable $C_{3}$ such that 
\begin{equation*}
\left\vert Y^{\left( k\right) }-Y\right\vert _{q-var;\left[ 0,1\right] }\leq
C_{3}\rho _{q-var}\left( S_{N}(X^{\left( k\right) }),S_{N}\left( \mathbf{X}%
\right) \right) \leq C_{3}\cdot C_{2}\epsilon _{k}^{\eta }
\end{equation*}%
holds for every $k\in \mathbb{N}$ outside $M$. Finally,%
\begin{equation*}
\left\vert Y_{t}^{\left( k\right) }-Y_{t}\right\vert =\left\vert
Y_{0,t}^{\left( k\right) }-Y_{0,t}\right\vert \leq \left\vert Y^{\left(
k\right) }-Y\right\vert _{q-var;\left[ 0,t\right] }\leq \left\vert Y^{\left(
k\right) }-Y\right\vert _{q-var;\left[ 0,1\right] }
\end{equation*}%
is true for all $t\in \left[ 0,1\right] $ and the claim follows.
\end{proof}

\subsection{Mollifier approximations}

Let $\phi $ be a mollifier function with support $\left[ -1,1\right] $, i.e. 
$\phi \in C_{0}^{\infty }\left( \left[ -1,1\right] \right) $ is positive and 
$\left\vert \phi \right\vert _{L^{1}}=1$. If $x\colon \left[ 0,1\right]
\rightarrow \mathbb{R}$ is a continuous path, we denote by $\bar{x}\colon 
\mathbb{R}\rightarrow \mathbb{R}$ its continuous extension to the whole real
line, i.e. 
\begin{equation*}
\bar{x}_{u}=\left\{ 
\begin{array}{ccc}
x_{0} & \text{for} & x\in (-\infty ,0] \\ 
x_{u} & \text{for} & x\in \left[ 0,1\right] \\ 
x_{1} & \text{for} & x\in \lbrack 1,\infty )%
\end{array}%
\right.
\end{equation*}%
For $\epsilon >0$ set 
\begin{eqnarray*}
\phi _{\epsilon }\left( u\right) &:&=\frac{1}{\epsilon }\phi \left(
u/\epsilon \right) \quad \text{and} \\
x_{t}^{\epsilon } &:&=\int_{\mathbb{R}}\phi _{\epsilon }\left( t-u\right) 
\bar{x}_{u}\,du.
\end{eqnarray*}%
Let $\left( \epsilon _{k}\right) _{k\in \mathbb{N}}$ be a sequence of real
numbers such that $\epsilon _{k}\rightarrow 0$ for $k\rightarrow \infty $.
Define%
\begin{equation*}
\Lambda _{k}\left( x\right) :=x^{\epsilon _{k}}.
\end{equation*}%
In \cite{FV10}, Chapter 15.2.3 it is shown that the sequence $\left( \Lambda
_{k}\right) _{k\in \mathbb{N}}$ fulfils the conditions of Theorem \ref%
{theorem_main01}.

\begin{corollary}
Let $X$ be as in Theorem \ref{theorem_main01_intro} and assume that there is
a constant $C$ such that $V_{\rho }\left( R_{X};\left[ s,t\right]
^{2}\right) \leq C\left\vert t-s\right\vert ^{1/\rho }$ holds for all $s<t$.
Choose $\left( \epsilon _{k}\right) _{k\in \mathbb{N}}\in $ $%
\dbigcup\limits_{r\geq 1}l^{r}$ and set $X^{\left( k\right) }=X^{\epsilon
_{k}}$. Then the solutions $Y^{\left( k\right) }$ of the SDE $\left( \ref%
{eqn_RS_general}\right) $ converge pathwise to the solution $Y$ of $\left( %
\ref{eqn_RDE_GP_general}\right) $ in the sense of $\left( \ref%
{eqn_pathwise_wong_zakai}\right) $ with rate $O\left( \epsilon _{k}^{\eta
}\right) $ where $\eta $ is chosen as in Theorem \ref%
{theorem_as_wong_zakai_rate}.
\end{corollary}

\begin{proof}
It suffices to note that for every $\epsilon >0$, $Z\in \left\{ X^{1},\ldots
,X^{d}\right\} $ and $t\in \left[ 0,1\right] $ we have%
\begin{eqnarray*}
E\left[ \left\vert Z_{t}^{\epsilon }-Z_{t}\right\vert ^{2}\right] &=&E\left[
\left( \int_{\mathbb{R}}\phi _{\epsilon }\left( t-u\right) \left( \bar{Z}%
_{u}-Z_{t}\right) \,du\right) ^{2}\right] \\
&=&E\left[ \left( \int_{\left[ t-\epsilon ,t+\epsilon \right] }\phi
_{\epsilon }\left( t-u\right) \left( \bar{Z}_{u}-Z_{t}\right) \,du\right)
^{2}\right] \\
&=&E\left[ \int_{\left[ t-\epsilon ,t+\epsilon \right] ^{2}}\phi _{\epsilon
}\left( t-u\right) \phi _{\epsilon }\left( t-v\right) \left( \bar{Z}%
_{u}-Z_{t}\right) \left( \bar{Z}_{v}-Z_{t}\right) \,du\,dv\right] \\
&=&\int_{\left[ t-\epsilon ,t+\epsilon \right] ^{2}}\phi _{\epsilon }\left(
t-u\right) \phi _{\epsilon }\left( t-v\right) E\left[ \left( \bar{Z}%
_{u}-Z_{t}\right) \left( \bar{Z}_{v}-Z_{t}\right) \right] \,du\,dv \\
&\leq &\sup_{\substack{ t\in \left[ 0,1\right]  \\ \left\vert
h_{1}\right\vert ,\left\vert h_{2}\right\vert \leq \epsilon }}\left\vert E%
\left[ \left( \bar{Z}_{t+h_{1}}-Z_{t}\right) \left( \bar{Z}%
_{t+h_{2}}-Z_{t}\right) \right] \right\vert \\
&\leq &\sup_{\substack{ t\in \left[ 0,1\right]  \\ \left\vert h\right\vert
\leq \epsilon }}E\left[ \left( \bar{Z}_{t+h}-Z_{t}\right) ^{2}\right] \leq
c_{1}\epsilon ^{1/\rho }
\end{eqnarray*}%
from which follows that $\sup_{0\leq t\leq 1}\left\vert X_{t}^{\epsilon
_{k}}-X_{t}\right\vert _{L^{2}}^{2}\leq c_{1}\epsilon _{k}^{1/\rho }$. We
conclude with Theorem \ref{theorem_as_wong_zakai_rate}.
\end{proof}

\subsection{Piecewise linear approximations}

If $D=\{0=t_{0}<t_{1}<\ldots <t_{\#D-1}=1\}$ is a partition of $[0,1]$ and $%
x\colon \left[ 0,1\right] \rightarrow \mathbb{R}$ a continuous path, we
denote by $x^{D}$ the piecewise linear approximation of $x$ at the points of 
$D$, i.e. $x^{D}$ coincides with $x$ at the points $t_{i}$ and if $t_{i}\leq
t<t_{i+1}$ we have%
\begin{equation*}
\frac{x_{t_{i+1}}^{D}-x_{t}^{D}}{t_{i+1}-t}=\frac{x_{t_{i+1}}-x_{t_{i}}}{%
t_{i+1}-t_{i}}.
\end{equation*}%
Let $\left( D_{k}\right) _{k\in \mathbb{N}}$ be a sequence of partitions of $%
\left[ 0,1\right] $ such that $\left\vert D_{k}\right\vert :=\max_{t_{i}\in
D_{k}}\left\{ \left\vert t_{i+1}-t_{i}\right\vert \right\} \rightarrow 0$
for $k\rightarrow \infty $. If $x\colon \left[ 0,1\right] \rightarrow 
\mathbb{R}$ is continuous, we define%
\begin{equation*}
\Lambda _{k}\left( x\right) :=x^{D_{k}}.
\end{equation*}%
In \cite[Chapter 15.2.3]{FV10} it is shown that $\left( \Lambda _{k}\right)
_{k\in \mathbb{N}}$ fulfils the conditions of Theorem \ref{theorem_main01}.
If $R_{X}$ is the covariance of a Gaussian process, we set%
\begin{equation*}
\left\vert D\right\vert _{R_{X},\rho }=\left( \max_{t_{i}\in D}V_{\rho
}\left( R_{X};\left[ t_{i},t_{i+1}\right] ^{2}\right) \right) ^{\rho }.
\end{equation*}

\begin{corollary}
\label{cor_wong_zakai_piecew_lin}Let $X$ be as in Theorem \ref%
{theorem_main01_intro}. Choose a sequence of partitions $\left( D_{k}\right)
_{k\in \mathbb{N}}$ of the interval $\left[ 0,1\right] $ such that $\left(
\left\vert D_{k}\right\vert _{R_{X},\rho }\right) _{k\in \mathbb{N}}\in $ $%
\dbigcup\limits_{r\geq 1}l^{r}$ and set $X^{\left( k\right) }=X^{D_{k}}$.
Then the solutions $Y^{\left( k\right) }$ of the SDE $\left( \ref%
{eqn_RS_general}\right) $ converge pathwise to the solution $Y$ of $\left( %
\ref{eqn_RDE_GP_general}\right) $ in the sense of $\left( \ref%
{eqn_pathwise_wong_zakai}\right) $ with rate $O\left( \epsilon _{k}^{\eta
}\right) $ where $\left( \epsilon _{k}\right) _{k\in \mathbb{N}}=\left(
\left\vert D_{k}\right\vert _{R_{X},\rho }\right) _{k\in \mathbb{N}}$ and $%
\eta $ is chosen as in Theorem \ref{theorem_as_wong_zakai_rate}.
\end{corollary}

\begin{proof}
Let $D$ be any partition of $\left[ 0,1\right] $ and $t\in \left[
t_{i},t_{i+1}\right] $ where $t_{i},t_{i+1}\in D$. Take $Z\in \left\{
X^{1},\ldots ,X^{d}\right\} $. Then%
\begin{equation*}
Z_{t}^{D}-Z_{t}=Z_{t_{i},t_{i+1}}\frac{t-t_{i}}{t_{i+1}-t_{i}}-Z_{t_{i},t}.
\end{equation*}%
Therefore%
\begin{equation*}
\left\vert Z_{t}^{D}-Z_{t}\right\vert _{L^{2}}\leq \left\vert
Z_{t_{i},t_{i+1}}\right\vert _{L^{2}}+\left\vert Z_{t_{i},t}\right\vert
_{L^{2}}\leq 2V_{\rho }\left( R_{X};\left[ t_{i},t_{i+1}\right] ^{2}\right)
^{1/2}\leq 2\left\vert D\right\vert _{R_{X},\rho }^{\frac{1}{2\rho }}.
\end{equation*}%
We conclude with Theorem \ref{theorem_as_wong_zakai_rate}.
\end{proof}

\begin{example}
Let $X=B^{H}$ be the fractional Brownian motion with Hurst parameter $H\in
(1/4,1/2]$. Set $\rho =\frac{1}{2H}<2$. Then one can show that $R_{X}$ has
finite $\rho $-variation and $V_{\rho }\left( R_{X};\left[ s,t\right]
^{2}\right) \leq c\left( H\right) \left\vert t-s\right\vert ^{1/\rho }$ for
all~$\left( s,t\right) \in \Delta $ (see \cite{FV11}, Example 1). Assume
that the vector fields in $\left( \ref{eqn_RDE_GP_general}\right) $ are
sufficiently smooth by which we mean that $1/\rho -1/2\leq 1/\left( 2\rho
\right) -1/\theta $, i.e. 
\begin{equation*}
\theta \geq \frac{2\rho }{\rho -1}=\frac{1}{1/2-H}.
\end{equation*}%
Let $\left( D_{k}\right) _{k\in \mathbb{N}}$ be the sequence of uniform
partitions. By Corollary \ref{cor_wong_zakai_piecew_lin}, for every $\eta
<2H-1/2$ there is a random variable $C$ such that%
\begin{equation*}
\left\vert Y^{\left( k\right) }-Y\right\vert _{\infty }\leq C\left( \frac{1}{%
k}\right) ^{\eta }\quad \text{a.s.}
\end{equation*}%
hence we have a Wong-Zakai convergence rate arbitrary close to $2H-1/2$. In
particular, for the Brownian motion, we obtain a rate close to $1/2$, see
also \cite{GS06} and \cite{FR11}. For $H$ $\rightarrow 1/4$, the convergence
rate tends to $0$ which reflects the fact that the L\'{e}vy area indeed
diverges for $H=1/4$, see \cite{CQ02}.
\end{example}

\subsection{The simplified step-$N$ Euler scheme\label%
{subsection_simple_euler}}

Consider again the SDE%
\begin{equation*}
dY_{t}=V(Y_{t})\,dX_{t},\quad Y_{0}\in \mathbb{R}^{n}
\end{equation*}%
interpreted as a pathwise RDE driven by the lift $\mathbf{X}$ of a Gaussian
process $X$ which fulfils the conditions of Theorem \ref%
{theorem_main01_intro}. Let $D$ be a partition of $\left[ 0,1\right] $. We
recall the simplified step-$N$ Euler scheme from the introduction:%
\begin{eqnarray*}
Y_{0}^{\text{sEuler}^{N};D} &=&Y_{0} \\
Y_{t_{j+1}}^{\text{sEuler}^{N};D} &=&Y_{t_{j}}^{\text{sEuler}%
^{N};D}+V_{i}\left( Y_{t_{j}}^{\text{sEuler}^{N};D}\right)
X_{t_{j},t_{j+1}}^{i}+\frac{1}{2}\mathcal{V}_{i_{1}}V_{i_{2}}\left(
Y_{t_{j}}^{\text{sEuler}^{N};D}\right)
X_{t_{j},t_{j+1}}^{i_{1}}X_{t_{j},t_{j+1}}^{i_{2}} \\
&&+\ldots +\frac{1}{N!}\mathcal{V}_{i_{1}}\mathcal{\ldots V}%
_{i_{N-1}}V_{i_{N}}\left( Y_{t_{j}}^{\text{sEuler}^{N};D}\right)
X_{t_{j},t_{j+1}}^{i_{1}}\ldots X_{t_{j},t_{j+1}}^{i_{N}}
\end{eqnarray*}%
where $t_{j}\in D$. In this section, we will investigate the convergence
rate of this scheme. For simplicity, we will assume that%
\begin{equation*}
V_{\rho }\left( R_{X};\left[ s,t\right] ^{2}\right) =O\left( \left\vert
t-s\right\vert ^{1/\rho }\right)
\end{equation*}%
which can always be achieved at the price of a deterministic time-change
based on 
\begin{equation*}
\left[ 0,1\right] \ni t\mapsto \frac{V_{\rho }\left( R_{X};\left[ 0,t\right]
^{2}\right) ^{\rho }}{V_{\rho }\left( R_{X};\left[ 0,1\right] ^{2}\right)
^{\rho }}\in \left[ 0,1\right] .
\end{equation*}%
Set $D_{k}=\left\{ \frac{i}{k}:i=0,\ldots ,k\right\} $.

\begin{corollary}
\label{Cor_rate_simple_euler}Let $p>2\rho $ and assume that $\left\vert
V\right\vert _{Lip^{\theta }}<\infty $ for $\theta >p$. Choose $\eta $ and $%
N $ such that%
\begin{equation*}
\eta <\min \left\{ \frac{1}{\rho }-\frac{1}{2},\frac{1}{2\rho }-\frac{1}{%
\theta }\right\} \quad \text{and\quad }N\leq \left[ \theta \right] .
\end{equation*}%
Then there are random variables $C_{1}$ and $C_{2}$ such that%
\begin{equation*}
\max_{t_{j}\in D_{k}}\left\vert Y_{t_{j}}-Y_{t_{j}}^{\text{sEuler}%
^{N};D_{k}}\right\vert \leq C_{1}\left( \frac{1}{k}\right) ^{\eta
}+C_{2}\left( \frac{1}{k}\right) ^{\frac{N+1}{p}-1}\quad \text{a.s. for all }%
k\in \mathbb{N}\text{.}
\end{equation*}
\end{corollary}

\begin{proof}
Recall the step-$N$ Euler scheme from the introduction (or cf. \cite[Chapter
10]{FV10}). Set $X^{\left( k\right) }=X^{D_{k}}$ and let $Y^{\left( k\right)
}$ be the solution of the SDE $\left( \ref{eqn_RS_general}\right) $. Then $%
Y_{t_{j}}^{\text{sEuler}^{N};D_{k}}=\left( Y^{\left( k\right) }\right)
_{t_{j}}^{\text{Euler}^{N};D_{k}}$ for every $t_{j}\in D_{k}$ and therefore,
using the triangle inequality,%
\begin{equation*}
\max_{t_{j}\in D_{k}}\left\vert Y_{t_{j}}-Y_{t_{j}}^{\text{sEuler}%
^{N};D_{k}}\right\vert \leq \sup_{t\in \left[ 0,1\right] }\left\vert
Y_{t}-Y_{t}^{\left( k\right) }\right\vert +\max_{t_{j}\in D_{k}}\left\vert
Y_{t_{j}}^{\left( k\right) }-\left( Y^{\left( k\right) }\right) _{t_{j}}^{%
\text{Euler}^{N};D_{k}}\right\vert .
\end{equation*}%
By the choice of $D_{k}$ we have $\left\vert D_{k}\right\vert _{R_{X},\rho
}=O\left( k^{-1}\right) $. Applying Corollary \ref{cor_wong_zakai_piecew_lin}
we obtain for the first term $\left\vert Y-Y^{\left( k\right) }\right\vert
_{\infty }=O\left( k^{-\eta }\right) $. Refering to \cite[Theorem 10.30]%
{FV10} we see that the second term is of order $O\left( k^{-\left( \frac{N+1%
}{p}-1\right) }\right) $.
\end{proof}

\begin{remark}
Assume that the vector fields are sufficiently smooth, i.e. $\theta \geq 
\frac{2\rho }{\rho -1}$. Then we obtain an error of $O\left(
k^{-(2/p-1/2)}\right) +O\left( k^{-\left( \frac{N+1}{p}-1\right) }\right) $,
any $p>2\rho $. That means that in the case $\rho =1$, the step-$2$ scheme
(i.e. the simplified Milstein scheme) gives an optimal convergence rate of
(almost) $1/2$. For $\rho \in (1,2)$, the step-$3$ scheme gives an optimal
rate of (almost) $1/\rho -1/2$. In particular, we see that using higher
order schemes does not improve the convergence rate since in that case, the
Wong-Zakai error persists. In the fractional Brownian motion case, the
simplified Milstein scheme gives an optimal convergence rate of (almost) $%
1/2 $ for the Brownian motion and for $H\in (1/4,1/2)$ the step-$3$ scheme
gives an optimal rate of (almost) $2H-1/2$. This answers a conjecture stated
in \cite{DT}.
\end{remark}

\end{document}